\documentclass{article}
\usepackage{graphicx}
\usepackage{tikz-cd}
\usepackage{tikz}
\usepackage{amsmath}
\usepackage{amssymb}
\usepackage[left=2.7cm,right=2.7cm,top=3cm,bottom=2.8cm]{geometry}
\usepackage{setspace}
\usepackage[OT2,T1]{fontenc}
\usepackage{leftindex,tensor,mhchem}
\usepackage{mathtools}
\usepackage{mathrsfs}
\usepackage{indentfirst}
\usepackage{xypic}
\usepackage{amsthm} 
\usepackage[scr=pxtx]{mathalpha}

\usepackage{hyperref}
\hypersetup{colorlinks=true,linkcolor=blue,anchorcolor=blue,citecolor=blue}
\usepackage{cleveref}

\usepackage{etoolbox}
\usepackage{bbding} % 星积记号宏包
\usepackage{pifont} % diagonal complex 记号宏包
\usepackage{color}
\usepackage{relsize}
\usepackage[bbgreekl]{mathbbol}
\usepackage{amsfonts}
\usepackage{quiver}
\usepackage[normalem]{ulem}

\input pdfmsym
\pdfmsymsetscalefactor{10}

\DeclareSymbolFontAlphabet{\mathbb}{AMSb}
\DeclareSymbolFontAlphabet{\mathbbl}{bbold}

\newcommand{\rmL}{\mathrm{L}}
\newcommand{\stimes}{{\scriptstyle\text{\ding{84}}}} % star times

\newcommand{\cdstimes}{{\scriptstyle\widehat{\text{\ding{84}}^{{\scriptscriptstyle\rmL}}}}}%completed derived star times

\newcommand{\Z}{\mathbb{Z}}
\newcommand{\N}{\mathbb{N}}

\newcommand{\G}{\mathbb{G}}
\newcommand{\cris}{\mathrm{crys}}

\newcommand{\HH}{\mathrm{H}} % cohomology H is mathrm H
\newcommand{\Hom}{\mathrm{Hom}} 
\newcommand{\RHom}{\mathrm{RHom}}
\newcommand{\id}{\mathrm{id}} 
\newcommand{\Fil}{\mathrm{Fil}} % Filtration
\newcommand{\D}{\mathbb{D}}

\newcommand{\Tot}{\mathrm{Tot}} % Total complex of a double complex
\newcommand{\dt}{\widetilde{\tau}} % diagonal truncation 
\newcommand{\dtimes}{\otimes^{\rmL}} % derived tensor product
\newcommand{\rmI}{\mathrm{I}}
\newcommand{\DGr}{\mathcal{DG}r}
\newcommand{\calD}{\mathcal{D}}
\newcommand{\Ext}{\mathrm{Ext}}
\newcommand{\Spec}{\mathrm{Spec}}
\newcommand{\Ker}{\mathrm{Ker}}
\newcommand{\RR}{\mathscr{R}}

\newcommand{\td}{\text{-}}
\newcommand{\coeur}{\text{c\oe ur}}
\newcommand{\Tor}{\mathrm{Tor}}

\numberwithin{equation}{section}

\theoremstyle{plain}
\newtheorem{theorem}[equation]{Theorem}

\newtheorem{proposition}[equation]{Proposition}
\newtheorem{lemma}[equation]{Lemma}

\newtheorem{corollary}[equation]{Corollary}

\newtheorem{construction/proposition}[equation]{Construction/Proposition}

\theoremstyle{definition}
\newtheorem{definition}[equation]{Definition}

\newtheorem{notation}[equation]{Notation}

\newtheorem{construction}[equation]{Construction}

\newtheorem{example}[equation]{Example}
\newtheorem{digression}[equation]{Digression}

\newtheorem{situation}[equation]{Situation}
\newtheorem{remark}[equation]{Remark}

\newtheorem{warning}[equation]{Warning}

\AddToHook{env/theorem/begin}{\crefalias{equation}{theorem}}
\AddToHook{env/thm/begin}{\crefalias{equation}{theorem}}
\AddToHook{env/conjecture/begin}{\crefalias{equation}{conjecture}}
\AddToHook{env/proposition/begin}{\crefalias{equation}{proposition}}
\AddToHook{env/lemma/begin}{\crefalias{equation}{lemma}}
\AddToHook{env/claim/begin}{\crefalias{equation}{claim}}
\AddToHook{env/corollary/begin}{\crefalias{equation}{corollary}}
\AddToHook{env/cor/begin}{\crefalias{equation}{corollary}}
\AddToHook{env/construction/proposition/begin}{\crefalias{equation}{construction/proposition}}

\AddToHook{env/definition/begin}{\crefalias{equation}{definition}} 
\AddToHook{env/defn/begin}{\crefalias{equation}{definition}}
\AddToHook{env/notation/begin}{\crefalias{equation}{notation}}
\AddToHook{env/convention/begin}{\crefalias{equation}{convention}}
\AddToHook{env/construction/begin}{\crefalias{equation}{construction}}
\AddToHook{env/variant/begin}{\crefalias{equation}{variant}}
\AddToHook{env/question/begin}{\crefalias{equation}{question}}
\AddToHook{env/hypothesis/begin}{\crefalias{equation}{hypothesis}}
\AddToHook{env/example/begin}{\crefalias{equation}{example}}
\AddToHook{env/digression/begin}{\crefalias{equation}{digression}}
\AddToHook{env/examples/begin}{\crefalias{equation}{example}} % Aliasing plural to singular name
\AddToHook{env/situation/begin}{\crefalias{equation}{situation}}
\AddToHook{env/remark/begin}{\crefalias{equation}{remark}}
\AddToHook{env/rmk/begin}{\crefalias{equation}{remark}}
\AddToHook{env/exercise/begin}{\crefalias{equation}{exercise}}
\AddToHook{env/warning/begin}{\crefalias{equation}{warning}}
\AddToHook{env/setup/begin}{\crefalias{equation}{setup}}

\newcommand{\simp}[3]{\xymatrix@1{#1  & {\ }#2 \ar@<.4ex>[l] \ar@<-.4ex>[l] & {\ } #3 \ar@<0.8ex>[l] \ar[l] \ar@<-.8ex>[l]& \cdots \ar@<1.2ex>[l] \ar@<.4ex>[l] \ar@<-.4ex>[l] \ar@<-1.2ex>[l]  }}
\newcommand{\cosimp}[3]{\xymatrix@1{#1 \ar@<.4ex>[r] \ar@<-.4ex>[r] & {\ }#2 \ar@<0.8ex>[r] \ar[r] \ar@<-.8ex>[r] & {\ } #3 \ar@<1.2ex>[r] \ar@<.4ex>[r] \ar@<-.4ex>[r] \ar@<-1.2ex>[r] & \cdots }} % simplicial and cosimplicial 

\usepackage{pict2e}

\makeatletter
\newcommand{\dc}{\mathord{\mathpalette\rDelta@\relax}}% diagonal complex 记号
\newcommand{\rDelta@}[2]{%
  \begingroup
  \edef\rDelta@thick{%
    1.5\fontdimen8
      \ifx\displaystyle#1\textfont\else
      \ifx\textstyle#1\textfont\else
      \ifx\scriptstyle#1\scriptfont\else
      \scriptscriptfont\fi\fi\fi 3
    }%
  \edef\rDelta@slant{%
      \ifx\displaystyle#10.71\else
      \ifx\textstyle#10.71\else
      \ifx\scriptstyle#10.71\else
      0.69\fi\fi\fi
  }%
  \sbox\z@{$\m@th#1\Delta$}%
  \begin{picture}(\wd\z@,\ht\z@)
    \linethickness{\rDelta@thick}
    \put(0,0){\usebox{\z@}}
    \Line(\rDelta@slant\wd\z@,0.2)(0.4\wd\z@,0.8\ht\z@)
  \end{picture}%
  \endgroup
}

\title{On de Rham--Witt Cohomology of Classifying Stacks}
\author{Shizhang Li, Yuan Yang}
\date{\today}

\begin{document}

\maketitle

\begin{abstract}
We give an example of proper smooth fourfold over a perfect field \(k\) 
of characteristic \(p > 0\) with asymmetric Hodge--Witt numbers in total degree $3$. 
Our example is sharp both in terms of dimension and total degree.
We arrive at our example by computing and approximating the Hodge--Witt cohomology groups 
of the classifying stack \(B\alpha_p\). 
\end{abstract}

\tableofcontents

\section{Introduction}

de Rham--Witt complex of a smooth variety $X$ over a perfect field $k$ of positive characteristic
was introduced by Illusie \cite{IlldRW}, and further studied by Illusie--Raynaud \cite{IRdRW},
Ekedahl \cite{Ek1, Ek2} and others in the 1980s.
It plays a central role in the study of the arithmetic of such varieties.
%For a proper smooth variety \(X\) over a perfect field \(k\) of characteristic \(p>0\), the de Rham--Witt complex \(W\Omega^\bullet_{X/k}\) gives rise to cohomology groups
%\(\HH^j(X,W\Omega^i_{X/k}). \)
%In particular, the complex \(R\Gamma(X,W\Omega^\bullet_{X/k})\) may be viewed as an object of Illusie’s derived category \(\calD_c^b(R)\), and provides a bridge between Hodge-theoretic and \(p\)-adic invariants of algebraic varieties. 

When $X$ is further assumed to be proper,
associated with the cohomology of its de Rham--Witt complex is a collection of numerical invariants:
Namely the \emph{Hodge--Witt numbers} \(h_W^{i,j}(X)\) introduced by Ekedahl \cite[Definition IV.3.1]{Ek3}. 
These numbers arise from the homological algebra of the de Rham--Witt complex 
and are defined in terms of slope multiplicities and so-called ``\emph{domino numbers}''. 
Although their definition is somewhat indirect, 
Crew’s formula \cite[Theorem 4]{Cr} and Ekedahl’s inequality theorem \cite[Theorem IV.3.3]{Ek3} show 
that they are closely related to the classical Hodge numbers \(h^{i,j}(X)\). 
From this perspective, the Hodge--Witt numbers may be regarded as 
characteristic \(p\) analogues of Hodge numbers.
The search for such an analogue is reasonable, as it is well-known that Hodge numbers in
positive characteristics behave very poorly.

It is natural to ask to what extent the familiar symmetries of Hodge numbers
persist for Hodge--Witt numbers.
Ekedahl proved that they satisfy the analogue of Serre duality \cite[Proposition VI.3.2]{Ek3}:
\[
h_W^{i,j}=h_W^{N-i,N-j},
\]
whenever $X$ is proper and smooth of equidimension $N$.
This symmetry reflects the duality theory of Hodge--Witt cohomology. 
Another fundamental symmetry of Hodge numbers in characteristic \(0\) is \emph{Hodge symmetry},
\[
h^{i,j}=h^{j,i}.
\]
While knowing that in positive characteristic Hodge numbers can be asymmetric, 
it is natural to ask whether the same symmetry holds for Hodge--Witt numbers:
\[
h_W^{i,j}=h_W^{j,i}.
\]

In low degree or dimension, such symmetry holds true:
For instance, Ekedahl proved that Hodge symmetry holds for Hodge--Witt numbers 
whenever \(i+j\leqslant 2\) \cite[Corollary VI.3.3(ii)]{Ek3}, and more generally for all 
smooth proper varieties of dimension \(\leqslant 3\) \cite[Corollary VI.3.3(iii)]{Ek3}. 
Beyond these cases, however, no general principle is known 
that would force ``Hodge--Witt symmetry'' in higher dimensions. 
To this question, Ekedahl himself commented \cite[Remark in Page.~113]{Ek3}
that ``he saw no reason why such a symmetry should hold in general''. 
Nevertheless, no counterexample seems to have appeared in the literature. 
The purpose of this paper is to show that such counterexamples exist starting exactly in dimension $4$
and cohomological degree $3$:
\begin{theorem}[\Cref{the counterexample}]
\label{Main Thm:counterexample}
There exists a smooth proper $4$-fold over any perfect
field of characteristic $p > 0$, with $h_W^{0,3} = 1$, $h_W^{1,2} = -2$,
$h_W^{2,1} = 1$, and $h_W^{3,0} = 0$.
\end{theorem}

We note that, given Ekedahl's result, our theorem is sharp both in dimension and in total degree.
Our approach is inspired by the work of Antieau--Bhatt--Mathew \cite{ABM}. 

Below let us describe the contents of each section.
In \Cref{dRW of smooth stacks}, we define de Rham--Witt cohomology of smooth geometric Artin
stacks and establish several of its various basic properties.
A key result (\Cref{Hodge-proper dRW is coherent})
shows that if the stack is Hodge-proper in the sense of Kubrak--Prikhodko \cite[Definition 0.2.1]{KP22},
then its de Rham--Witt cohomology groups are coherent in the sense of Illusie--Raynaud \cite[D\'efinition I.3.8]{IRdRW}.
%The key idea is to study Hodge--Witt cohomology for geometric Artin stacks. 
%Working in a derived \(\infty\)-categorical framework, we define the de Rham--Witt complex \(W\Omega^\bullet_{\mathfrak X/k}\) for a geometric Artin stack \(\mathfrak X\) over \(k\), together with its Hodge--Witt cohomology. 
%Under mild hypotheses we prove that these cohomology groups are coherent in Illusie’s sense.

In \Cref{dRW of Balphap}, we analyze the stack \(B\alpha_p\) and compute its low-degree Hodge--Witt cohomologies. 
These computations show that the Hodge--Witt numbers of \(B\alpha_p\) fail to satisfy 
the analogue of Hodge symmetry for \(i+j=3\). 
Following the strategy of \cite{ABM}, we then approximate the stack \(B\alpha_p\times B\G_m\)
by smooth proper varieties, thereby producing the counterexample stated in \Cref{Main Thm:counterexample}.

The computations for the stack $B\alpha_p$ do not come from nowhere; rather, 
they are built upon fundamental work of Illusie, Illusie--Raynaud, and Ekedahl.
In \Cref{R-complexes}, we review the theory of graded left \(\RR\)-modules developed by Illusie--Raynaud
and Ekedahl, 
together with their contributions to de Rham--Witt cohomology groups. 
The \Cref{deeper structure} reviews and extends Ekedahl's results
on the deep structures of the derived $\infty$-category \(\DGr(\mathscr{R})\) of graded left $\mathscr{R}$-complexes. 
Along the way, we reinterpret parts of the theory using the modern language of \(\infty\)-categories:
We try our best to extend Ekedahl's construction from the level of homotopy categories to the level of $\infty$-categories.
%In particular, we introduce the derived tensor product functor \(R_n\otimes_R^{\rmL}(-)\), Ekedahl’s duality functor \(D_R(-)\), and the star-product functor \(\stimes\), and we develop the diagonal \(t\)-structure on the unbounded derived \(\infty\)-category \(\DGr_c(R)\). 
%This framework leads to two structural results which appear not to have been previously observed.

%\begin{enumerate}
%\item An object \(M\in\calD(R)\) is coherent if it is bounded below and its Hodge cohomology groups are finite-dimensional.
%\item Ekedahl’s \(\cdstimes\)-product can be computed using free \(\widehat{R}\)-resolutions.
%\end{enumerate}

%The first result may be viewed as a dual form of Ekedahl’s coherence theorem and allows us to establish the coherence of Hodge--Witt cohomology for geometric Artin stacks. 
%The second provides a practical method for computing certain classical star products appearing in \cite{Ek2}, and recovers Ekedahl’s original calculations.

\subsection*{Notation and conventions}
Throughout the article, we shall denote by $k$ a perfect field of characteristic $p > 0$.
We use $W$ to denote the ring of Witt vectors of $k$.
We shall use $\sigma$ to denote the Frobenius operator on $W$: for any element $a \in W$,
we write both $a^{\sigma}$ and $\sigma(a)$ for its image under Frobenius.

We write $\RR$ for the \emph{Cartier--Dieudonn\'e--Raynaud} graded ring,
see \Cref{IRDring}. This is an ordinary associative graded $\mathbb{Z}_p$-algebra.
We denote by $\DGr(\RR)$ the $\infty$-category $\mathrm{LMod}_{\RR}(\DGr(\mathbb{Z}_p))$
of left $\RR$-module objects in the graded derived $\infty$-category of $\mathbb{Z}_p$.
Whenever we refer to a graded left $\RR$-module or left graded $R$-module, we mean an object
in the heart of the standard $t$-structure on $\DGr(\RR)$.

Let \(M \in \DGr(\RR)\). Then each cohomology group \(\HH^j(M)\) is a graded \(\RR\)-module,
and each graded component \(M^i\) can be viewed as an object in $\calD(W)$.
Following Ekedahl's convention, we say that \(M\) is
\emph{bounded below} (resp.~\emph{bounded above})
if \(\HH^j(M) = 0\) for \(j \ll 0\) (resp.~\(j \gg 0\)).
Similarly, we say $M$ is \emph{grading-left bounded} (resp.~\emph{grading-right bounded})
if \(M^i = 0\) for \(i \ll 0\) (resp.~\(i \gg 0\)).
The grading $i$ piece of $\HH^j(M)$ is denoted by $\HH^j(M)^i$.

Following {\cite[Section~0, p.~188]{Ek1}}, we denote cohomological and grading
shifts of $M$ as follows.
For integers \((i,j)\), we write \(M(i)[j]\) for the object obtained from \(M\)
by shifting \(i\) degrees in the \(\RR\)-grading (horizontally)
and \(j\) degrees in the cohomological grading (vertically).
As a result, there are canonical isomorphisms $\HH^n(M(i)[j])^m \cong \HH^{n + j}(M)^{(m + i)}$.
%More precisely, if $M = (M^{k, \ell}, d_v, d_h)$ is represented by a double complex,
%then the $M(i)[j]$ is represented by the following double complex:
%\[
%M(i)[j]^{k,l} = M^{k+i,l+j},
%\]
%with differentials given by
%\[
%d_v = (-1)^j d_v, \qquad d_h = (-1)^i d_h .
%\]
In contrast, Illusie--Raynaud \cite{IRdRW} denotes the \(\RR\)-module grading shifting by $M[i]$,
which we do not follow because it conflicts with the notation of the cohomological shift.
%\textcolor{teal}{There is a problem in this convention: \(\Tot(M(i))=\Tot(M)[i]\) but \(\Tot(M[i])\neq \Tot(M)[i]\). The problem is: all horizontal maps are changed by \((-1)^{i}\) but all vertical maps are the same. There is only a quasi-isomorphism in the second case. }

Occasionally we shall encounter other ordinary graded rings, such as $W[d]/d^2$ or $k[d]/d^2$,
we shall use similar notation such as $\DGr(W[d]/d^2)$ and $\DGr(k[d]/d^2)$ and so on.

\subsection*{Acknowledgement}
The influence of the work by Illusie, Raynaud, and Ekedahl on this article
should be obvious to readers.
S.L.~is deeply grateful to Luc Illusie, who suggested some 4 years ago that S.L.~revisit Ekedahl’s work 
and see what more can be said through a modern lens, 
a proposal that provided an initial motivation for S.L.~to join Y.Y.~in this project.
We are very grateful to Longke Tang for patiently answering many basic questions when we are trying
to put some of Ekedahl's work in the framework of $\infty$-categories
and for numerous discussions and suggestions.
%We have also benefited from communications regarding this work with the following mathematicians:
%\textcolor{red}{TBA}.
Y.Y.~would like to express his gratitude to BICMR and Morningside Center for their support during the writing of this paper. 
S.L.~is supported by the National Key R $\&$ D Program of China No.~2023YFA1009701 
and the National Natural Science Foundation of China (No.~12288201).
Y.Y.~is partially supported by a grant from National Science Foundation of China NSFC-12231001.

\section{Left graded \(\RR\)-modules}
\label{R-complexes}
\addtocontents{toc}{\protect\setcounter{tocdepth}{1}}
It was discovered by Illusie that the crystalline complex admits a canonical representative,
namely the de Rham--Witt complex. 
Its existence is rather striking: it provides a concrete complex equipped with Frobenius \(F\) and Verschiebung \(V\) satisfying the fundamental relation \(FdV=d\), thereby encoding the structure of crystalline cohomology.
Illusie also observed the appearance of certain torsion phenomena in the slope spectral sequence in the study of supersingular abelian and K3 surfaces (\cite[Example II.7.1, II.7.2]{IlldRW}).
%These structures were later organized into the elementary objects now known as dominoes \(\rmU_t\), 
%which form the basic building blocks of the theory. 
It soon became clear that the failure of degeneration for the slope spectral sequence is largely governed by 
certain objects now known as dominoes (see \Cref{Ut}). 
Illusie and Raynaud formalized this observation by introducing the notion of coherent graded \(\RR\)-modules. 
The homological algebra of such \(\RR\)-modules reveals a remarkably rich structure and gives rise to a
number of subtle numerical invariants associated with de Rham--Witt cohomology of smooth proper varieties over $k$.

\subsection{The ring $\RR$ and totalization} 
In the 1970s, Illusie introduced and studied in \cite{IlldRW} the de Rham--Witt complex \(W\Omega_{-/k}^\bullet\)
as a contravariant functor defined on all smooth schemes over $k$
(see also \cite{BLM}). 
This is a complex of \'etale sheaves equipped with two operators:
Frobenius \(F\) and Verschiebung \(V\), satisfying the relations
\[
FV=VF=p, \qquad Fa=a^{\sigma}F, \qquad Va=a^{\sigma^{-1}}V, \qquad FdV=d .
\]
To analyze the structure of the cohomology of these de Rham--Witt complexes, Illusie and Raynaud \cite{IRdRW}
introduced a certain graded ring \(\RR\):

\begin{definition}
\label{IRDring}
Let \(k\) be a perfect field of characteristic $p > 0$. The \emph{Cartier--Dieudonn\'e--Raynaud ring} over \(k\) is the
\(\mathbb{Z}\)-graded ring $\RR = \RR^0 \oplus \RR^1$,
where \(\RR^0 \cong W_{\sigma}[F,V]\) is the (noncommutative) Dieudonn\'e ring generated by
Frobenius \(F\) and Verschiebung \(V\), subject to the relations
\[
Fa = a^{\sigma}F, \qquad Va = a^{\sigma^{-1}}V, \qquad FV = VF = p.
\]
The grading \(1\) component \(\RR^1\) is the two-sided \(\RR^0\)-module generated by an element \(d\),
satisfying
\[
d^2 = 0, \qquad FdV = d.
\]
\end{definition}

Concretely, the graded components of \(\RR\) can be described as:
\begin{equation}
\begin{aligned}
\RR^0 &\cong 
\Biggl\{
\sum_{n>0} a_{-n} F^n + a_0 + \sum_{n>0} a_n V^n
\;\Bigg|\;
a_n \in W \text{ and is } 0 \text{ for all but finitely many } n
\Biggr\},\\
\RR^1 &\cong 
\Biggl\{
\sum_{n>0} a_{-n} F^n d + a_0 d + \sum_{n>0} a_n d V^n
\;\Bigg|\;
a_n \in W \text{ and is } 0 \text{ for all but finitely many } n
\Biggr\}.
\end{aligned}
\end{equation}

Given a complex
\[
\cdots \longrightarrow M^i \longrightarrow M^{i+1} \longrightarrow \cdots
\]
in which each \(M^i\) is a left \(\RR^0\)-module and the compatibility relation \(FdV=d\) holds,
the direct sum totalization \(\bigoplus_{i\in\mathbb Z} M^i\) naturally acquires the structure of a left graded \(\RR\)-module.
Conversely, any left graded \(\RR\)-module determines such a compatible complex.
In particular, the de Rham--Witt complex \(W\Omega^\bullet_{X/k}\) can be regarded as a sheaf of left graded \(\RR\)-modules.

Following \cite{IRdRW}, let \(\DGr(\RR)\) denote the \(\infty\)-category of left \(\RR\)-module objects in the graded derived \(\infty\)-category \(\DGr(\mathbb{Z}_p)\).
The forgetful functor induced by the graded ring map \(W[d]/d^2 \to \RR\) gives rise to a functor
\(
\DGr(\RR) \longrightarrow \DGr(W[d]/d^2).
\)
The \(\infty\)-category \(\DGr(W[d]/d^2)\) admits a more familiar description, which we briefly recall below.

\begin{digression}[{\cite{Ari21, Ant24}}]
\label{Ariotta identification}
Recall Joyal's pointed \(1\)-category \(\Xi\), whose set of objects is \(\mathbb{Z} \cup \{*\}\), where \(*\) is both initial and final, and the morphism sets between integers are given by
\[
\mathrm{Hom}_{\Xi}(m,n) =
\begin{cases}
\{*\} & \text{if } n \neq m,\, m-1,\\
\{*, \mathrm{id}\} & \text{if } n = m,\\
\{*, \partial\} & \text{if } n = m-1.
\end{cases}
\]

Let \(\mathcal{C}\) be a small stable \(\infty\)-category admitting sequential limits. Denote by \(\mathrm{Ch}^{\bullet}(\mathcal{C})\) the \(\infty\)-category of pointed functors from
\(\Xi^{\mathrm{op}}\) to \(\mathcal{C}\).
A result of Ariotta \cite[Theorem 4.7]{Ari21} (see also \cite[Theorem 3.21]{Ant24}) establishes an equivalence of \(\infty\)-categories
\(
\mathrm{Ch}^{\bullet}(\mathcal{C}) \simeq \widehat{\mathcal{DF}}(\mathcal{C}),
\)
where the right-hand side denotes the \(\infty\)-category of completely filtered objects in \(\mathcal{C}\).
Moreover, by \cite[Remark 4.9]{Ari21}, this equivalence intertwines the functor
\(
(\td)\Big| _{\{n\}}:\mathrm{Ch}^{\bullet}(\mathcal{C}) \longrightarrow \mathcal{C}\)
and \(
\mathrm{gr}^n[n]:
\widehat{\mathcal{DF}}(\mathcal{C}) \longrightarrow \mathcal{C}.
\)
\end{digression}

\begin{construction}
\label{DGr and Ch*}
Let $A$ be a commutative ring and consider the graded associative ring
$A[d]/d^2$, where $d$ is placed in grading $1$.
We specialize the discussion of \Cref{Ariotta identification} to the case
$\mathcal{C}=\mathcal{D}(A)$. 
There is a natural pointed functor
\(
A[d]/d^2(\td)\colon \;\Xi \rightarrow \DGr(A[d]/d^2)
\)
which sends $n\in\Xi$ to $A[d]/d^2(-n)$ and the morphism
$n \xrightarrow{\partial} n-1$ to the map of right multiplication by \(d\): 
\(\;A[d]/d^2(-n) \xrightarrow{ } A[d]/d^2(1-n).
\)
Here, for a graded object $M^\bullet$, we adopt the convention that the twist
is defined by
\(
M(n)^i \coloneqq M^{i+n}.
\)
\footnote{As a sanity check, $A[d]/d^2(1)$ is supported in gradings $-1$ and $0$,
and multiplication by $d$ indeed defines a map
$A[d]/d^2 \to A[d]/d^2(1)$.}
\end{construction}

\begin{proposition}
\label{d gives rise to coherent cochain}
Let the notation be as in \Cref{DGr and Ch*}.
Then the composite of the Yoneda embedding with restriction along the
(opposite of the) functor $A[d]/d^2(\bullet)$ constructed above
\[
\DGr(A[d]/d^2)
\xrightarrow{\ \mathrm{Yoneda}\ }
\mathrm{Fun}_*\!\big(\DGr(A[d]/d^2)^{\mathrm{op}},\mathcal{D}(A)\big)
\xrightarrow{\ \mid_{\Xi^{\mathrm{op}}}\ }
\mathrm{Ch}^{\bullet}(\mathcal{D}(A))
\]
is an equivalence of $\infty$-categories.
\end{proposition}

Denote this composite functor by
\(
\underline{(-)}\colon \DGr(A[d]/d^2)\longrightarrow \mathrm{Ch}^{\bullet}(\mathcal{D}(A)).
\)

\begin{proof}
Concretely, the composite sends an object $M\in \DGr(A[d]/d^2)$ to the functor \(\underline{M}\) with
\(
\underline{M}(n) \coloneqq 
R\Hom_{\DGr(A[d]/d^2)}\!\big(A[d]/d^2(n),M\big)
\in \mathcal{D}(A).
\)
It is formal that the functor $\underline{(-)}$ preserves limits. 
Since $A[d]/d^2(n)$ is compact for every integer $n$, the functor $\underline{(-)}$
also preserves colimits. Moreover, the canonical $t$-structures on both sides are
left and right $t$-complete, and the functor $\underline{(-)}$ is $t$-exact.

We now show that $\underline{(-)}$ is an equivalence. For fully faithfulness,
since the cohomological and grading shifts of $A[d]/d^2$ generate
$\DGr(A[d]/d^2)$ under colimits, it suffices to prove that the natural map
\(
\RHom_{\DGr(A[d]/d^2)}\!\big(A[d]/d^2(m)[n],M\big)
\longrightarrow
\RHom_{\mathrm{Ch}^{\bullet}(\mathcal{D}(A))}
\!\big(\underline{A[d]/d^2(m)[n]},\underline{M}\big)
\)
is an equivalence for all integers $m,n$ and all $M$. 
Since $\underline{(-)}$ preserves limits, using the Postnikov filtration of $M$, we are furthur
reduced to the case where $M$ is a shift of an ordinary graded
$A[d]/d^2$-module. Because $\underline{(-)}$ commutes with shifts, we may further
reduce to the case $n=0$ and $M$ concentrated in cohomological degree $0$.
In this situation both sides identify with the grading $-m$ component of $M$,
and the induced map is the identity. 
As for essential surjectivity, 
since
\begin{itemize}
\item the functor $\underline{(-)}$ preserves both limits and colimits,
\item both categories are left and right $t$-complete, and
\item the functor $\underline{(-)}$ is $t$-exact,
\end{itemize}
it suffices to show that $\underline{(-)}$ induces an equivalence on the hearts: 
We may check explicitly that both hearts are given by cochain complexes in $A$-modules
and the induced functor is identity.
\end{proof}

\begin{remark}
\label{intertwining in d gives rise to coherent cochain}
The equivalence $\underline{(-)}$ intertwines the functors
\((-)^n:\; \DGr(A[d]/d^2) \longrightarrow \mathcal{D}(A)\) and \((-)\Big|_{\{n\}}: \;
\mathrm{Ch}^{\bullet}(\mathcal{D}(A)) \longrightarrow \mathcal{D}(A).
\)
Indeed, for any $M \in \DGr(A[d]/d^2)$ we have
\[
\RHom_{\DGr(A[d]/d^2)}\big(A[d]/d^2(-n), M\big)
\cong
\RHom_{\DGr(A)}\big(A(-n), M\big)
\cong
M^n .
\]
Here the first equivalence follows from the tensor–Hom adjunction, and the second is tautological from the definition of the grading.
\end{remark}

Combining \Cref{Ariotta identification}, \Cref{d gives rise to coherent cochain},
and \Cref{intertwining in d gives rise to coherent cochain}, we obtain the following.

\begin{corollary}
\label{d gives rise to DFhat}
Let the notation be as in \Cref{DGr and Ch*}. Then there is an equivalence
\(
\DGr(A[d]/d^2) \longrightarrow \widehat{\mathcal{DF}}(A).
\)
Moreover, this equivalence intertwines the functors
\((-)^n:\;
\DGr(A[d]/d^2) \longrightarrow \mathcal{D}(A)\) and \(\mathrm{gr}^n[n]: \;
\widehat{\mathcal{DF}}(A) \longrightarrow \mathcal{D}(A).
\)
\end{corollary}

\begin{notation}\label{Totalization}
We denote the functor
\(
\DGr(A[d]/d^2) \longrightarrow \widehat{\mathcal{DF}}(A)
\)
by \(\mathrm{Fil}^{\bullet}\mathrm{Tot}\), referring to it as the \emph{filtered totalization}.
We further denote by
\(
\mathrm{Tot}\colon \DGr(A[d]/d^2) \longrightarrow \mathcal{D}(A)
\)
the composite of this functor with the forgetful functor sending a filtered object to its underlying object; we refer to this as the \emph{totalization}.
\end{notation}

\begin{remark}
\label{Totalization of cochain complex}
Unwinding the construction, one finds that if a cochain complex of \(A\)-modules
is viewed as a graded \(A[d]/d^2\)-module, then its filtered totalization is canonically
identified with the same cochain complex equipped with the \emph{stupid filtration}.
\end{remark}

In view of the totalization functor defined above, the result of Illusie
\cite[Theorem II.1.4]{IlldRW} can be re-stated as follows (see also \cite[Theorem 1.1.2]{BLM}):
\begin{theorem}
\label{dRW-crys comparison}
There is a commutative diagram:
\[
\begin{tikzcd}[column sep=large]
\mathrm{SmSch}_k^{\mathrm{op}} \arrow[rr, "{\mathrm{R\Gamma}(-, W\Omega^{\bullet})}"] 
\arrow[dr, "{\mathrm{R\Gamma}_{\mathrm{\cris}}(-/W(k))}"'] && 
\DGr(\RR) \arrow[dl, "{\mathrm{Tot} \circ \mathrm{forget}}"] \\
& \mathcal{D}(W(k)). & 
\end{tikzcd}
\]
\end{theorem}

\begin{proof}
By \cite[Remark 9.3.5]{BLM} together with \Cref{Totalization of cochain complex},
there is a canonical natural transformation
\[
\mathrm{R\Gamma}(-, W\Omega^{\bullet})
\longrightarrow
\mathrm{R\Gamma}_{\cris}(-/W(k))
\]
when both functors are regarded as defined on affine smooth \(k\)-schemes.
By \cite[Theorem II.1.4]{IlldRW}, this natural transformation is an equivalence.

Note that although the functor \(\mathrm{Tot}\) does not commute with limits in general,
it does commute with limits when the objects involved are uniformly grading-left bounded; that is, when there exists an integer \(N\)
such that all grading pieces \(\leqslant -N\) vanish.
The reason being, in this situation, the totalization is the same as taking the \((-N)\)-th filtration piece,
which clearly commutes with arbitrary limits.

The general case follows from the fact that both functors are Zariski sheaves.
\end{proof}

For later use we record the following observation.

\begin{proposition}
\label{totalization of acyclic}
Let the notation be as in  {\Cref{DGr and Ch*}}. Then the composite
\[
\DGr(A) \xrightarrow{A[d]/d^2 \otimes_A -} \DGr(A[d]/d^2)
\xrightarrow{\mathrm{Tot}} \mathcal{D}(A)
\]
is naturally equivalent to the zero functor.
\end{proposition}

\begin{proof}
For an integer \(j\), instead of taking the underlying object of the filtered totalization,
consider the ``totalization modulo the \((j+1)\)-st filtration'', which we denote by
\(
\mathrm{Tot}/\mathrm{Fil}^{j+1} \colon \DGr(A[d]/d^2) \to \mathcal{D}(A).
\)
This functor commutes with small colimits. Hence the composite
\(
\DGr(A) \xrightarrow{A[d]/d^2 \otimes_A -} \DGr(A[d]/d^2)
\xrightarrow{\mathrm{Tot}/\mathrm{Fil}^{j+1}} \mathcal{D}(A)
\)
also commutes with small colimits.
We claim that this composite is naturally equivalent to the functor
\(
(-)^j[-j] \colon \DGr(A) \longrightarrow \mathcal{D}(A).
\)

First, let us construct a natural transformation. 
Consider the natural transformation of functors 
\[
\mathrm{gr}^j \longrightarrow \mathrm{Tot}/\mathrm{Fil}^{j+1}
\colon \DGr(A[d]/d^2) \longrightarrow \mathcal{D}(A).
\]
By the second part of \Cref{d gives rise to DFhat}, we compute
\[
\mathrm{gr}^j \circ (A[d]/d^2 \otimes_A -)
\cong (A[d]/d^2 \otimes_A -)^j[-j]
\cong \big((-)^{j-1} \oplus (-)^j\big)[-j]
\colon \DGr(A) \longrightarrow \mathcal{D}(A).
\]
Precomposing with the inclusion of the direct summand
\(
(-)^j \hookrightarrow (-)^{j-1} \oplus (-)^j
\)
produces the desired natural transformation. 
We next show that this natural transformation is an equivalence.
Since both functors preserve colimits and \(\DGr(A)\) is generated under colimits by the objects
\(A(m)[n]\) for all integers \(m,n\), it suffices to verify the claim on these generators.
For such objects the statement follows by direct inspection.

Finally, we compute the limit
\(
\mathrm{Rlim}_{j\to\infty}
\mathrm{Tot}/\mathrm{Fil}^{j+1}\circ(A[d]/d^2\otimes_A -)
\cong
\mathrm{Rlim}_{j\to\infty} (-)^j[-j].
\)
We do not need to identify the transition maps explicitly.
Indeed, we claim that any natural transformation
\(
(-)^i \to (-)^j[n]
\)
vanishes whenever \(i\neq j\).
This immediately implies that the above limit is zero. 
To prove the claim, observe that every object of \(\DGr(A)\) admits an action of the ring
\(\prod_{\mathbb{Z}} A\) by grading projections.
Let \(e_i\) denote the idempotent projecting to the \(i\)-th grading component.
Then the endomorphism \(1-e_i\) acts as \(0\) under the functor \((-)^i\),
but as the identity under \((-)^j[n]\) when \(i\neq j\).
Naturality therefore forces any such transformation to vanish.
\end{proof}

\subsection{Coherent \(\RR\)-modules}

As a consequence of \Cref{dRW-crys comparison}, one obtains the
so-called ``slope spectral sequence''
\begin{equation}
\label{SS}
E_1^{i,j} \coloneqq \HH^j(X,W\Omega^i_{X/k}) \Longrightarrow \HH^{i+j}_{\cris}(X/W).
\end{equation}

The main structural result of \cite{IlldRW} may be stated as follows.

\begin{theorem}[{\cite[Theorem II.3.2]{IlldRW}}]
\label{SSbasic}
For a smooth proper variety $X$ over $k$.
Its associated slope spectral sequence \((\ref{SS})\) degenerates at the \(E_1\)-page after inverting \(p\).
Moreover, the isocrystal \(\HH^j(X,W\Omega_{X/k}^i)\otimes_W K\), endowed with Frobenius \(p^iF\),
coincides with the slope interval \([i,i+1)\) of
\(\HH^{i+j}_{\cris}(X/W)\otimes_W K\).
\end{theorem}

However, without inverting \(p\)—that is, when the \(p^\infty\)-torsion is taken into account—the spectral sequence \((\ref{SS})\) need not degenerate at the \(E_1\)-page.
%Illusie showed that the obstruction arises from certain infinite torsion subquotients.
%These are organized into structures called \emph{dominoes}.
%Since they satisfy suitable finiteness properties, their sizes can be measured by the corresponding domino numbers \(T^{i,j}(X)\)
%(see \Cref{numerical invariants of smooth proper}). 
In fact, it is known that the non-degeneration of the slope spectral sequence
is equivalent to non-finiteness of Hodge--Witt cohomology groups.
In particular, the Hodge--Witt cohomology groups of smooth proper varieties over $k$ can be 
non-finitely generated as $W$-modules.
As a replacement for the correct ``finiteness'' condition for Hodge--Witt cohomology groups
of smooth proper varieties over $k$,
%To analyze the finer structure of Hodge--Witt cohomology groups, 
Illusie--Raynaud \cite[Remark I.3.10.1]{IRdRW} and Ekedahl \cite[Definition 0.5.13]{Ek2} 
introduced a distinguished subcategory of graded left \(\RR\)-modules consisting of so-called 
\emph{coherent} graded \(\RR\)-modules (see also \cite[Proposition III.1.1]{Ek2}). 
We briefly recall their definition.

\begin{definition}
\label{complete R-module}
Let \(M\) be a left graded \(\RR\)-module. 
The \emph{standard filtration} on \(M\) is defined by
\[
\Fil^n(M^i) \coloneqq V^n M^i + dV^n M^{i-1}.
\]
These form graded \(W\)-submodules of \(M\).
We say that \(M\) is \emph{classically complete} if for each $i$ 
the natural map \(M^i \to \varprojlim_n M^i/\Fil^n(M^i)\) is an isomorphism.
We say that \(M\) is \emph{profinite} if it is classically complete and each quotient \(M^i/\Fil^n(M^i)\) has finite length as a \(W\)-module.
\end{definition}

In the study of Hodge--Witt cohomology, a fundamental class of graded \(\RR\)-modules naturally arises. 
These are the modules \(U_t\), known as \emph{elementary dominoes}, which first appear in the cohomology of supersingular abelian surfaces and K3 surfaces.

\begin{definition}[{\cite[I.2.(D)]{IRdRW}}]
\label{Ut}
For each integer \(t\in\mathbb Z\), the left graded \(\RR\)-module
\(U_t=U_t^0\oplus U_t^1\) is defined by:
\begin{enumerate}
\item The grading $0$ piece is $U_t^0 \coloneqq k_{\sigma}[[V]]
  = \biggl\{\sum_{n\geqslant 0} a_n V^n \mid a_n\in k\biggr\}$,
where $F$ acts by $0$ and 
\[
V\bigg(\sum_{n\geqslant 0} a_n V^n\bigg)
   = \sum_{n\geqslant 0} a_n^{\sigma^{-1}} V^{n+1};
\]
\item The grading $1$ piece is $U_t^1 \coloneqq
 \prod_{n\geqslant  t } k \cdot dV^n$,
%k_{\sigma}[[dV]]_{\geqslant t}$,
%where  k_{\sigma}[[dV]]_{\geqslant t} \coloneqq \biggl\{\sum_{n\geqslant \max\{0, t\}} a_n dV^n \mid a_n\in k\biggr\} k_{\sigma}[[dV]]_{\geqslant t} \coloneqq \biggl\{\sum_{n\geqslant \max\{0, t\}} a_n dV^n \mid a_n\in k\biggr\}
where \(V\) acts by \(0\) and
\[
F\bigg(\sum_{n\geqslant t} a_n dV^n\bigg)
= 
\sum_{n\geqslant t} a_{n+1}^{\sigma} dV^n.
\]
In the above definition, we adopt the convention that $dV^n=F^{-n}d$ when $n \leqslant 0$.
\item The action of $d$ is given by
\[
d\bigg(\sum_{n\geqslant 0} a_n V^n\bigg)
   = \sum_{n\geqslant \max\{0,t\}} a_n dV^n .
\]
\end{enumerate}
\end{definition}

For \(t\ge0\), the structure of \(U_t\) may be illustrated as follows (see also \cite[p.~108]{IRdRW}):
\[
\begin{tikzcd}
{U_t^0:} & k & kV & \cdots & {kV^{i-1}} & {kV^i} & {kV^{i+1}} & \cdots \\
{U_t^1:} &&&& 0 & {kdV^i} & {kdV^{i+1}} & \cdots
\arrow["d"', from=1-1, to=2-1]
\arrow["V", from=1-2, to=1-3]
\arrow["V", from=1-3, to=1-4]
\arrow["V", from=1-4, to=1-5]
\arrow["V", from=1-5, to=1-6]
\arrow["V", from=1-6, to=1-7]
\arrow["d"', from=1-6, to=2-6]
\arrow["V", from=1-7, to=1-8]
\arrow["d"', from=1-7, to=2-7]
\arrow["F"', from=2-6, to=2-5]
\arrow["F"', from=2-7, to=2-6]
\arrow["F"', from=2-8, to=2-7]
\end{tikzcd}
\]

Illusie and Raynaud \cite[Proposition I.2.19]{IRdRW} proved that every grading-left bounded
(resp.~grading-left and right bounded)
profinite left graded \(\RR\)-module admits a separated and exhaustive (resp.~finite) 
increasing filtration by graded \(\RR\)-submodules whose graded pieces are, up to shift, of one of the following types:

    Type \(\rmI_a\): a finite length
    Dieudonn\'{e} module, with \(V\) nilpotent;
    
    Type \(\rmI_b\) : a finite free Dieudonn\'{e} module, with \(V\) 
    topologically nilpotent\footnote{Here the finite free $W$-module is equipped with the $p$-adic
    topology, so concretely we are saying that the operator $V^{N}$ is divisible by $p$ for some $N$.};
    
    Type \(\rmI_c\): the module \(k_{\sigma}[[V]]\), defined as in \Cref{Ut}; 
    
    Type \(\rmI\rmI\): an elementary domino \(U_t\) also defined in \Cref{Ut}.

\begin{definition}
\label{coeuretc}
Let \(M\) be a profinite left graded \(\RR\)-module.

\begin{enumerate}
\item (\cite[D\'efinition I.3.8]{IRdRW})
We say that \(M\) is \emph{coherent} if it admits a finite filtration by graded \(\RR\)-submodules whose successive quotients are of type \(\rmI_a\), \(\rmI_b\), or type \(\rmI\rmI\). In particular, coherent left graded $\RR$-modules have bounded gradings. 

\item (\cite[D\'efinition II.3.1]{IRdRW}) The \emph{heart} of \(M^i\) is defined by
\[
\text{c\oe ur}(M^i)\coloneqq V^{-\infty}Z^{i}/F^{\infty}B^{i},
\]
where
\[
V^{-\infty}Z^{i}\coloneqq \bigcap_{n}\ker(dV^n: M^i\to M^{i+1}),
\qquad
F^{\infty}B^{i}\coloneqq \bigcup_{n}F^n\operatorname{im}(d: M^{i-1}\to M^i).
\]
One checks that \(V^{-\infty}Z^{i}\) is the largest \(\RR^0\)-submodule contained in \(\ker(d:M^i\to M^{i+1})\), while \(F^{\infty}B^{i}\) is the smallest \(\RR^0\)-submodule containing \(\operatorname{im}(d:M^{i-1}\to M^i)\).

\item The \(i\)-th \emph{domino} of \(M\) is the graded \(\RR\)-module
\[
M^i/V^{-\infty}Z^{i}\longrightarrow F^\infty B^{i+1},
\]
concentrated in degrees \(0\) and \(1\).
According to \cite[Proposition I.2.18]{IRdRW}, it is an iterated extension of elementary dominoes \(U_t\).
The number of such factors is called the $i$-th domino number of $M$
and is denoted \(T^{i}(M)\) (see the remark below for why this is well-defined).
\end{enumerate}
\end{definition}

\begin{remark}
Let \(M\) be a profinite left graded \(\RR\)-module concentrated in finitely many gradings.
By \cite[Theorem I.3.8]{IRdRW}, we have that \(M\) is coherent if and only if 
\(\coeur(M^i)\) is finitely generated for all \(i\in\mathbb Z\).
Moreover, by \cite[Proposition I.2.18]{IRdRW}, the \(i\)-th domino number admits an alternative description
\[
T^i(M)\coloneqq \dim_k M^i/(V^{-\infty}Z^{i}+VM^i).
\]

Assuming that \(M\) is coherent, it follows from \Cref{independence} that \(T^i(M)\) is also equal to the number of factors of the form \(U_t(-i)\), for varying \(t\), appearing in the graded pieces of any filtration of \(M\) satisfying the conclusion of \cite[Proposition I.2.19]{IRdRW}.

The category of coherent \(\RR\) modules is an abelian subcategory \cite[page.~55]{Ek2}.\footnote{This is a consequence of the fact that \(\DGr_c^b(\RR)\) is a triangulated subcategory, see the discussion after \Cref{alternative definition of coherence}.}
%, but it is itself not easy to prove directly. 
%For instance, it is not clear if any extension between two profinite \(R\) modules is always profinite.} %\textcolor{red}{This is an abelian subcategory, yes?} 
\end{remark}

\begin{definition} The category \(\DGr_c^b(\RR)\) (resp.~$\DGr_c^{-}(\RR)$, resp.~$\DGr_c(\RR)$)
is the full subcategory of \(\DGr^b(\RR)\) (resp.~$\DGr^{-}(\RR)$, resp.~$\DGr(\RR)$) spanned by those objects
whose cohomology are all coherent as left graded \(\RR\)-modules.
\end{definition}

We now recall several invariants associated with a coherent left graded \(\RR\)-module \(M\), as introduced by Illusie--Raynaud \cite{IRdRW} and further studied by Crew \cite{Cr} and Ekedahl \cite{Ek3}. 

\begin{definition}
\label{numerical invariants}
Let $M \in \DGr_c(\RR)$, following Ekedahl \cite[Chapter IV.3.1]{Ek3}, we define some numerical invariants: 

\begin{enumerate} \item the \emph{domino number} \(T^{i,j}(M)\) is defined to be the \(i\)-th domino number of the coherent \(\RR\)-module \(\HH^j(M)\). \item \label{mij} The \emph{slope number}\footnote{In \cite[Definition IV.3.4]{Ek3}, the Hodge polygon formed by these numbers is called the Newton--Hodge polygon.} $m^{i,j}(M)$ is defined by: \begin{equation} \begin{aligned} m^{i,j}(M)\coloneqq &\sum_{\lambda\in [0,1)}(1-\lambda)m_{\lambda}( \text{c\oe ur}(\HH^j(M)^i)\otimes K)\\ &+\sum_{\lambda\in [0,1)}\lambda m_{\lambda}(\text{c\oe ur}(\HH^{j+1}(M)^{i-1})\otimes K). \end{aligned} \end{equation} where $m_{\lambda}(H)$ denotes the multiplicity of slope $\lambda$ in an (iso)crystal $H$. \item The \emph{Hodge--Witt number} \(h_W^{i,j}(M)\) is defined by: \[h_W^{i,j}(M)\coloneqq m^{i,j}(M)+T^{i,j}(M)-2T^{i-1,j+1}(M)+T^{i-2,j+2}(M).\] \end{enumerate} \end{definition}

\begin{remark}
The slope numbers admit a combinatorial interpretation (see \cite[Proposition IV.2.5]{Ek3}). 
Let \(M\in \DGr_c(\RR)\). Then the element \(d\in \RR^1\) acts by \(0\) on the graded \(K\)-vector space
\(\HH^i(M)[1/p]\).
Let \(n\) be an integer such that
\(
\bigoplus_{i\in\mathbb Z}\HH^i(M)^{\,n-i}[1/p]
\)
is finite-dimensional over \(K\).
We may then consider the isocrystal
\(
\left(\bigoplus_{i\in\mathbb Z}\HH^i(M)^{\,n-i}[1/p],\varphi\right),
\)
where \(\varphi\) acts as \(p^iF\) on \(\HH^i(M)^{\,n-i}[1/p]\), and form its Newton polygon.
The \emph{Newton--Hodge polygon} is the polygon obtained by assigning slope \(i\) with multiplicity \(m^{i,n-i}\).
It is then the maximal convex polygon lying below the Newton polygon with the same endpoints and with integral slopes.
\end{remark}

\subsection{Ekedahl's completion functor}
The following result, due to Illusie--Raynaud, is fundamental: 

\begin{theorem}[{\cite[Theorem II.2.2]{IRdRW}}]
\label{coherent}
    Let \(X\) be a proper smooth variety over the perfect field \(k\) of dimension $n$, then
    \[R\Gamma(X,W\Omega^\bullet_{X/k})\in \DGr_c^b(\RR).\]
    That is, the left graded $\RR$-module associated to the complex
    \begin{equation*}\label{Cplx}
        \HH^j(X,WO_{X/k}) \xrightarrow{d} \HH^j(X,W\Omega^1_{X/k}) \xrightarrow{d} \ldots \xrightarrow{d} 
        \HH^j(X,W\Omega^n_{X/k})
    \end{equation*} is coherent for all $j$.
\end{theorem}

Let us sketch an alternative proof of this fact, due to Ekedahl. To do so, we take a brief detour to introduce his completion functor on $\DGr(\RR)$ (see \cite[page~63]{Ek2}).

\begin{digression}[Naturally induced symmetric monoidal structures]
\label{Digression: lim is lax symmetric monoidal}
Let $\mathcal{C}^{\otimes}$ be a small symmetric monoidal $\infty$-category. Then its opposite category $\mathcal{C}^{\mathrm{op}}$ inherits a natural symmetric monoidal structure in which the tensor product of objects is unchanged; see \cite[Remark 2.4.2.7]{HA}.
There is also a canonical symmetric monoidal structure on $\mathrm{Ind}\td\mathcal{C}$, the category of Ind-objects of $\mathcal{C}$, characterized by requiring that the Yoneda embedding \(j \colon \mathcal{C} \to \mathrm{Ind}\td\mathcal{C}\)
be symmetric monoidal and that the tensor product preserve small filtered colimits separately in each variable; see \cite[Corollary 6.3.1.13]{HA}.

Consequently, the category
\(
\mathrm{Pro\text{-}}\mathcal{C} = (\mathrm{Ind}\td\mathcal{C}^{\mathrm{op}})^{\mathrm{op}}
\)
inherits a natural symmetric monoidal structure for which the canonical Yoneda embedding
\(
j \colon \mathcal{C} \to \mathrm{Pro}\td\mathcal{C}
\)
is symmetric monoidal. Unwinding the definitions, if
$\{X_{\lambda}\}_{\lambda \in \Lambda}$ and $\{Y_{\gamma}\}_{\gamma \in \Gamma}$
are two pro-systems, then their tensor product is given by
\[
\{X_{\lambda}\}_{\lambda \in \Lambda} \otimes_{\mathrm{Pro}\td\mathcal{C}}
\{Y_{\gamma}\}_{\gamma \in \Gamma}
=
\{X_{\lambda} \otimes_{\mathcal{C}} Y_{\gamma}\}_{(\lambda,\gamma)\in\Lambda\times\Gamma}.
\]

Finally, suppose that $\mathcal{C}$ admits sufficiently many small limits so that the symmetric monoidal functor
\(
j \colon \mathcal{C} \to \mathrm{Pro}\td\mathcal{C}
\)
admits a right adjoint
\(
\lim \colon \mathrm{Pro}\td\mathcal{C} \to \mathcal{C}.
\)
Then this right adjoint $\lim$ is naturally a lax symmetric monoidal functor.
\end{digression}

% right adjoint of symmetric monoidal functor has a natural lax symmetric monoidal structure, by Proposition A of the paper "Lax monoidal adjunctions, two-variable fibrations and the calculus of mates". The authors said that this was shown by Lurie in HA, but I didn't find the exact location where this is stated.

%\subsection{Tensoring \(\RR_n\)}

After this interlude, we can introduce the completion functor.
For each \(n\geqslant 1\), define \(\RR_n \coloneqq \RR/(V^n\RR + dV^n\RR)\).
Then \(\RR_n\) naturally carries the structure of a graded \((W_n[d]/d^2, \RR)\)-bimodule.
The natural projection induces graded maps
\(\RR_{n+1} \xrightarrow{\pi} \RR_n,\)
while left multiplication by \(F\) (resp.~\(V\), resp.~\(d\)) induces graded maps
\(
F \colon \RR_n \to \sigma_*\RR_{n-1}, \quad
V \colon \sigma_*\RR_{n-1} \to \RR_n, \quad
d \colon \RR_n \to \RR_n(1).
\)
Consequently, the pro-system \(\{\RR_n\}_n\) (with transition maps given by \(\pi\)) forms a graded \((\RR,\RR)\)-bimodule object in \(\mathrm{Pro}\td\mathrm{Mod}_{\mathbb{Z}_p}\), the abelian category of Pro-\{\(\mathbb{Z}_p\) modules\}. 
From now on, we shall view it as a graded \((\RR,\RR)\)-bimodule object in \(\mathrm{Pro}\td\DGr(\mathbb{Z}_p)\).

\begin{construction}[completion functor, cf.~{\cite[page~63]{Ek2}}]
\label{completion on DGr(R)}
Note that \(\DGr(\RR)=\mathrm{LMod}_\RR(\DGr(\mathbb{Z}_p))\).
Let us consider the following composite functor
\[
\DGr(\RR) \xrightarrow{j} \mathrm{LMod}_\RR(\mathrm{Pro}\td\DGr(\mathbb{Z}_p))
\xrightarrow{\{\RR_n\}_n \otimes_\RR^{\rmL} -} \mathrm{LMod}_\RR(\mathrm{Pro}\td\DGr(\mathbb{Z}_p))
\xrightarrow{\lim} \DGr(\RR).
\]
This defines an endo-functor on \(\DGr(\RR)\), which we call the \emph{completion functor} and denote by \(\widehat{(-)}\).
Here we use the symmetric monoidal structure on \(\mathrm{Pro}\td\DGr(\mathbb{Z}_p)\) 
induced from that on \(\DGr(\mathbb{Z}_p)\), as explained in \Cref{Digression: lim is lax symmetric monoidal}.
Moreover, the functor \(\lim\) preserves the left \(\RR\)-module structure because it is lax symmetric monoidal (again by \Cref{Digression: lim is lax symmetric monoidal}).
We shall denote by \(\widehat{\DGr(\RR)}\) the full subcategory of \(\DGr(\RR)\) spanned by the essential image of \(\widehat{(-)}\).
\end{construction}

\begin{proposition}[{\cite[Proposition I.3.2]{IRdRW}}]
\label{Rn is right perfect}
For each integer $n\geqslant 1$, the right graded $\RR$-module $\RR_n$ admits the following resolution:
\begin{equation}
\label{Rn finite flat dim}
0\rightarrow \RR(-1)\xrightarrow{u_n} \RR(-1)\oplus \RR
\xrightarrow{v_n} \RR
\rightarrow \RR_n\rightarrow 0,
\end{equation}
where \(u_n(x)=(F^nx,-F^ndx)\) and \(v_n(x,y)=dV^nx+V^ny\).

Moreover, left multiplication by a scalar \(c\in W\) is compatible with this resolution and is described by the following commutative diagram of right graded \(\RR\)-modules:
\[
\begin{tikzcd}
	0 & {\RR(-1)} & {\RR(-1)\oplus \RR} & \RR & {\RR_n} & 0 \\
	0 & {\RR(-1)} & {\RR(-1)\oplus \RR} & \RR & {\RR_n} & 0
	\arrow[from=1-1, to=1-2]
	\arrow["{u_{n}}", from=1-2, to=1-3]
	\arrow["{c}"', from=1-2, to=2-2]
	\arrow["{v_{n}}",from=1-3, to=1-4]
	\arrow["{\bigl(\sigma^n(c) ,\sigma^n(c)\bigr)}"', from=1-3, to=2-3]
	\arrow[from=1-4, to=1-5]
	\arrow["{c \cdot}"', from=1-4, to=2-4]
	\arrow[from=1-5, to=1-6]
	\arrow["{c\cdot}"', from=1-5, to=2-5]
	\arrow[from=2-1, to=2-2]
	\arrow["{u_{n}}", from=2-2, to=2-3]
	\arrow["{v_{n}}", from=2-3, to=2-4]
	\arrow[from=2-4, to=2-5]
	\arrow[from=2-5, to=2-6]
\end{tikzcd}
\]

Similarly, left multiplication by \(d\) is described by the following commutative diagram of right graded \(\RR\)-modules:
\[
\begin{tikzcd}
	0 & {\RR(-1)} & {\RR(-1)\oplus \RR} & \RR & {\RR_n} & 0 \\
	0 & \RR & {\RR\oplus \RR(1)} & {\RR(1)} & {\RR_n(1)} & 0
	\arrow[from=1-1, to=1-2]
	\arrow["{u_{n}}", from=1-2, to=1-3]
	\arrow["{-d\cdot }"', from=1-2, to=2-2]
	\arrow["{v_{n}}",from=1-3, to=1-4]
	\arrow["{w_n}"', from=1-3, to=2-3]
	\arrow[from=1-4, to=1-5]
	\arrow["{d \cdot}"', from=1-4, to=2-4]
	\arrow[from=1-5, to=1-6]
	\arrow["{d\cdot}"', from=1-5, to=2-5]
	\arrow[from=2-1, to=2-2]
	\arrow["{u_{n}}", from=2-2, to=2-3]
	\arrow["{v_{n}}", from=2-3, to=2-4]
	\arrow[from=2-4, to=2-5]
	\arrow[from=2-5, to=2-6]
\end{tikzcd}
\]
where \(w_n=\begin{pmatrix}
0&1\\
0&0
\end{pmatrix}\).

In particular, the derived functor
\[
\RR_n\otimes_\RR^{\rmL}-\colon \DGr(\RR)\longrightarrow \DGr(W_n[d]/d^2)
\]
has amplitude contained in $[-2,0]$.

Letting $n$ vary, the above resolutions are compatible and fit into the following commutative diagram:
\[
\begin{tikzcd}
	0 & {\RR(-1)} & {\RR(-1)\oplus \RR} & \RR & {\RR_{n+1}} & 0 \\
	0 & {\RR(-1)} & {\RR(-1)\oplus \RR} & \RR & {\RR_n} & 0
	\arrow[from=1-1, to=1-2]
	\arrow["{u_{n+1}}", from=1-2, to=1-3]
	\arrow["p"', from=1-2, to=2-2]
	\arrow["{v_{n+1}}", from=1-3, to=1-4]
	\arrow["{(V,\;V) }"', from=1-3, to=2-3]
	\arrow[from=1-4, to=1-5]
	\arrow["{=}"', from=1-4, to=2-4]
	\arrow[from=1-5, to=1-6]
	\arrow[from=1-5, to=2-5]
	\arrow[from=2-1, to=2-2]
	\arrow["{u_n}", from=2-2, to=2-3]
	\arrow["{v_n}", from=2-3, to=2-4]
	\arrow[from=2-4, to=2-5]
	\arrow[from=2-5, to=2-6]
\end{tikzcd}
\]

Consequently, these resolutions assemble into a resolution of the pro-system
\(\{\RR_n\}\) as a right graded $\RR$-module object in \(\mathrm{Pro}\td\mathrm{Mod}_{\mathbb{Z}_p}\):
\begin{equation}
\label{pro-system resolution of Rn}
0\rightarrow \{\RR(-1)\}_p
\xrightarrow{u}
\{\RR(-1)\oplus \RR\}_V
\xrightarrow{v}
\{\RR\}
\rightarrow \{\RR_n\}
\rightarrow 0 .
\end{equation}
Here the three pro-systems are indexed by the natural numbers, with the terms as above, and with transition maps given by left multiplication by $p$, $V$, and $1$, respectively.
\end{proposition}

Notice that there is a natural morphism
\(j(\RR) \longrightarrow \{\RR_n\}\)
of graded $(\RR,\RR)$-bimodule objects in $\mathrm{Pro}\td\DGr(W)$. 
Consequently, we obtain a natural transformation
\[
\eta \colon \id \longrightarrow \widehat{(-)}
\]
between endofunctors of $\DGr(\RR)$.

\begin{proposition}[cf.~{\cite[Proposition 2.1]{Ek2}}]
\label{completion is a localization}
The functor $\widehat{(-)}$ is a localization.
\end{proposition}
\begin{proof}
We verify the criterion of \cite[Proposition 5.2.7.4.(3)]{HTT} for the natural transformation $\eta$. 
Thus it suffices to show that the two natural transformations
\[
\eta \circ \widehat{(-)} \;\;\text{and}\;\;
\widehat{(\eta)} \colon \widehat{(-)} \longrightarrow \widehat{\widehat{(-)}}
\]
are equivalences.
Since the forgetful functor $\DGr(\RR)\to\DGr(W)$ is conservative, it suffices to verify this after composing with the forgetful functor. 
Unwinding the definitions, the double completion functor is given by
\[
\lim_{\mathbb{N}}\Big(\{\RR_n\}\otimes_\RR^{\rmL} \big(\lim_{\mathbb{N}}(\{\RR_m\}\otimes_\RR^{\rmL}-)\big)\Big)
\simeq
\lim_{\mathbb{N}\times\mathbb{N}}
\big(\{\RR_n\}\otimes_\RR^{\rmL}\{\RR_m\}\otimes_\RR^{\rmL}-\big).
\]
Here the interchange of limits follows from the fact that each $\RR_n$ is perfect as a right $\RR$-module (see \eqref{Rn finite flat dim}), so that the functor $\RR_n\otimes_\RR^{\rmL}-$ commutes with limits.

Under this identification, it suffices to show that both maps
\[
\id\otimes 1 \text{ and } 1\otimes\id \colon 
\{\RR_n\}\longrightarrow \{\RR_n\}\otimes_\RR^{\rmL}\{\RR_n\}
\]
are equivalences in $\mathrm{RMod}_\RR(\mathrm{Pro}\td\DGr(W))$.
Applying the resolution \eqref{pro-system resolution of Rn} to the first copy of $\{\RR_n\}$, we find that
$\{\RR_n\}\otimes_\RR^{\rmL}\{\RR_n\}$ is represented by the following complex of pro-systems,
concentrated in cohomological degrees $[-2,0]$:
\[
\text{``}\lim_{p\cdot -}\text{''}\{\RR_n(-1)\}
\longrightarrow
\text{``}\lim_{V\cdot -}\text{''}\{\RR_n(-1)\oplus \RR_n\}
\longrightarrow
\{\RR_n\}.
\]
Moreover, the morphism
\(
\RR\otimes_\RR \{\RR_n\} \xrightarrow{1\otimes\id}
\{\RR_n\}\otimes_\RR^{\rmL}\{\RR_n\}
\)
is induced by the stupid truncation of this representing complex.
Since the first two terms are pro-$0$\footnote{Indeed, each $\RR_n$ is annihilated by $p^n$, and the composite
$\RR_n\xrightarrow{V^n\cdot -}\RR_{2n}\xrightarrow{\pi}\RR_n$ vanishes by definition.},
it follows that $1\otimes\id$ is an equivalence.

On the other hand, the composite
\(
(1\otimes\id)^{-1}\circ(\id\otimes 1)
\)
is an endomorphism of $\{\RR_n\}$ which commutes with the canonical $(\RR,\RR)$-bimodule map
$j(\RR)\to\{\RR_n\}$. 
As a graded right $\RR$-module morphism, such an endomorphism must necessarily be the identity. 
Hence $\id\otimes 1$ is also an equivalence, completing the proof.
\end{proof}

\begin{proposition}
\label{completeness closed under taking limit}
The full subcategory $\widehat{\DGr}(\RR) \subset \DGr(\RR)$ is closed under limits.
\end{proposition}

\begin{proof}
By \Cref{completion is a localization}, it suffices to show that the completion functor
\[
\widehat{(-)} \coloneqq \mathrm{\RR}\!\lim_{n \in \mathbb{N}} \bigl(\RR_n \otimes_\RR^{\rmL} -\bigr)
\]
commutes with small limits. This follows from the fact that both
$\mathrm{\RR}\!\lim_{n \in \mathbb{N}}(-)$ and the functor
$\RR_n \otimes_\RR^{\rmL} -$ commute with small limits. 
The latter property holds because each $\RR_n$ is perfect as a right $\RR$-module (see \Cref{Rn is right perfect}).
\end{proof}

Let us also record the following important fact.

\begin{proposition}[{\cite[Proposition 2.1]{Ek2}}]
\label{tensor Rn and completion}
The natural transformation of functors
\[
\RR_n \otimes_\RR^{\rmL} - 
\xrightarrow{\,\RR_n \otimes_\RR^{\rmL} \eta\,}
\RR_n \otimes_\RR^{\rmL} \widehat{(-)}
\colon \DGr(\RR) \longrightarrow \DGr(W_n[d]/d^2)
\]
is an equivalence.
\end{proposition}

\begin{warning}
%\textcolor{red}{Complete as graded $\RR$-mod vs complete as graded $\RR$-complex.}
We warn the readers that for a left graded \(\RR\)-module \(M\), it is unclear to us if
being classically complete as a graded left $\RR$-module in \Cref{complete R-module} is related to
it being complete when viewed as an object in $\DGr(\RR)$ in the sense of \Cref{completion on DGr(R)}.
\end{warning}

From now on, unless stated otherwise we always use ``complete'' in the sense of
\Cref{completion on DGr(R)}.

\begin{proposition}[{\cite[Proposition III.1.1]{Ek2}}]
\label{alternative definition of coherence}

    Let $M\in \DGr^-(\RR)$, then the following conditions are equivalent:

{\rm (1)} $M\in \DGr_c^-(\RR)$;

{\rm (2)} $M$ is complete and for all $n$ the object $\RR_n\otimes^{\rmL}_\RR M\in \DGr^-(W_n)$ 
obtained by forgetting grading and the action of $d$ has finitely generated cohomology.

{\rm (3)} $M$ is complete and $\RR_1\otimes^{\rmL}_\RR M\in \DGr^-(k)$
obtained by forgetting grading and the action of $d$ has finite dimensional cohomology.
\end{proposition}

The above proposition shows that $\DGr_c^-(\RR)$ is a triangulated subcategory of $\DGr^-(\RR)$.
In particular, coherent $\RR$-modules form a thick abelian subcategory of the category of left graded
$\RR$-modules. Consequently, $\DGr_c(\RR)$ (resp.~$\DGr_c^b(\RR)$) is also a triangulated subcategory of
$\DGr(\RR)$ (resp.~$\DGr^b(\RR)$).

\begin{definition}
\label{Hodge number of DGrc}
Let $M \in \DGr_c(\RR)$. 
By \Cref{alternative definition of coherence}, after passing to the colimit we know that for every pair $(i,j)$
the group
\(
\HH^j(\RR_1 \otimes_\RR^{\rmL} M)^i
\)
is a finite-dimensional $k$-vector space. We denote its dimension by $h^{i,j}(M)$ and call it the
\((i,j)\)-th \emph{Hodge number} of $M$.
\end{definition}

The above definition is justified by the following result.

\begin{theorem}[{\cite[Th\'{e}or\`{e}me II.1.2]{IRdRW}}]
\label{Wn and Rn}
For any smooth scheme $X$ over the perfect field $k$, there are functorial identifications
\[
W_n\Omega^\bullet_{X/k}
\cong
\RR_n \otimes_\RR W\Omega^\bullet_{X/k}
\cong
\RR_n \otimes_\RR^{\rmL} W\Omega^\bullet_{X/k}.
\]
\end{theorem}

We can now give the promised proof of \Cref{coherent}.

\begin{proof}[Proof of \Cref{coherent}]
By \Cref{alternative definition of coherence} and \Cref{Wn and Rn}, the claim follows immediately
from the finiteness of the Hodge cohomology of the smooth proper variety $X/k$.
\end{proof}

Next, let us introduce some homological algebra concerning the functor
\(
\RR_1 \otimes^{\rmL}_\RR - \colon \DGr(\RR) \to \DGr(k[d]/d^2).
\)

\begin{lemma}[{c.f.~\cite[Lemma III.5.6.1]{Ek1}}]
\label{Hodge-de Rham via tensor R1}
Let $M \in \DGr(\RR)$. Then there is a natural map
\[
\RR/p \otimes^{\rmL}_\RR M \to \RR_1 \otimes^{\rmL}_\RR M
\]
in $\DGr(k[d]/d^2)$, which becomes an isomorphism after applying the functor $\mathrm{Tot}$ from \Cref{Totalization}.
\end{lemma}

\begin{proof}
There is a natural surjection $\RR/p \to \RR_1$ of graded $(k[d]/d^2,\RR)$-bimodules,
giving rise to the map in the statement.
Let us compute its kernel: it is $\bigoplus_{n \geqslant 1} k \cdot V^n$ in grading $0$ and
$\bigoplus_{n \geqslant 1} k \cdot dV^n$ in grading $1$.
Note that the right $\RR$-module structure has the action of $d$ given by $0$.
Hence, as a graded $(k[d]/d^2,\RR)$-bimodule, we can describe it as
\(
k[d]/d^2 \otimes_k \bigoplus_{n \geqslant 1} k \cdot V^n.
\)
Therefore, the fiber of
\(
\RR/p \otimes^{\rmL}_\RR M \to \RR_1 \otimes^{\rmL}_\RR M
\)
becomes
\(
k[d]/d^2 \otimes_k
\Bigl((\bigoplus_{n \geqslant 1} k \cdot V^n) \otimes^{\rmL}_\RR M \Bigr).
\)
This object vanishes after applying $\mathrm{Tot}$, thanks to \Cref{totalization of acyclic}.
\end{proof}

\begin{construction}
\label{Hodge--de Rham spectral sequence}
As a consequence of \Cref{Hodge-de Rham via tensor R1},
for any grading-left bounded object \(M\in \DGr^{l}(\RR)\), we obtain the spectral sequence
\[
E_1^{i,j}=\HH^j(\RR_1\otimes^{\rmL}_\RR M)^i
\Rightarrow
\HH^{i+j}(k\otimes^{\rmL}_W \Tot(M)),
\]
which we refer to as the Hodge--de Rham spectral sequence for \(M\).
\end{construction}

Let us also record a result of Ekedahl concerning the homological algebra of $\RR_n\dtimes_\RR -$.
Note that the original statement in loc.~cit.~contains a typo;
here we state the corrected version based on the proof given there.

\begin{proposition}[{\cite[Proposition I.1.1]{Ek2}}]
\label{Ek2I.1.1}
Let \(A\) be a thick subcategory of \(W\)-modules (ungraded)
stable under \(\sigma_*\), and let \(M\in \DGr(\RR)\).

{\rm (1)} If \(M\) is either grading-left bounded or bounded from above,
and for some \(r,s\in \Z\), \(\HH^j(\RR_1\dtimes_\RR M)^i\in A\) for all
\[
\begin{aligned}
(i,j)\in \;\;&\bigl\{(i,j)\in \Z^2 \mid i+j=r,\; j\geqslant s\bigr\}\\
\cup &\bigl\{(i,j)\in \Z^2 \mid i+j=r+1,\; j\geqslant s+1\bigr\}\\
\cup &\bigl\{(i,j)\in \Z^2 \mid i+j=r-1,\; j\geqslant s\bigr\},
\end{aligned}
\]
then for all \(n\), \(\HH^j(\RR_n\dtimes_\RR M)^i\in A\)
for all \((i,j)\in \Z^2\) such that \(i+j=r\) and \(j\geqslant s\).

{\rm (2)} If \(M\) is either grading-right bounded or bounded from below,
and for some \(r,s\in \Z\), \(\HH^j(\RR_1\dtimes_\RR M)^i\in A\) for all
\[
\begin{aligned}
(i,j)\in\;\; &\bigl\{(i,j)\in \Z^2 \mid i+j=r,\; j\leqslant s\bigr\}\\
\cup &\bigl\{(i,j)\in \Z^2 \mid i+j=r+1,\; j\leqslant s\bigr\}\\
\cup &\bigl\{(i,j)\in \Z^2 \mid i+j=r-1,\; j\leqslant s-1\bigr\},
\end{aligned}
\]
then for all \(n\), \(\HH^j(\RR_n\dtimes_\RR M)^i\in A\)
for all \((i,j)\in \Z^2\) such that \(i+j=r\) and \(j\leqslant s\).

In particular, by taking $A=\{0\}$ we obtain the following consequence:
suppose $M$ is bounded in one of the four directions.
Then $M=0$ if and only if $M$ is complete and $\RR_1 \otimes^\rmL_\RR M=0$.
\end{proposition}

\begin{proposition}
\label{coherence and Hodge-d-equivalence gives de Rham--Witt d-equivalence}
Let \(M\in \DGr_c(\RR)\), and assume that it is grading left-bounded. 
Suppose \(\HH^j(\RR_1\otimes^{\rmL}_\RR M)^i=0\) for all \(i+j< d\). Then \(\HH^j(M)^i=0\) for all \(i+j< d\) as well. 
\end{proposition}

\begin{proof}
By \Cref{Ek2I.1.1}, we see that \(\HH^j(M)^i=0\) for all \(i+j< d-1\). 
It therefore suffices to show that \(\HH^j(M)^i=0\) for all \(i+j=d-1\).

By assumption, there exists an integer \(N\) such that \(\HH^j(M)^a=0\) for all \(a < N\).
In particular, we see that \(\HH^{(d-1-i)}(M)^i = 0\) when \(i\) is sufficiently small.
We prove by induction on \(i\) that this vanishing holds for all \(i\).
Suppose it has been shown for all \(i < m\); we need to prove that
\(\HH^{(d-1-m)}(M)^m = 0\).

Consider the spectral sequence of left graded \(k[d]/d^2\)-modules:
\[
E_2^{\ell,j}=\Tor_{-\ell}^{\RR}(\RR_1,\HH^j(M))\Rightarrow \HH^{\ell+j}(\RR_1\otimes^{\rmL}_\RR M).
\]
By \Cref{Rn is right perfect}, the above groups vanish whenever \(\ell \notin [-2,0]\).
Moreover, if \(M\) is grading-left bounded by \(D\), then \(\Tor_{2}^\RR(\RR_1,M)\) is grading-left bounded by \(D+1\).
As a result, the cokernel of the map
\[
\Tor_2^\RR(\RR_1,\HH^{j+1}(M))\rightarrow \Tor_0^\RR(\RR_1,\HH^j(M))
\]
is a subobject of \(\HH^{j}(\RR_1\otimes^{\rmL}_\RR M)\).

Now consider \(\HH^{(d-1-m)}(\RR_1\otimes^{\rmL}_\RR M)^m\).
Our assumption implies that it is \(0\).
The induction hypothesis implies that \(\HH^{(d - m)}(M)^{(m - 1)} = 0\).
From the initial observation, we also know that \(\HH^{< (d-m)}(M)^{(m-1)} = 0\).
Hence \(\HH^{(d - m)}(M)\) is grading-left bounded by \(m\).
Consequently,
\(\Tor_2^\RR(\RR_1,\HH^{(d-m)}(M))\) is grading-left bounded by \(m + 1\).

Combining this with the previous paragraph, we conclude that
\(\Tor_0^\RR(\RR_1,\HH^{(d-1-m)}(M))\) has no component in grading \(m\).
Since \(\HH^{(d-1-m)}(M)\), again by the result observed at the beginning,
is grading-left bounded by \(m\), we conclude that
\[
\Tor_0^\RR(\RR_1,\HH^{(d-1-m)}(M))^m
=
\HH^{(d-1-m)}(M)^m / V\cdot \HH^{(d-1-m)}(M)^m
=0.
\]

Therefore,
\(\HH^{(d-1-m)}(M)^m / V^n\cdot \HH^{(d-1-m)}(M)^m = 0\)
for all positive integers \(n\).
By assumption, the cohomology \(\HH^{(d-1-m)}(M)\) is coherent,
which by definition implies that it is classically complete.
Hence the above vanishing implies that
\(\HH^{(d-1-m)}(M)^m = 0\).
\end{proof}

%To compare the slope numbers with the Hodge numbers, Crew and Ekedahl introduced the Hodge--Witt numbers \(h_W^{i,j}\). 
\subsection{Hodge--Witt numbers}

Below we recall several relations between Hodge--Witt numbers and Hodge numbers.

\begin{theorem}[Crew's formula, {\cite[Theorem 4]{Cr}}, {\cite[Theorem IV.3.2(2)]{Ek3}}]
\label{CrewFor} 
Let \(M\in \DGr_c^b(\RR)\). 
For any natural number \(i\), one has
\begin{equation}
    \sum_{j}(-1)^j h_W^{i,j}
    =
    \sum_{j}(-1)^j h^{i,j}.
\end{equation}
\end{theorem}

\begin{corollary}[see also {\cite[Lemma 2.5]{MR15}}]
\label{independence}
For any short exact sequence of coherent \(\RR\)-modules
\[
0 \rightarrow M_1 \rightarrow M \rightarrow M_2 \rightarrow 0,
\] 
for every \(i \in \mathbb{Z}\) one has
\[
T^i(M)=T^i(M_1)+T^i(M_2).
\]
Consequently, let \(M\) be a coherent \(\RR\)-module and let
\(M_\bullet\) be any filtration satisfying the conclusion of
 {\cite[Proposition I.2.19]{IRdRW}}.
Then the multiplicity of the $i$-th shift of type \(\mathrm{II}\) in the associated graded object
is independent of the choice of filtration.
\end{corollary}

\begin{proof}
By definition, all differentials \(d\) of a coherent
\(\RR\)-module become zero after inverting \(p\). Consequently, the multiplicity of a slope
\(\lambda\), namely
\(m_{\lambda}(\coeur(M^i)\otimes K),\)
is additive in short exact sequences of coherent \(\RR\)-modules.
It follows that the slope numbers \(m^{i,j}\) and the quantities
\[
\sum_{\lambda\in [0,1)}(1-\lambda)m_{\lambda}(\coeur(M^i)\otimes K),
\qquad
\sum_{\lambda\in [0,1)}\lambda m_{\lambda}(\coeur(M^{i-1})\otimes K)
\]
are additive for any short exact sequence of coherent \(\RR\)-modules.

Viewing \(M,M_1,M_2\) as objects of \(\DGr_c^b(\RR)\), the short exact sequence
\(0\to M_1\to M\to M_2\to 0\) induces an exact triangle
\(
\RR_1\otimes_\RR^{\rmL} M_1
\longrightarrow
\RR_1\otimes_\RR^{\rmL} M
\longrightarrow
\RR_1\otimes_\RR^{\rmL} M_2
\)
in \(\DGr(k[d]/d^2)\).
Taking the component of module degree \(i\) yields long exact sequences in cohomology.
Hence the Euler characteristic
\(
\sum_j (-1)^j h^{i,j}(M)
\)
is additive. By \Cref{CrewFor}, the same holds for
\(
\sum_j (-1)^j h_W^{i,j}(M)=\sum_j (-1)^j h^{i,j}(M).
\)

For a coherent \(\RR\)-module \(M\), viewed as an object of \(\DGr_c^b(\RR)\),
the Hodge--Witt numbers are given by
\[
\begin{aligned}
h_W^{i,0}(M)
&=
\sum_{\lambda\in [0,1)}
(1-\lambda)m_{\lambda}(\coeur(M^i)\otimes K)
+T^{i}(M),\\
h_W^{i,-1}(M)
&=
\sum_{\lambda\in [0,1)}
\lambda m_{\lambda}(\coeur(M^{i-1})\otimes K)
-2T^{i-1}(M),\\
h_W^{i,-2}(M)
&=
T^{i-2}(M).
\end{aligned}
\]

Since the slope multiplicities are additive, it follows that the quantity
\[
T^{i}(M)+2T^{i-1}(M)+T^{i-2}(M)
\]
is additive for every \(i\in\mathbb Z\).
An induction on \(i\) then shows that each \(T^i\) is additive for short exact sequences of coherent \(\RR\)-modules.
\end{proof}

\begin{definition}
\label{numerical invariants of smooth proper}
Let \(X\) be a smooth proper variety over a perfect field \(k\) of characteristic \(p\).
We define the \emph{domino numbers} \(T^{i,j}(X)\), \emph{slope numbers} \(m^{i,j}(X)\),
\emph{Hodge--Witt numbers} \(h_W^{i,j}(X)\), and \emph{Hodge numbers} \(h^{i,j}(X)\)
of \(X\) to be the corresponding invariants (see \Cref{numerical invariants} and \Cref{Hodge number of DGrc})
of the object
$R\Gamma(X,W\Omega^\bullet_{X/k}) \in \DGr_c^b(\RR)$.
%as in \Cref{numerical invariants} and \Cref{Hodge number of DGrc}.
\end{definition}

\begin{definition}[{\cite[Definition IV.4.6]{IRdRW}}]
\label{Hodge--Witt varieties}
A smooth proper variety $X$ over $k$ is called a Hodge--Witt variety
if all of its domino numbers are $0$.
\end{definition}

\begin{theorem}[{\cite[Theorem IV.4.5]{IRdRW}}]
\label{Hodge--Witt decomposition}
Let \(X\) be a smooth proper variety over \(k\). Assume that $X$ is Hodge--Witt.
Then for each \(n\in \Z\) there is a canonical decomposition
\(
\HH_{\cris}^{n}(X/W) \cong \bigoplus_{i+j=n}\HH^j(X,W\Omega^i_{X/k}),
\)
such that the Frobenius action on the left-hand side corresponds to the action of
\(\bigoplus p^i F\) on the right-hand side.
\end{theorem}

Ekedahl also studied whether
the Hodge--Witt numbers satisfy symmetry analogous to Hodge numbers in characteristic $0$,
he obtained the following: 

\begin{proposition}[{\cite[Proposition IV.3.2, 3.3]{Ek3}}]\label{basicsymmetry}
    Let \(X/k\) be a proper smooth variety of pure dimension \(d\). 
    
    {\rm(1)} Let $n \in \N$. Then \(h_W^{i,j}=h_W^{j,i}\) for all \((i,j)\in \N^2\) such that \(i+j=n\)
    if and only if \(T^{i,j}=T^{j-2,i+2}\) for all \((i,j)\in \N^2\) such that \(i+j=n\). 
    
    {\rm(2)} The equivalent conditions in $(1)$ is satisfied whenever $n \leqslant 2$ or $d \leqslant 3$.
\end{proposition}

\medskip

\addtocontents{toc}{\protect\setcounter{tocdepth}{2}}

\section{Deeper structures on $\RR$-complexes}
\label{deeper structure}

In this section, we review finer structures on $\DGr(\RR)$ and $\DGr_c(\RR)$,
as exhibited by Ekedahl \cite{Ek1, Ek2, Ek3}.

\subsection{Duality}

The goal of this subsection is to prove \Cref{coherence criteria}.
Surprisingly, we need to use a dualizing functor introduced by Ekedahl in \cite[Chapter III]{Ek1},
which we first review below.\footnote{Note that in \cite[Definition I.6.1]{Ek2} 
Ekedahl also defined another dualizing functor, and showed that
these two dualizing functors agree on \(\DGr_c^b(\RR)\) (see \cite[Proposition III.1.6]{Ek2}).}

\begin{definition}[{\cite[Page.~189]{Ek1}}]
    For each \(n\), let \(\rho:\RR_n\rightarrow \RR_{n+1}\) be the unique injective map such that the following diagram commutes (the map \(\rho\) is unique because \(\pi\) and \(p\) have the same kernel):
    \[\begin{tikzcd}
	{\RR_{n+1}} & {\RR_n} \\
	& {\RR_{n+1}}
	\arrow["\pi", from=1-1, to=1-2]
	\arrow["p"', from=1-1, to=2-2]
	\arrow["\rho", from=1-2, to=2-2]
\end{tikzcd}\]
\end{definition}

\begin{definition}[{\cite[Page.~205]{Ek1}}]
For each $n \in \mathbb{N}$, let 
\[
\varwidecheck{\RR}_n \coloneqq \mathrm{Hom}_{W_n}(\RR_n, W_n),
\]
viewed as a graded left $\RR$-module via the graded right $\RR$-module structure of the source.
Let $\rho^* \colon \varwidecheck{\RR}_{n+1} \to \varwidecheck{\RR}_n$ be the map sending
a functional $f$ to $f \circ \rho$ (identifying the $p^n$-torsion in $W_{n+1}$ with ``$p \cdot W_n$'').
We view $\{\varwidecheck{\RR}_n, \rho^*\}$ as a projective system of graded left $\RR$-modules.
We equip it with another graded left $\RR$-module structure as follows:
\begin{itemize}
\item Define the action of
\(
F \colon \Hom_{W_{n+1}}(\RR_{n+1},W_{n+1}) \rightarrow \Hom_{W_n}(\RR_n,W_n)
\)
by the rule that for any functional $f \in \Hom_{W_{n+1}}(\RR_{n+1},W_{n+1})$
and any $r \in \RR_n$, we have the following equality in $W_{n+1}$:
\[
p \cdot F(f)(r) = \sigma(f(V \cdot r)),
\]
where $W_n \xhookrightarrow{p \cdot -} W_{n+1}$.

\item Similarly, define the action of 
\(
V \colon \Hom_{W_n}(\RR_n,W_n) \rightarrow \Hom_{W_{n+1}}(\RR_{n+1},W_{n+1})
\)
by the rule that for any functional $f \in \Hom_{W_n}(\RR_n,W_n)$
and any $r \in \RR_{n+1}$, we have the following equality in $W_{n+1}$:
\[
\sigma(V(f)(r)) = p \cdot f(F \cdot r).
\]

\item Lastly, define the action of 
\(
d \colon \varwidecheck{\RR}_n \rightarrow \varwidecheck{\RR}_n(1)
\coloneqq \mathrm{Hom}_{W_n}(\RR_n(-1), W_n)
\)
by the rule that for any functional $f \in \Hom_{W_n}(\RR_n,W_n)$
and any $r \in \RR_n$, we have the equality in $W_n$:
\[
d(f)(r) = f(d \cdot r).
\]
\end{itemize}

Since left multiplication commutes with right multiplication, one easily checks that
the above actions respect the individual graded left $\RR$-module structures.
Finally, we define
\[
\varwidecheck{\RR} \coloneqq \lim_{n \in \mathbb{N}} \varwidecheck{\RR}_n,
\]
viewed as a graded left $\RR \otimes_{\mathbb{Z}_p} \RR$-module.
We follow Ekedahl's notation: the action of the first copy of $\RR$
is induced by the right \(\RR\)-module structure on \(\RR_n\),
and the second copy acts by taking the inverse limit of the actions defined above.
\end{definition}

\begin{notation}
In order to minimize confusion, let us denote the inclusion $\RR \to \RR \otimes_{\mathbb{Z}_p} \RR$
via the first (resp.~second) factor by $i$ (resp.~$j$).
When we wish to emphasize that we view $\varwidecheck{\RR}$ as a graded left $\RR$-module via the
first (resp.~second) action of $\RR$, we denote it by $i_*(\varwidecheck{\RR})$
(resp.~$j_*(\varwidecheck{\RR})$).
\end{notation}

\begin{definition}[Ekedahl's dualizing functor {\cite[Definition III.2.8, Proposition III.3.2]{Ek1}}]
Define a contravariant functor from $\DGr(\RR)$ to itself as follows.
Given $M \in \DGr(\RR)$, let
\[
D_\RR(M) \coloneqq \mathrm{RHom}_\RR(M, i_*(\varwidecheck{\RR})).
\]
Here the left $\RR$-action comes from the second $\RR$-action on $\varwidecheck{\RR}$.
\end{definition}

In order to examine the properties of the dualizing functor, we need to explicate
$\varwidecheck{\RR}_n$ and $\varwidecheck{\RR}$.
To that end, let us consider the following \(W\)-linear graded functional
$\psi: \RR(-1) \rightarrow W$ defined by
\[
\begin{cases}
    \psi(F^n)=\psi(V^n)=0\; \text{for all}\;\; n\geqslant 0, \\
    \psi(dV^n)=\psi(F^nd)=0\; \text{for all}\;\; n\geqslant 1,\\
    \psi(d)=1.
\end{cases}
\]

Next we define a graded left $\RR$-module as follows:
\[
_n\overline{\RR} \coloneqq 
\Biggl\{\sum_{m\geqslant 0}V^m a_m +\sum_{1\leqslant m\leqslant n-1}b_m F^m 
+ \sum_{m\geqslant 0} dV^m c_m+\sum_{1\leqslant m\leqslant n-1}e_m F^md
\;\bigg|\; a_m,c_m\in W/p^n,\; b_m,e_m\in W/p^{n-m}\Biggr\}.
\]

We now define a pairing $\RR_n \times _n\overline{\RR} \to W_n$
by sending $(r, s) \mapsto$ ``$\psi(r \cdot s)$''.
Namely, we lift $r$ to an element $\widetilde{r}$ in $\RR$;
the result is the coefficient ``$c_0$'' of the element $\widetilde{r} \cdot s \in _n\overline{\RR}$.
A direct verification shows the following.

\begin{lemma}
\label{explicate checkRn}
The above construction does not depend on the choice of the lift $\widetilde{r}$ and 
induces an isomorphism of graded left $\RR$-modules
\[
\alpha_n \colon _n\overline{\RR} \xrightarrow{\cong} \varwidecheck{\RR}_n.
\]
Moreover, the map
\(
\alpha_n^{-1} \circ \rho^* \circ \alpha_{n+1} \colon _{n+1}\overline{\RR} \to _n\overline{\RR}
\) 
is the natural projection obtained by reducing each coefficient modulo a smaller power of $p$.
In particular, the transition maps $\rho^*$ are all surjective.
\end{lemma}

\begin{proposition}[{{\cite[Proposition III.3.5]{Ek1}}}]
\label{explicate checkR}
The isomorphisms $\{\alpha_n\}$ in  {\Cref{explicate checkRn}} give rise to an identification\footnote{Following
Ekedahl, we write the coefficients from $W$ in this somewhat awkward way,
so that the bijection $\beta$ from \Cref{construction of evR} has a convenient form.}
\[
\alpha \colon \big(\prod_{m \geqslant 0} V^m \cdot W \big) \oplus 
\big(\prod_{m \geqslant 1} W \cdot F^m \big) \oplus
\big(\prod_{m \geqslant 0} dV^m \cdot W \big) \oplus \big(\prod_{m \geqslant 1} W \cdot F^m d \big)
\xrightarrow{\cong} \varwidecheck{\RR}.
\]Under the above identification, the first left $\RR$-action on $\varwidecheck{\RR}$
corresponds to left multiplication, while the second $\RR$-action on $\varwidecheck{\RR}$
corresponds to right multiplication.
This defines a left $\RR^{\mathrm{op}}$-action, which can then be viewed as a left $\RR$-action via the isomorphism
\(
\iota \colon \RR \xrightarrow{\cong} \RR^{\mathrm{op}}
\)
sending \(F,d,V\) and \(a\in W\) to \(V,d,F\) and \(a\in W\), respectively.
\end{proposition}

\begin{corollary}[{\cite[Corollary III.3.5.1]{Ek1}}]
\label{Relation between checkR and checkRn}
The projection \(\varwidecheck{\RR}\rightarrow \varwidecheck{\RR}_n\)
induces an isomorphism:
%, where tensor product is taken with the second left \(\RR\)-module structure, 
\[
\RR_n\otimes^{\rmL}_\RR j_*\varwidecheck{\RR} \xrightarrow{\cong} \varwidecheck{\RR}_n
\]
of graded left $\RR$-modules. Here the action on the left hand side is given by
$r \cdot (a \otimes b) \coloneqq a \otimes (i(r) \cdot b)$.
\end{corollary}

Now we can deduce the first property of Ekedahl's dualizing functor:
\begin{proposition}
\label{dualizing commutes with tensor Rn}
There is a natural equivalence of contravariant functors 
\[
\RR_n \otimes^{\rmL}_\RR D_\RR(-) \cong D_{W_n}(\RR_n \otimes^{\rmL}_\RR -) \colon \DGr(\RR)^{\mathrm{op}} \to \DGr(W_n[d]/d^2),
\]
where $D_{W_n} \coloneqq \mathrm{RHom}_{W_n}(-, W_n)$ is the $W_n$-linear dualizing functor.
\end{proposition}

\begin{proof}
Let us compute:
\[
\RR_n \otimes^\rmL_\RR D_\RR(-) \cong \mathrm{RHom}_\RR(-, \RR_n\otimes^{\rmL}_\RR j_*\varwidecheck{\RR})
\cong \mathrm{RHom}_\RR(-, \mathrm{RHom}_{W_n}(\RR_n, W_n))
\cong \mathrm{RHom}_{W_n}(\RR_n \otimes^\rmL_\RR -, W_n).
\]
Here the first identification follows from the fact that $\RR_n$ is a perfect right $\RR$-complex,
the second identification uses \Cref{Relation between checkR and checkRn},
and the last identification is the tensor–Hom adjunction.
\end{proof}

We also know that Ekedahl's dualizing functor automatically takes values in 
$\widehat{\DGr(\RR)}$ (see \Cref{completion on DGr(R)}).

\begin{proposition}
\label{dualizing is complete}
The functor $D_\RR$ defines a contravariant functor from $\DGr(\RR)$ to
$\widehat{\DGr(\RR)}$.
\end{proposition}

\begin{proof}
By \Cref{completion is a localization}, it suffices to show that the natural
map $D_\RR(M) \to \widehat{D_\RR(M)}$ is an equivalence.
Since $\RR_n$ is a perfect right $\RR$-complex, we have
\[
R\lim_n \big(\RR_n \otimes^\rmL_\RR D_\RR(-) \big) \cong 
R\lim_n \mathrm{RHom}_\RR(-, \RR_n\otimes^{\rmL}_\RR j_*\varwidecheck{\RR})
\cong \mathrm{RHom}_\RR\!\left(-, R\lim_n \RR_n\otimes^{\rmL}_\RR j_*\varwidecheck{\RR}\right).
\]
The natural arrow is induced by the natural map
$i_*(\varwidecheck{\RR}) \to R\lim_n \RR_n\otimes^{\rmL}_\RR j_*\varwidecheck{\RR}$, which is an isomorphism
of graded left $\RR$-modules
due to \Cref{Relation between checkR and checkRn}, the surjectivity of $\rho^*$
(\Cref{explicate checkRn}), and the definition of $\varwidecheck{\RR}$.
\end{proof}

A common feature of all known dualizing functors is the existence of a natural transformation from
the identity functor to the double dualizing functor.
Our task is to exhibit such a natural transformation in the context of Ekedahl's
dualizing functor, following \cite[Section IV.1]{Ek1}.
To that end, we begin with a brief digression.

\begin{notation}
Let $\mathcal{C}$ be a stable $\infty$-category equipped with a $t$-structure.
We denote
\(
\mathcal{C}^{-} \coloneqq \bigcup_{n \in \mathbb{Z}} \mathcal{C}^{\leqslant n},
\)
and call it the full subcategory of right-bounded objects.
\end{notation}

\begin{notation}
Let $\mathcal{C}$ and $\mathcal{D}$ be stable $\infty$-categories equipped with $t$-structures.
We say that a functor $F \colon \mathcal{C} \to \mathcal{D}$ is \emph{left $t$-bounded}
if there exists an integer $N$ such that
\(
F(\mathcal{C}^{\geq 0}) \subset \mathcal{D}^{\geq N}.
\)
The notion of a \emph{right $t$-bounded} functor is defined analogously.
A functor is called \emph{$t$-bounded} if it is both left and right $t$-bounded. 
We denote $\mathrm{Fun_{ex}^{lb}}(\mathcal{C}, \mathcal{D})$ the full sub-category
of $\mathrm{Fun}(\mathcal{C}, \mathcal{D})$ spanned by left $t$-bounded exact functors.
\end{notation}

\begin{example}
\label{double dual is t-bounded}
Let $\mathcal{C} = \mathcal{D} = \DGr(\RR)$, equipped with the standard $t$-structure,
which is both left and right complete. The identity functor is clearly $t$-exact,
hence $t$-bounded. We claim that the double dual functor $D_\RR \circ D_\RR$ is also $t$-bounded:
Using \Cref{dualizing commutes with tensor Rn} and \Cref{dualizing is complete},
we obtain
\[
D_\RR(D_\RR(M)) =
\mathrm{Rlim}_n\, D_{W_n}\!\big(D_{W_n}(\RR_n \otimes^{\rmL}_\RR M)\big).
\]
Thus the claim follows from the fact that the following functors:
$(\RR_n \otimes^\rmL_\RR -)$, double dual over $W_n$, and $\mathrm{Rlim}_{n \in \mathbb{N}}$ are 
all $t$-bounded.
\end{example}

\begin{proposition}
\label{restriction to right bounded}
Let $\mathcal{C}$ and $\mathcal{D}$ be stable $\infty$-categories equipped with
$t$-structures. Assume that $\mathcal{D}$ is right complete.
Then restriction along $\mathcal{C}^{-} \subset \mathcal{C}$
induces an equivalence: 
$\mathrm{Fun_{ex}^{lb}}(\mathcal{C}, \mathcal{D}) \xrightarrow{\cong}
\mathrm{Fun_{ex}^{lb}}(\mathcal{C}^{-}, \mathcal{D})$.
\end{proposition}

\begin{proof}
First we claim that every left $t$-bounded exact functor
$F \colon \mathcal{C}^{-} \to \mathcal{D}$
admits a left Kan extension.
According to \cite[Lemma 4.3.2.13]{HTT}, it suffices to check that
for any object $C \in \mathcal{C}$ the colimit
\(
\mathrm{colim}_{X \in \mathcal{C}^{-}_{/C}} F(X)
\)
exists in $\mathcal{D}$.
Since the sequence of truncations $\{ \tau^{\leqslant n} C \}_{n \in \mathbb{Z}}$
is cofinal in $\mathcal{C}^{-}_{/C}$, it suffices to show that the colimit
\(
\mathrm{colim}_{n \in \mathbb{Z}} F(\tau^{\leqslant n}C)
\)
exists in $\mathcal{D}$.
This follows from the assumption that $F$ is left $t$-bounded and exact,
together with the right completeness of $\mathcal{D}$.

One checks directly that the left Kan extension
\(
\mathrm{Lan}(F)(C) \coloneqq
\mathrm{colim}_{n \in \mathbb{Z}} F(\tau^{\leqslant n}C)
\)
is again left $t$-bounded and exact.
Therefore, by \cite[Proposition 4.3.2.15]{HTT}, we obtain a functor \[\mathrm{Fun_{ex}^{lb}}(\mathcal{C}^{-}, \mathcal{D}) \xrightarrow{\mathrm{Lan}(-)}
\mathrm{Fun_{ex}^{lb}}(\mathcal{C}, \mathcal{D})\] 
sending $F$ to its left Kan extension.

Similar in the proof of \cite[Proposition 4.3.2.17]{HTT}, using
\cite[Lemma 4.3.2.12]{HTT}, one sees that $\mathrm{Lan}(-)$ is left adjoint
to the restriction functor.
Moreover, since the restriction of left Kan extension is identity,
the functor $\mathrm{Lan}(-)$ is fully faithful.

Finally, it suffices to show that $\mathrm{Lan}(-)$ is essentially surjective.
This amounts to proving that for every left $t$-bounded exact functor
$G \colon \mathcal{C} \to \mathcal{D}$, the natural map
\(
\mathrm{colim}_{n \in \mathbb{Z}} G(\tau^{\leqslant n}C)
\longrightarrow G(C)
\)
is an equivalence.
This again follows from the assumption that $G$ is left $t$-bounded and exact
and that $\mathcal{D}$ is right complete.
\end{proof}

\begin{proposition}
\label{restriction to proj mods}
Let $\mathcal{D}$ be a stable $\infty$-category equipped with a $t$-structure
that is both left and right complete.
Let $K$ be a directed graph, and let $\{F_k\}$ be a collection of left $t$-bounded,
uniformly right $t$-bounded, exact functors
from $\DGr(\RR)$ to $\mathcal{D}$ indexed by the vertices of $K$. 
For any full subcategory $\mathcal{C} \subset \DGr(\RR)$,
we denote by $\mathrm{Fun}_{F}(K \times \mathcal{C}, \mathcal{D})$ the fiber of
$\{F_k|_{\mathcal{C}}\}$ under the restriction to vertices map. 
Let $\mathcal{P}roj\mathcal{G}r(\RR) \subset \DGr(\RR)$ denote the full subcategory of 
projective graded left $\RR$-modules.
Then the restriction map induces an equivalence
\[
\mathrm{Fun}_{F}(K \times \DGr(\RR), \mathcal{D})
\xrightarrow{\;\cong\;}
\mathrm{Fun}_{F}(K \times \mathcal{P}roj\mathcal{G}r(\RR), \mathcal{D}).
\]
\end{proposition}

\begin{proof}
By \Cref{restriction to right bounded} and the assumption that the functors $\{F_k\}$
are all left $t$-bounded and exact, the restriction map induces an equivalence
\(
\mathrm{Fun}_{F}(K \times \DGr(\RR), \mathcal{D})
\xrightarrow{\;\cong\;}
\mathrm{Fun}_{F}(K \times \DGr^{-}(\RR), \mathcal{D}).
\)

Since the collection of functors $\{F_k\}$ is uniformly right $t$-bounded, after a finite shift
we may assume without loss of generality that all of them are right $t$-exact
(this does not affect the condition that they are left $t$-bounded).
As $\DGr^{-}(\RR)$ is left complete and right bounded, and $\mathcal{D}$ is assumed to be left complete,
the statement now follows from the combination of
\cite[Theorem 1.3.3.8 and Lemma 1.3.3.11]{HA}.
\end{proof}

\begin{lemma}
\label{coh estimate of dual and double dual of projective mod}
For any projective graded left $\RR$-module $M$, the dual $D_\RR(M)$ lies in $\DGr^{[0,0]}(\RR)$,
and the double dual $D_\RR(D_\RR(M))$ lies in $\DGr^{[0,1]}(\RR)$.\footnote{In fact,
with a more involved argument, one can show that $D_\RR(D_\RR(M))$ 
actually lives in cohomological degree $0$. However, since the present cohomological estimate
is sufficient for our purposes, we do not pursue this stronger statement.}
\end{lemma}

\begin{proof}
Such an $M$ is a direct summand of a graded free\footnote{By this we mean a module of the form
$M = \bigoplus_{i \in I} \RR(n_i)$ for some set $I$.} left $\RR$-module,
so we may assume that $M$ is graded free.
Writing $M$ as a filtered colimit of finite direct sums, the dual $D_\RR(M)$ becomes
a cofiltered limit of finite products of $\varwidecheck{\RR}(-n_i)$,
and hence lies in cohomological degree $0$. 
Using \Cref{dualizing commutes with tensor Rn} and \Cref{dualizing is complete},
we obtain
\[
D_\RR(D_\RR(M))
=
\mathrm{Rlim}_n\, D_{W_n}\!\big(D_{W_n}(\RR_n \otimes^\rmL_\RR M)\big).
\]
The claim now follows from the facts that $\RR_n \otimes^\rmL_\RR M$ lies in cohomological degree $0$,
the $W_n$-linear double dual is $t$-exact,
and the functor $\mathrm{Rlim}_{n \in \mathbb{N}}$ has cohomological amplitude contained in $[0,1]$.
\end{proof}

\begin{construction}[Natural map to the double dual]
\label{construction of evR}
We now construct a natural transformation
\[
\id \xrightarrow{\mathrm{ev}_\RR} D_\RR \circ D_\RR .
\]
In view of \Cref{double dual is t-bounded} and \Cref{restriction to proj mods},
it suffices to construct a natural transformation $\id \to D_\RR \circ D_\RR$
between functors from $\mathcal{P}roj\mathcal{G}r(\RR)$ to $\DGr(\RR)$. By \Cref{coh estimate of dual and double dual of projective mod},
it suffices to construct natural maps of graded left $\RR$-modules
\[
M \to \text{``}j_*\text{''}\mathrm{Hom}_\RR(\text{``}j_*\text{''}\mathrm{Hom}_\RR(M, i_*(\varwidecheck{\RR})), 
i_*(\varwidecheck{\RR}))
\]
for projective graded left $\RR$-modules. Indeed, maps from $M$ to an object in
$\DGr^{[0,n]}(\RR)$ automatically factor through its $\mathrm{H}^0$.
Here we write ``$j_*$'' to remind the reader that the $\RR$-module structure arises
from the second action on the target $\varwidecheck{\RR}$. 
Unwinding the definition, this is equivalent to constructing natural maps of graded left $\RR$-modules
\[
\text{``}j_*\text{''}\mathrm{Hom}_\RR(M, i_*(\varwidecheck{\RR}))
\longrightarrow
\text{``}i_*\text{''}\mathrm{Hom}_\RR(M, j_*(\varwidecheck{\RR})).
\]
Finally, this is achieved by exhibiting a bijection
\(
\varwidecheck{\RR} \xrightarrow[\cong]{\beta} \varwidecheck{\RR}
\)
that swaps the $\RR$-actions via $i$ and $j$.
Using the isomorphism $\alpha$ from \Cref{explicate checkR}, this map is given by
\[
\sum_{m\geqslant 0} V^m a_m
+\sum_{m\geqslant 1} b_m F^m
+\sum_{m\geqslant 0} dV^m c_m
+\sum_{m\geqslant 1} e_m F^m d
\mapsto
\sum_{m\geqslant 1} V^m b_m
+\sum_{m\geqslant 0} a_m F^m
+\sum_{m\geqslant 1} dV^m e_m
+\sum_{m\geqslant 0} c_m F^m d.
\footnote{See \cite[Proposition III.3.5]{Ek1}.}
\]
Concretely, the map
\(
M \to
\text{``}j_*\text{''}\mathrm{Hom}_\RR(\text{``}j_*\text{''}\mathrm{Hom}_\RR(M,i_*(\varwidecheck{\RR})),
i_*(\varwidecheck{\RR}))
\)
sends an element $m \in M$ to the functional that associates to
$f \in \mathrm{Hom}_\RR(M,i_*(\varwidecheck{\RR}))$
the element $\beta(f(m)) \in \varwidecheck{\RR}$.
\end{construction}

After defining the natural transformation, our next task is to show its compatibility with the
$W_n$-linear dualizing functor via $\RR_n \otimes^\rmL_\RR -$.

\begin{proposition}[{\cite[Lemma IV.1.5]{Ek1}}]
\label{Rn and Wn double dual compatibility}
There is a commutative diagram of functors from $\DGr(\RR)$ to $\DGr(W_n[d]/d^2)$:
\[
\begin{tikzcd}
	{\RR_n\otimes^{\rmL}_\RR \id} && {\RR_n\otimes^{\rmL}_\RR (D_\RR(D_\RR(-))} \\
	&& {D_{W_n}(\RR_n\otimes^{\rmL}_\RR D_\RR(-))} \\
	{\RR_n\otimes^{\rmL}_\RR (-)} && {D_{W_n}(D_{W_n}(\RR_n\otimes^{\rmL}_\RR (-)))}.
	\arrow["{\RR_n\otimes^{\rmL}_\RR(\mathrm{ev}_\RR)}", from=1-1, to=1-3]
	\arrow[shift left, no head, from=1-1, to=3-1]
	\arrow[shift right, no head, from=1-1, to=3-1]
	\arrow["\cong", from=1-3, to=2-3]
	\arrow["\cong", from=2-3, to=3-3]
	\arrow["{\mathrm{ev}_{W_n}}", from=3-1, to=3-3]
\end{tikzcd}
\]
Here $\mathrm{ev}_{W_n} \colon \id \to D_{W_n} \circ D_{W_n}$ denotes the usual
natural transformation between endofunctors of $\DGr(W_n[d]/d^2)$,
and the identifications in the right column come from
\Cref{dualizing commutes with tensor Rn}.
\end{proposition}

\begin{proof}
By \Cref{Rn is right perfect} and \Cref{double dual is t-bounded}, all
functors involved are $t$-bounded.
Using \Cref{restriction to proj mods}, it therefore suffices to exhibit such a commutative
diagram with all functors viewed as defined only on projective graded left $\RR$-modules. 
The values of all functors on projective graded left $\RR$-modules are concentrated in
cohomological degree $0$: 
For the right column, it suffices to note that $D_{W_n} \circ D_{W_n}$ is $t$-exact.
Thus we are merely verifying a compatibility condition rather than specifying additional data. 
Finally, the diagram commutes as explained by Ekedahl in \emph{loc.\ cit.}.
\end{proof}

\begin{remark}
For the convenience of the reader, we briefly explain the commutativity above.
First, we may reduce to the case where $M$ is graded free.
By writing $M$ as a colimit of graded finite free modules, we may further reduce
to the case where $M$ is a graded shift of $\RR$.
Finally, since every arrow respects the graded structure, we may reduce to the case
$M = \RR$. In this case, after unwinding the definitions, the commutativity claim follows
from the formula
\(
\psi(r \cdot s) = \psi(\beta(s) \cdot \iota(r)),
\) 
where $r \in \RR$, $s$ is a power series as in \Cref{explicate checkR},
$\iota$ is the isomorphism $\RR \cong \RR^{\mathrm{op}}$ described in
\Cref{explicate checkR}, $\beta$ is the bijection introduced in
\Cref{construction of evR}, and $\psi$ is the linear functional given by
taking the coefficient $c_0$ in the series expansion in
\Cref{explicate checkR}.
\end{remark}

From the above discussion, Ekedahl draws the following conclusion.

\begin{proposition}[{\cite[Proposition IV.1.1]{Ek1}}]
\label{dual of coherent is coherent}
The dualizing functor \(D\) maps \(\DGr_c^b(\RR)\) to itself, and the evaluation
transformation \(\mathrm{ev}:\id \to D(D(-))\) is an isomorphism on
\(\DGr_c^b(\RR)\).
\end{proposition}

We now generalize Ekedahl's result.
Let $n$ be a positive integer.
Denote by $\DGr^{\RR_n\text{-finite}}(\RR)$ the full subcategory of $\DGr(\RR)$
spanned by objects $M$ such that
$\RR_n\otimes^{\rmL}_\RR M \in \DGr^-(W_n)$ 
obtained by forgetting grading and the action of $d$ has finitely generated cohomology. 

\begin{corollary}
\label{anti-equivalence on Rn-finite}
Ekedahl's dualizing functor $D_\RR$ preserves $\DGr^{\RR_n\text{-}\mathrm{finite}}(\RR)$.
Moreover, if $M \in \DGr^{\RR_n\text{-}\mathrm{finite}}(\RR)$, then 
\(
\RR_n \otimes^\rmL_\RR \mathrm{ev}_\RR :
\RR_n \otimes^\rmL_\RR M
\longrightarrow
\RR_n \otimes^\rmL_\RR D_\RR(D_\RR(M))
\)
is an equivalence.
In particular, Ekedahl's dualizing functor $D_\RR$ restricts to an anti-equivalence
on \(\widehat{\DGr}(\RR) \cap \bigcap_{n \geq 1} \DGr^{\RR_n\text{-}\mathrm{finite}}(\RR).\)
\end{corollary}

\begin{proof}
The first claim follows immediately from \Cref{dualizing commutes with tensor Rn}.
The second claim follows immediately from \Cref{Rn and Wn double dual compatibility}. 
For the last claim, it suffices to check that for
\(
M \in \widehat{\DGr}(\RR) \cap \bigcap_{n \geq 1} \DGr^{\RR_n\text{-finite}}(\RR),
\)
the map
\(
M \xrightarrow{\mathrm{ev}_\RR} D_\RR(D_\RR(M))
\)
is an equivalence.
By the assumption that $M$ is complete and by \Cref{dualizing is complete},
we are reduced to the second claim.
\end{proof}

\begin{corollary}
\label{coherence criteria}
Let $M \in \DGr(\RR)$ be a complete object.
Assume that $M$ is bounded either from above or from below.
Then the following conditions are equivalent:

{\rm (1)} $M \in \DGr_c(\RR)$;

{\rm (2)} for all $n$, the object $\RR_n\otimes^{\rmL}_\RR M \in \DGr(W_n)$,
obtained by forgetting the grading and the action of $d$, has finitely generated cohomology;

{\rm (3)} the object $\RR_1\otimes^{\rmL}_\RR M \in \DGr(k)$,
obtained by forgetting the grading and the action of $d$, has finite-dimensional cohomology.

Moreover, the functor $D_\RR$ induces an anti-equivalence between $\DGr^+_c(\RR)$ and
$\DGr^-_c(\RR)$.
\end{corollary}

\begin{proof}
When $M$ is bounded from above, this is \Cref{alternative definition of coherence}.
From now on, we assume that $M$ is bounded from below.
It is clear that {\rm (1)} implies {\rm (2)}, and that {\rm (2)} implies {\rm (3)}.
Below we show that {\rm (3)} implies {\rm (1)}.

First, similarly to the reasoning in \Cref{double dual is t-bounded}, we note that
$D_\RR$ is also $t$-bounded.
Indeed, examining the argument there shows that
$D_\RR(M) \in \DGr^{[0,3]}(\RR)$ for any $M \in \DGr^{[0,0]}(\RR)$.

Using the first statement in \Cref{anti-equivalence on Rn-finite} and the case where $M$
is bounded above, we deduce that $D_\RR(M) \in \DGr_c^-(\RR)$.
Then \Cref{dual of coherent is coherent}, together with the $t$-boundedness of $D_\RR$,
implies that $D_\RR(D_\RR(M)) \in \DGr_c^+(\RR)$.
The desired conclusion now follows from \Cref{anti-equivalence on Rn-finite}.
\end{proof}

\subsection{Ekedahl's star product}

The goal of this subsection is to review Ekedahl's star product. 
%, \textcolor{red}{and to calculate \(E_{1/2} \cdstimes k\)}. 
In \cite{Ek2}, Ekedahl defined the star product 
\(\stimes\) which ``corresponds to the correct K\"unneth
formula'' for the de Rham--Witt cohomology \(R\Gamma(X,W\Omega^\bullet_{X/k})\).
Below we review his theory.
%We now review its theory. 

\begin{construction}[{\cite[Definition 1.3.1]{Ek2}}]
\label{star product abelian level}
Let \(M\) and \(N\) be two left graded \(\RR\)-modules, the star product \(M\stimes N\) is 
defined to be the left graded 
\(\RR\)-module generated by homogeneous elements of the form \(m\stimes n\) 
where $m$ and $n$ range over
all homogeneous elements of $M$ and $N$ respectively
(with $m \stimes n$ of grading \(\mathrm{deg}(m)+\mathrm{deg}(n)\)), subject to 
the following relations (where $\lambda$ ranges over $W(k)$):
    \[\begin{aligned}
        (Vm)\stimes n=V(m\stimes Fn),&\quad m\stimes (Vn)=V(Fm \stimes n),\\
        F(m\stimes n)=&(Fm)\stimes (Fn),\\
        d(m\stimes n)=(dm)\stimes n&+(-1)^{\mathrm{deg}(m)}m\stimes (dn),\\
        (m_1+m_2)\stimes n=m_1\stimes n+m_2&\stimes n,\quad (\lambda m)\stimes n=\lambda(m\stimes n),\\
        m\stimes (n_1+n_2)=m\stimes n_1+m&\stimes n_2,\quad 
        m\stimes (\lambda n)=\lambda(m\stimes n).
    \end{aligned}\]
This construction comes equipped with canonical isomorphisms 
$M \stimes N \xrightarrow{\cong} N \stimes M$ sending homogeneous elements
$m \stimes n$ to $(-1)^{\deg(m) \cdot \deg(n)} n \stimes m$ (\cite[Page.~71, Equation (3.4)]{Ek2}).
Moreover, let $W = W(k)$ be the graded left $\RR$-module placed in grading $0$,
with $F$ and $V$ acting as usual Witt vector Frobenius and Verschiebung, and $d = 0$.
Then one checks that $M \xrightarrow{m \mapsto m \stimes 1} M \stimes W$ is an isomorphism.
In this way we obtain a symmetric monoidal structure on
the category of graded left $\RR$-modules.
\end{construction}

By definition, it is clear that $- \stimes M$ is a right-exact functor
for any graded left $\RR$-module $M$.
Here are some properties of this symmetric monoidal structure.

\begin{example}[{\cite[Proposition 3.2]{Ek2}}]
\label{stimesexamples}
The graded left $\RR$-module
\(\RR\stimes \RR\) is a graded 
free module generated by the following two sets of homogeneous elements: 
\[\begin{aligned}
&\mathrm{grading}\;0: F^i\stimes 1,\; i\geqslant 0;\quad 1\stimes F^i,\;i\geqslant 1;\\
&\mathrm{grading}\;1: F^id\stimes 1,\; i\geqslant 0;\quad 1\stimes F^id,\;i\geqslant 1.
\end{aligned}\]
\end{example}

\begin{corollary}[{\cite[Corollary 3.2.3]{Ek2}}]
\label{R is stimes-flat}
For any left graded \(\RR\)-module $M$,
the underlying graded $W$-module (obtained by  forgetting the actions of $F$, $V$ and $d$)
\(\RR\stimes M\) decomposes as: 
\[
\RR\stimes M \cong \Bigl(\bigoplus_{i>0} V^i(1\stimes M)\Bigr) \oplus 
\Bigl(\bigoplus_{i\geqslant 0} F^i\stimes M\Bigr) \oplus 
\Bigl(\bigoplus_{i>0}dV^i(1\stimes M)\Bigr) \oplus 
\Bigl(\bigoplus_{i\geqslant 0}F^id\stimes M\Bigr).
\]
In particular, the functor $\RR \stimes -$ is exact.
\end{corollary}

\begin{proposition}[{\cite[Proposition 3.3]{Ek2}}]
\label{stimes Frob bijection}
Let \(M\) and \(N\) be left graded \(\RR\)-modules. Suppose that \(F\) is bijective on \(N\). Then the natural map
$M \otimes_W N \to M \stimes N$, sending $m \otimes n$ to $m \stimes n$
for homogeneous elements $m \in M$ and $n \in N$, is an isomorphism.
\end{proposition}

\begin{proposition}[{\cite[Lemma 4.6]{Ek2}}]
\label{tensor R1 symmetric monoidal}
Let \(M\) and \(N\) be left graded \(\RR\)-modules. Then the natural map
$M \otimes_W N \to M \stimes N$, sending $m \otimes n$ to $m \stimes n$
for homogeneous elements $m \in M$ and $n \in N$, induces an isomorphism
$(\RR_1 \otimes_\RR M) \otimes_k (\RR_1 \otimes_\RR N) \xrightarrow{\cong} \RR_1 \otimes_\RR (M \stimes N)$.
\end{proposition}

Our next goal is to show that the symmetric monoidal structure on graded left \(\RR\)-modules
extends to the whole category \(\DGr(\RR)\).
In \cite[Proposition 4.5]{Ek2}, Ekedahl obtained such a structure only at the level of the homotopy category,
and one of the two inputs had to be bounded above.
For our purposes, it is more convenient to construct such a symmetric monoidal structure
at the $\infty$-categorical level and remove the bounded-above restriction. 

Recall that for a collection $K$ of simplicial sets, Lurie \cite[Definition 4.8.1.1]{HA} 
uses $\mathrm{Cat}_{\infty}(K)$ to denote the subcategory of $\mathrm{Cat}_{\infty}$
spanned by $\infty$-categories admitting colimits indexed by elements of $K$,
and by functors preserving such colimits.
Moreover, according to \cite[Proposition 4.8.1.3]{HA}, each 
$\mathrm{Cat}_{\infty}(K)$ is a symmetric monoidal $\infty$-category. 
Furthermore, following the proof of \emph{loc.\ cit.}, one sees that if $\mathcal{C}$ is a $1$-category admitting
colimits indexed by elements of $K$, then promoting $\mathcal{C}$ to a (non-unital)
commutative algebra object in $\mathrm{Cat}_{\infty}(K)$ is equivalent to equipping
$\mathcal{C}$ with a (non-unital) symmetric monoidal structure that preserves 
colimits indexed by elements of $K$ in each variable. 

\begin{construction}
\label{star product free level}
Let $K$ be the collection of all sets, viewed as a collection of simplicial sets
via their nerves, so that colimits indexed by elements of $K$ simply mean direct sums.
By \Cref{stimesexamples} and the fact that $- \stimes -$ respects
direct sums in one variable, we see that $M \stimes N$ is a graded free left
$\RR$-module whenever both $M$ and $N$ are.
Thus \Cref{star product abelian level} gives rise to an object
$\mathcal{GM}od^{\mathrm{free}}(\RR)$ in
$\mathrm{CAlg}^{\mathrm{nu}}(\mathrm{Cat}_{\infty}(K))$
whose underlying object is the category of graded free left $\RR$-modules.
Here the superscript ``nu'' stands for non-unital\footnote{The unit for Ekedahl's product is \emph{not}
a graded free $\RR$-module.}
(see \cite[Definition 5.4.4.1]{HA}). 
Moreover, the functor
\(
\RR_1 \otimes_\RR - \colon
\mathcal{GM}od^{\mathrm{free}}(\RR)
\to
\mathcal{GM}od^{\mathrm{free}}(k)
\)
is a morphism in
$\mathrm{CAlg}^{\mathrm{nu}}(\mathrm{Cat}_{\infty}(K))$,
thanks to \Cref{tensor R1 symmetric monoidal}.
\end{construction}

Next, we construct a symmetric monoidal structure on the connective part of
$\DGr(\RR)$. By \cite[Remark 4.8.1.8]{HA}, for any inclusion of collections
$K \subset K'$ of simplicial sets, the functor
$\mathcal{P}^{K'}_K$ of ``freely adjoining colimits indexed by elements of $K'$
while preserving those indexed by elements of $K$'' is symmetric monoidal. 

\begin{construction}
\label{star product connective level}
Now let $K'$ denote the collection of all simplicial sets.
Using the fact that $\mathcal{P}^{K'}_K$ is symmetric monoidal, and the non-unital
commutative algebra structure on $\mathcal{GM}od^{\mathrm{free}}(\RR)$ from \Cref{star product free level},
we see that the object 
$\mathcal{P}^{K'}_K(\mathcal{GM}od^{\mathrm{free}}(\RR)) \cong \DGr^{\leqslant 0}(\RR)$ 
also admits a non-unital symmetric monoidal structure.
Here $\DGr^{\leqslant 0}(\RR)$ denotes the connective part of $\DGr(\RR)$.
\end{construction}

Let us explicate the induced symmetric product in the above construction.

\begin{remark}
\label{explicating star product connective level}
Recall that by Dold--Kan correspondence, each object in $\DGr^{\leqslant 0}(\RR)$
can be represented by a simplicial colimit of graded free left $\RR$-modules.
According to the proof of \cite[Proposition 4.8.1.3]{HA},
the induced non-unital symmetric monoidal structure preserves arbitrary colimits in each variable
and restricts to the given $- \stimes -$ when both inputs are graded free.
Therefore, at the homotopy level, the induced non-unital symmetric monoidal structure
on $\DGr^{\leqslant 0}(\RR)$ is the left derived functor $- \stimes^{\rmL} -$ of $- \stimes -$.
In particular, $\pi_0(M \stimes^{\rmL} N) = \pi_0(M) \stimes \pi_0(N)$.
\end{remark}

Recall \cite[Theorem 5.4.4.5]{HA}, which states that
a non-unital symmetric monoidal structure uniquely (up to contractible choices)
arises from a symmetric monoidal structure if and only if it admits a homotopy unit.

\begin{proposition}
\label{unital star product connective level}
The non-unital symmetric monoidal structure in \Cref{star product connective level}
admits a unit; therefore it defines a symmetric monoidal structure on $\DGr^{\leqslant 0}(\RR)$.
\end{proposition}

\begin{proof}
According to \Cref{explicating star product connective level}, the symmetric product
$- \stimes^{\rmL} -$ is the derived functor of $- \stimes -$.
By \Cref{R is stimes-flat}, we see that one only has to resolve one of the variables
by graded free modules when computing $- \stimes^{\rmL} -$.
Therefore, by \Cref{stimes Frob bijection}, we see that $W$ is a homotopy unit
of $- \stimes^{\rmL} -$.
\end{proof}

%It preserves arbitrary colimit in each variable, has a unit
%which at homotopy level is given by left derived functor of Ekedahl's $- \stimes -$.
\begin{remark}
\label{derived tensor R1 is symmetric monoidal connective level}
By the last statement of \Cref{star product free level}, we see that
$\RR_1 \otimes^{\rmL}_\RR (-) = \mathcal{P}^{K'}_K(\RR_1 \otimes_\RR -)$ is a symmetric monoidal functor
from $\DGr^{\leqslant 0}(\RR)$ to $\DGr^{\leqslant 0}(k)$.
\end{remark}

Lastly, we obtain an induced symmetric monoidal structure on the whole of $\DGr(\RR)$
by tensoring with the category of spectra $\mathrm{Sp}$, as follows:

\begin{construction}
\label{star product derived level}
Observe that $\DGr^{\leqslant 0}(\RR)$ is presentable; therefore
\Cref{star product connective level} promotes it to an object of
$\mathrm{CAlg}(Pr^{\rmL})$.
Since the category $\mathrm{Sp}$ of spectra is also an object of
$\mathrm{CAlg}(Pr^{\rmL})$, we obtain an object
\(
\DGr^{\leqslant 0}(\RR) \otimes \mathrm{Sp}
\in
\mathrm{CAlg}(Pr^{\rmL}).
\)
According to \cite[Example 4.8.1.23]{HA}, the underlying category is simply
\(
\mathrm{Sp}(\DGr^{\leqslant 0}(\RR)) = \DGr(\RR)
\)
(see \cite[Remark C.1.2.10.(b)]{SAG}).
Therefore we obtain a symmetric monoidal structure on $\DGr(\RR)$.
We continue to denote the induced symmetric monoidal product by $-\stimes^{\rmL}-$.
\end{construction}

\begin{remark}
Let us summarize the properties of $- \stimes^{\rmL} -$
obtained in the above construction:
\begin{enumerate}
\item It preserves arbitrary colimits in each variable:
This is a general feature of 
$\mathrm{CAlg}(Pr^{\rmL})$;
\item It also preserves finite limits in each variable:
This is because finite limits are always shifts of finite colimits in a stable $\infty$-category.
\end{enumerate}
Combining these two properties, we see that for any pair of objects $M, N \in \DGr(\RR)$, we may compute
$M \stimes^{\rmL} N = \mathrm{colim}_{(m, n) \in \mathbb{Z}^2} \tau^{\leqslant m}M \stimes^{\rmL} \tau^{\leqslant n} N$;
and for each $\tau^{\leqslant m}M \stimes^{\rmL} \tau^{\leqslant n} N$, we can compute by resolving
one of them by graded free left $\RR$-modules.
In particular, our $- \stimes^{\rmL} -$ extends Ekedahl's symmetric monoidal product
at the homotopy level.
\end{remark}

\begin{remark}
\label{derived tensor R1 is symmetric monoidal}
Taking \Cref{derived tensor R1 is symmetric monoidal connective level} and passing to
$- \otimes \mathrm{Sp}$, we also see that
$\RR_1 \otimes^{\rmL}_\RR -$ is a symmetric monoidal functor from $\DGr(\RR)$ to $\DGr(k)$.
\end{remark}

Since the cohomology of de Rham--Witt complexes is always complete, it is desirable
to have a symmetric monoidal structure on $\widehat{\DGr}(\RR)$.

\begin{proposition}
\label{star product complete level}
There is a unique way to endow $\widehat{\DGr}(\RR)$ with a symmetric monoidal structure
such that the completion functor $\widehat{(-)} \colon \DGr(\RR) \to \widehat{\DGr}(\RR)$ is symmetric monoidal.
\end{proposition}

We shall denote the symmetric monoidal product on $\widehat{\DGr}(\RR)$ by $- \cdstimes -$.

\begin{proof}
Let us mimic the proof of \cite[Theorem 6.2]{Condensed}. The requirement that
$\widehat{(-)}$ be symmetric monoidal forces $- \cdstimes - \coloneqq \widehat{(- \stimes^{\rmL} -)}$.
Similar to loc.~cit., all we need to check is that for any $M, N \in \DGr(\RR)$, the natural map
$M \stimes^{\rmL} N \to \widehat{M} \stimes^\rmL N$ becomes an isomorphism after completion.
Since both functors commute with colimits in $N$, we may first assume that $N$ is bounded above (so that it is 
represented by a bounded above complex of graded free modules).
We may then further assume that $N$ is a graded free left $\RR$-module concentrated in cohomological degree $0$
(by considering the stupid filtration on the representing complex).
%Finally, we may assume that $N = \RR(i)$ by commuting with filtered colimits and finite direct sums.
In particular, at this point $- \stimes^{\rmL} N$ is exact.
Since completion is $t$-bounded, we can repeat the same reduction process for $M$. 
As a result, we may also assume that $M$ is a graded free left $\RR$-module concentrated in cohomological degree $0$.
Therefore, it suffices to prove that
$\widehat{M \stimes N} \to \widehat{\widehat{M} \stimes N}$\footnote{Notice that we dropped the ``$\rmL$'' for both star products,
since $N$ is graded free.}
is an isomorphism. Since both sides are complete and (cohomologically) bounded, 
by \Cref{Ek2I.1.1}, it suffices to show that the map becomes an isomorphism
after applying $\RR_1 \otimes^\rmL_\RR -$.
Using \Cref{tensor Rn and completion}, after applying $\RR_1 \otimes^\rmL -$, the map becomes
$\RR_1 \otimes^\rmL_\RR (M \stimes N) \to \RR_1 \otimes^\rmL_\RR (\widehat{M} \stimes^\rmL N)$.
Using \Cref{derived tensor R1 is symmetric monoidal}, the map becomes
\[
(\RR_1 \otimes^\rmL_\RR M) \otimes_k (\RR_1 \otimes^\rmL_\RR N)
\xrightarrow{(\RR_1 \otimes^\rmL \eta) \otimes \id} (\RR_1 \otimes^\rmL_\RR \widehat{M}) \otimes^\rmL_k (\RR_1 \otimes^\rmL_\RR N).
\]
Using \Cref{tensor Rn and completion} again, we see that the map on the first factor is an isomorphism.
\end{proof}

\begin{theorem}[c.f.~{\cite[Theorem II.1.1]{Ek2}}]
\label{K\"unneth}
Let \(X, Y\) be smooth varieties over \(k\). Assume that one of the following holds:
\begin{enumerate}
\item Both $X$ and $Y$ are quasi-compact; or
\item The Hodge cohomology of either $X$ or $Y$ is finite-dimensional.
\end{enumerate}
Then there is a natural isomorphism:
\[
R\Gamma(X,W\Omega^\bullet_{X/k})\cdstimes R\Gamma(Y,W\Omega^\bullet_{Y/k})
\xrightarrow{\cong} R\Gamma(X \times Y, W\Omega^\bullet_{X \times Y/k}).
\]
\end{theorem}

Notice that the original statement in Ekedahl's paper seems to be incorrect: He did not assume
any finiteness condition, with the simplest counterexample being taking both $X = Y = \bigsqcup_{\mathbb{Z}} \mathrm{Spec}(k)$.
However, with the technical assumption added above, Ekedahl's proof holds.
We shall prove the above statement in greater generality, allowing $X$ and $Y$ to be smooth algebraic
stacks (\Cref{KF for stacks}), so let us postpone its proof until there. 

The remainder of this subsection is devoted to some concrete computations that will be
useful for later purposes.

\begin{notation}
We write $\mathbb{D}(\alpha_p)$ for the graded left $\RR$-module
given by $k$ placed in degree $0$, on which the operators
$F$, $V$, and $d$ act trivially. 
For coprime integers $i\geqslant 1,j \geqslant 0$, we denote by $E_{j/i+j}$ the graded left
$\RR$-module
\(
E_{j/i+j} := \RR^0 / \RR^0(F^i - V^j),
\)
concentrated in degree $0$.
%\footnote{Note that $E_{1/2}$
%was already defined in \Cref{The de Rham-Witt cohomology of curves}, but there is no conflict of notation
%as the first crystalline cohomology of a supersingular elliptic curve over an algebraically closed field $k$
%is indeed given by $E_{1/2}$.}
Equivalently, $E_{j/i+j}$ is the free $W$-module generated by
\[
\{f_i, f_{i-1}, \ldots, f_1, e_0, e_1, \ldots, e_{j-1}\},
\]
with $d = 0$, and the semi-linear operators $F$ and $V$ given by
\[
\begin{aligned}
Ff_s &= f_{s+1} && (0 \le s \le i-1), &
Fe_t &= p\,e_{t-1} && (1 \le t \le j),\\
Vf_s &= p\,f_{s-1} && (1 \le s \le i), &
Ve_t &= e_{t+1} && (0 \le t \le j-1).
\end{aligned}
\]
Here we use the conventions $f_0 = e_0$ and $e_j = f_i$.
\end{notation}

Observe that $F$ commutes with $F^i - V^j$, so right multiplication by $F$ on $\RR^0$ descends 
to a well-defined map of graded left $\RR$-modules:
\(
E_{j/i+j} \xrightarrow{\cdot F} E_{j/i+j}.
\)
For \(E_{1/2}\), one verifies that this map fits into a short exact sequence
\(
0 \longrightarrow E_{1/2} \xrightarrow{\cdot F} E_{1/2}
\longrightarrow \mathbb{D}(\alpha_p) \longrightarrow 0 .
\)

\begin{lemma}
\label{decompose wdhR cdstimes k}
There is a natural decomposition
\[
\widehat{\RR} \cdstimes \mathbb{D}(\alpha_p) \cong
\Bigl(\prod_{i > 0} V^i(1 \stimes \mathbb{D}(\alpha_p))\Bigr) \oplus 
\Bigl(\bigoplus_{i \geqslant 0} F^i \stimes \mathbb{D}(\alpha_p)\Bigr) \oplus 
\Bigl(\prod_{i > 0} dV^i(1 \stimes \mathbb{D}(\alpha_p))\Bigr) \oplus 
\Bigl(\bigoplus_{i \geqslant 0} F^i d \stimes \mathbb{D}(\alpha_p)\Bigr)
\]
compatible with the decomposition in  {\Cref{R is stimes-flat}}.\footnote{This compatibility,
together with the requirement that the action of $F$, $V$, and $d$ be continuous with respect
to the topology from \Cref{complete R-module},
automatically determines the graded left $\RR$-module structure on the right-hand side.}
In particular, the object $\widehat{\RR} \cdstimes \mathbb{D}(\alpha_p) \in \widehat{\DGr}(\RR)$ is concentrated 
in cohomological degree $0$.
\end{lemma}

\begin{proof}
By \Cref{star product complete level}, we need to apply the completion functor $\widehat{(-)}$
to the graded left $\RR$-module $\RR \stimes \mathbb{D}(\alpha_p)$.
In the decomposition displayed in \Cref{R is stimes-flat}, we note that the action
of $V$ annihilates the latter three summands and that the action of $p$ is $0$.
Therefore, by the resolution of the pro-system $\{\RR_n\}$ in \Cref{pro-system resolution of Rn},
we obtain an isomorphism of pro-systems: 
\[
\{\RR_n \otimes^\rmL_\RR (\RR \stimes \mathbb{D}(\alpha_p))\}
\cong \left\{
\Bigl(\bigoplus_{0 < i < n} V^i(1 \stimes \mathbb{D}(\alpha_p))\Bigr) \oplus 
\Bigl(\bigoplus_{i \geqslant 0} F^i \stimes \mathbb{D}(\alpha_p)\Bigr) \oplus 
\Bigl(\bigoplus_{0 < i < n} dV^i(1 \stimes \mathbb{D}(\alpha_p))\Bigr) \oplus 
\Bigl(\bigoplus_{i \geqslant 0} F^i d \stimes \mathbb{D}(\alpha_p)\Bigr)
\right\}.
\]
Taking the inverse limit gives our result.
\end{proof}

\begin{proposition}
\label{key computation}
The object \(E_{1/2} \cdstimes \mathbb{D}(\alpha_p)\) is cohomologically concentrated in 
$[-1, 0]$. Its cohomologies are given by the following dominoes:
$\mathrm{H}^{-1} = U_{-1}$ and $\mathrm{H}^0 = U_1$.
\end{proposition}

\begin{proof}
By \cite[Lemma III.1.5]{Ek2}, we have the following resolution of graded left $\RR$-modules:
\[
0 \rightarrow \widehat{\RR} \xrightarrow{\cdot(F^i - V^j)} \widehat{\RR} \rightarrow E_{j/i+j} \rightarrow 0.
\]
Therefore, we need to consider the map
$\widehat{\RR} \cdstimes \mathbb{D}(\alpha_p) \to \widehat{\RR} \cdstimes \mathbb{D}(\alpha_p)$
given by applying $\widehat{(-)}$ to
\[
\RR \stimes \mathbb{D}(\alpha_p) \xrightarrow{(- \cdot (F - V)) \stimes \id} \RR \stimes \mathbb{D}(\alpha_p).
\]
Using the decomposition in \Cref{decompose wdhR cdstimes k}, we obtain the following descriptions of the
kernel and cokernel:
\begin{itemize}
\item The kernel is given by $\prod_{i > 0} V^i(1 \stimes \mathbb{D}(\alpha_p))$
in grading $0$ and $\prod_{i \geqslant 0} dV^i(1 \stimes \mathbb{D}(\alpha_p))$
in grading $1$;
\item The cokernel is given by $\prod_{i \geqslant 0} V^i(1 \stimes \mathbb{D}(\alpha_p))$
in grading $0$ and $\prod_{i > 0} dV^i(1 \stimes \mathbb{D}(\alpha_p))$
in grading $1$.
\end{itemize}
This completes our computation.
\end{proof}

\subsection{Diagonal $t$-structure}
\label{subsection: Diagonal}

Recall that in \cite{Ek3}, Ekedahl defined a $t$-structure on \(\DGr_c^b(\RR)\), which he termed
the ``diagonal $t$-structure'', and proved his famous inequality \Cref{Ekeineq}. 
The goal of this subsection is to extend his diagonal $t$-structure to the whole \(\DGr_c(\RR)\).
We begin by reviewing his definition of the diagonal $t$-structure on \(\DGr_c^b(\RR)\).

\begin{definition}
For each coherent graded left \(\RR\)-module \(M\), define \(M^{\leqslant i}\) and \(M^{\geqslant i}\) 
to be the coherent \(\RR\)-modules given by
\[
\begin{aligned}
M^{\leqslant i} \coloneqq (\cdots \xrightarrow{d} M^{i-1}\xrightarrow{d} M^i 
\xrightarrow{d} F^\infty d(M^{i}) \rightarrow 0 \rightarrow \cdots), \\
M^{\geqslant i} \coloneqq (\cdots \rightarrow 0 \rightarrow M^i/F^\infty d(M^{i-1})\xrightarrow{d} M^{i+1}\rightarrow \cdots).
\end{aligned}
\]
We denote
\[
M^{[i, i]} \coloneqq (M^{\leqslant i})^{\geqslant i} = (M^{\geqslant i})^{\leqslant i} 
= (\cdots \rightarrow 0 \rightarrow M^i/F^\infty d(M^{i-1})\xrightarrow{d} F^\infty d(M^{i}) \rightarrow 0 \rightarrow \cdots).
\]

Note that these two formulas define functors \((-)^{\leqslant i}\) and \((-)^{\geqslant i}\).
Moreover, there are natural transformations
\(
(-)^{\leqslant i} \to \id \to (-)^{\geqslant j}
\)
for each \(i\) and \(j\).
The connective (resp.~co-connective) part of Ekedahl's diagonal $t$-structure on \(\DGr_c^b(\RR)\) is then defined by
\[
\begin{aligned}
\widetilde{\DGr}^{b,\leqslant 0}_{c} \coloneqq \{M \in \DGr_c^b(\RR) \mid 
\HH^i(M)^{\leqslant -i} \xrightarrow{\cong} \HH^i(M)\},\\
\widetilde{\DGr}^{b,\geqslant 0}_c \coloneqq \{M \in \DGr_c^b(\RR) \mid
\HH^i(M) \xrightarrow{\cong} \HH^i(M)^{\geqslant -i}\}.
\end{aligned}
\]
\end{definition}

\begin{theorem}[{\cite[Chapter 0 \& I]{Ek3}}]
\label{bounded diagonal t structure}
The above defines a $t$-structure on $\DGr_c^b(\RR)$. Moreover, if we
denote the canonical connective/co-connective truncation functors by $\tau^{\leqslant 0}$ and $\tau^{\geqslant 0}$,
and the truncation functors with respect to the diagonal $t$-structure by $\widetilde{\tau}^{\leqslant 0}$
and $\widetilde{\tau}^{\geqslant 0}$, then they satisfy the following properties:
\begin{enumerate}
\item For each pair of integers $(i, j)$ and each object $M \in \DGr_c^b(\RR)$, the map
$\HH^i(\widetilde{\tau}^{\leqslant j}(M)) \to \HH^i(M)$ induces a natural identification
$\HH^i(\widetilde{\tau}^{\leqslant j}(M)) \xrightarrow{\cong} \HH^{i}(M)^{\leqslant j-i}$;
similarly the map $\HH^i(M) \to \HH^i(\widetilde{\tau}^{\geqslant j}(M))$ induces a natural identification
$\HH^{i}(M)^{\geqslant j-i} \xrightarrow{\cong} \HH^i(\widetilde{\tau}^{\geqslant j}(M))$.
\item The canonical truncation functors preserve connective/co-connective part of the diagonal $t$-structure;
similarly the diagonal truncation functors preserve canonical connective/co-connective part.
\item For each pair of integers $(i, j)$, we have the following natural transformations
which are all equivalences:
$\widetilde{\tau}^{\leqslant j} \circ \tau^{\leqslant i} \xrightarrow{\cong} \tau^{\leqslant i} \circ \widetilde{\tau}^{\leqslant j}$,
$\tau^{\geqslant i} \circ \widetilde{\tau}^{\leqslant j} \xrightarrow{\cong} \widetilde{\tau}^{\leqslant j} \circ \tau^{\geqslant i}$,
$\widetilde{\tau}^{\geqslant i} \circ \tau^{\leqslant j} \xrightarrow{\cong} \tau^{\leqslant j} \circ \widetilde{\tau}^{\geqslant i}$, and
$\tau^{\geqslant i} \circ \widetilde{\tau}^{\geqslant j} \xrightarrow{\cong} \widetilde{\tau}^{\geqslant j} \circ \tau^{\geqslant i}$.
\end{enumerate}
\end{theorem}

Let us explain how the natural transformation 
\(
\tau^{\geqslant i} \circ \widetilde{\tau}^{\leqslant j} 
\xrightarrow{\cong} \widetilde{\tau}^{\leqslant j} \circ \tau^{\geqslant i}
\)
is defined. 
Applying $\widetilde{\tau}^{\leqslant j}$ to the natural arrow $\id \to \tau^{\geqslant i}$
gives an arrow $\widetilde{\tau}^{\leqslant j} \to \widetilde{\tau}^{\leqslant j} \circ \tau^{\geqslant i}$.
By the second property above, we see that the target lies in $\DGr^{b,\geqslant i}_c(\RR)$.
Hence the above arrow factors through $\tau^{\geqslant i} \circ \widetilde{\tau}^{\leqslant j}$.
The other arrows are defined similarly.\footnote{Namely, we always apply diagonal truncations to the natural transformation
between $\id$ and the canonical truncations, and then obtain the desired arrow by observing that the source or the target lies
in the appropriate shifts of the connective or co-connective part of the canonical $t$-structure, by the second property above.}

For the reader's convenience, we include a sketch of Ekedahl's proof. 

\begin{proof}[Proof sketch]
We say that a coherent $\RR$-module $M$ is ``concentrated in grading $m$'' if $M^{\ell} = 0$ for all $\ell \notin \{m, m+1\}$,
and $M^{m+1} = F^{\infty}d(M^m)$.
Using the resolutions in \cite[Lemma III.1.5]{Ek2},
Ekedahl observed (\cite[Proposition 0.4.1]{Ek3}) that
if $M$ and $N$ are coherent $\RR$-modules concentrated in gradings $m$ and $n$, respectively, then
\(
\Ext^i_{\RR}(M, N) = 0
\)
whenever $i < (n - m)$ and $2 \leqslant i \leqslant (n - m)$.
The first vanishing result immediately implies that if $M \in \widetilde{\DGr}^{b, \leqslant 0}_c(\RR)$
and $N \in \widetilde{\DGr}^{b, \geqslant 1}_c(\RR)$, then
\(
\Hom_{\DGr(\RR)}(M, N) = 0.
\)

Next, we show the existence of a diagonal connective cover that moreover satisfies property (1).
To this end, we argue by induction on the cohomological amplitude
of $M \in \DGr^b_c(\RR)$.
If the cohomological amplitude is $0$, that is, if $M$ is a cohomological shift of a coherent
$\RR$-module, the existence of a diagonal connective cover satisfying property (1) is obvious.

In general, assume that the statement is known for all objects of cohomological amplitude $\leqslant (n-1)$.
Suppose that $M \in \DGr^{[a-n, a]}_c(\RR)$, and consider the exact triangle
\(
\tau^{< a}M \to M \to \HH^a(M)[-a].
\)
By the induction hypothesis, both $\tau^{< a}M$ and $\HH^a(M)[-a]$
admit diagonal connective covers satisfying property (1).
The second vanishing in the first paragraph then implies that the composite
\[
\widetilde{\tau}^{\leqslant 0}(\HH^a(M)[-a]) \to \HH^a(M)[-a] \to (\tau^{< a}M)[1] \to 
\widetilde{\tau}^{\geqslant 1}(\tau^{< a}M)[1]
\]
is $0$. Therefore, this composite canonically factors through
\(
\widetilde{\tau}^{\leqslant 0}(\HH^a(M)[-a]) \to \widetilde{\tau}^{\leqslant 0}(\tau^{< a}M)[1].
\)
Taking its fiber, which then necessarily maps to $M$, produces the desired diagonal connective cover
of $M$ satisfying property (1).

Finally, properties (2) and (3) are formal consequences of property (1).
\end{proof}

We now extend Ekedahl's diagonal $t$-structure to $\DGr_c(\RR)$ without the cohomological boundedness constraint.

\begin{definition}
We define the following full subcategories of $\DGr_c(\RR)$:
\[
\begin{aligned}
\widetilde{\DGr}^{\leqslant 0}_{c} \coloneqq \{M \in \DGr_c(\RR) \mid 
\HH^i(M)^{\leqslant -i} \xrightarrow{\cong} \HH^i(M)\},\\
\widetilde{\DGr}^{\geqslant 0}_c \coloneqq \{M \in \DGr_c(\RR) \mid
\HH^i(M) \xrightarrow{\cong} \HH^i(M)^{\geqslant -i}\}.
\end{aligned}
\]
\end{definition}

\begin{theorem}
\label{diagonal t structure}
The above defines a $t$-structure on $\DGr_c(\RR)$. Moreover, if we
denote the canonical connective and co-connective truncation functors by $\tau^{\leqslant 0}$ and $\tau^{\geqslant 0}$,
and the truncation functors with respect to the diagonal $t$-structure by $\widetilde{\tau}^{\leqslant 0}$
and $\widetilde{\tau}^{\geqslant 0}$, then they satisfy the following properties:
\begin{enumerate}
\item For each pair of integers $(i, j)$ and each object $M \in \DGr_c(\RR)$, the map
$\HH^i(\widetilde{\tau}^{\leqslant j}(M)) \to \HH^i(M)$ induces a natural identification
$\HH^i(\widetilde{\tau}^{\leqslant j}(M)) \xrightarrow{\cong} \HH^{i}(M)^{\leqslant j-i}$.
Similarly, the map $\HH^i(M) \to \HH^i(\widetilde{\tau}^{\geqslant j}(M))$ induces a natural identification
$\HH^{i}(M)^{\geqslant j-i} \xrightarrow{\cong} \HH^i(\widetilde{\tau}^{\geqslant j}(M))$.

\item The canonical truncation functors preserve the connective and co-connective parts of the diagonal $t$-structure.
Similarly, the diagonal truncation functors preserve the connective and co-connective parts of the canonical $t$-structure.

\item For each pair of integers $(i, j)$, we have the following natural transformations,
all of which are equivalences:
\[
\widetilde{\tau}^{\leqslant j} \circ \tau^{\leqslant i} \xrightarrow{\cong} \tau^{\leqslant i} \circ \widetilde{\tau}^{\leqslant j},\quad
\tau^{\geqslant i} \circ \widetilde{\tau}^{\leqslant j} \xrightarrow{\cong} \widetilde{\tau}^{\leqslant j} \circ \tau^{\geqslant i},
\]
\[
\widetilde{\tau}^{\geqslant i} \circ \tau^{\leqslant j} \xrightarrow{\cong} \tau^{\leqslant j} \circ \widetilde{\tau}^{\geqslant i},\quad
\tau^{\geqslant i} \circ \widetilde{\tau}^{\geqslant j} \xrightarrow{\cong} \widetilde{\tau}^{\geqslant j} \circ \tau^{\geqslant i}.
\]
\end{enumerate}
\end{theorem}

\begin{proof}
Suppose that $M \in \widetilde{\DGr}^{\leqslant 0}_{c}(\RR)$ and $N \in \widetilde{\DGr}^{\geqslant 1}_c(\RR)$.
We need to show that $\mathrm{Hom}_{\DGr(\RR)}(M , N) = 0$.
Using the Postnikov filtrations on $M$ and $N$, we have a canonical isomorphism
\[
\mathrm{RHom}_{\DGr(\RR)}(M, N) \cong \mathrm{Rlim}_{n \to \infty} 
\mathrm{RHom}_{\DGr(\RR)}(\tau^{[-n, n]}M, \tau^{[-n, n]}N),
\]
which has no $\HH^{\leqslant 0}$. Indeed, for each $n$, we have 
\(
\tau^{[-n, n]}M \in \widetilde{\DGr}^{b,\leqslant 0}_{c}(\RR)\), and \(
\tau^{[-n, n]}N \in \widetilde{\DGr}^{b, \geqslant 1}_c(\RR),
\)
so 
\(
\mathrm{RHom}_{\DGr(\RR)}(\tau^{[-n, n]}M, \tau^{[-n, n]}N)
\)
has no $\HH^{\leqslant 0}$ by \Cref{bounded diagonal t structure}.

Next, we show the existence of diagonal truncations. Let $M \in \DGr_c(\RR)$ and
consider the diagonal connective covers of its
Postnikov truncations
\(
N^{[a, b]} \coloneqq \widetilde{\tau}^{\leqslant 0}(\tau^{[a, b]}M),
\)
where $a \leqslant b$ runs through all such pairs of integers.
By property (2) of \Cref{bounded diagonal t structure}, we see that
$N^{[a,b]} \in \DGr^{[a, b]}_c(\RR)$. Moreover, by property (3) of \Cref{bounded diagonal t structure},
for each pair of integers $a \leqslant b$, the natural maps
\[
\begin{aligned}
N^{[a, b]} &\longrightarrow \tau^{[a, b]} N^{[a, b+1]}, \\
\tau^{[a, b]} N^{[a-1, b]} &\longrightarrow  N^{[a, b]}
\end{aligned}
\]
are isomorphisms.
Therefore, we see that
\(
N \coloneqq \lim_{a \to -\infty} \mathrm{colim}_{b \to \infty} N^{[a, b]}
\)
exists together with natural isomorphisms
\(
\tau^{[a, b]}N \cong N^{[a, b]}
\)
for all pairs of integers $a \leqslant b$.

Using property (3) of \Cref{bounded diagonal t structure} again,
each $N^{[a, b]}$ has a natural map to $\tau^{[a, b]}M$ compatible with
Postnikov truncations. Therefore, we obtain a natural map $N \to M$ whose induced map on $\HH^i$
coincides with the map induced by $N^{[a, b]} \to \tau^{[a, b]}M$ for any $a \leqslant b$
with $i \in [a, b]$.
Hence property (1) of \Cref{bounded diagonal t structure}
implies that the map $N \to M$ satisfies the first half of property (1) here.

Therefore, we see that $N \in \widetilde{\DGr}_c^{\leqslant 0}(\RR)$ and
$\mathrm{Cone}(N \to M) \in \widetilde{\DGr}_c^{\geqslant 1}(\RR)$.
Moreover, the second half of property (1) also follows by considering the long exact sequence associated with
\(
N \to M \to \mathrm{Cone}(N \to M).
\)

Finally, properties (2)–(3) follow formally from property (1).
\end{proof}

\begin{notation}
We denote the heart of the diagonal $t$-structure by 
\(
\dc' \coloneqq \widetilde{\DGr}_c^{\leqslant 0}(\RR) \cap 
\widetilde{\DGr}_c^{\geqslant 0}(\RR).
\)
We denote the associated cohomology functors by
\(
\widetilde{\HH}^n(\td) \coloneqq 
(\dt^{\leqslant n} \circ \dt^{\geqslant n}(\td))[n] \colon 
\DGr_c(\RR)\rightarrow \dc'.
\)
\end{notation}

\begin{lemma}
\label{H and tildeH}
Let $M \in \DGr_c(\RR)$, for each pair of integers $(i, j)$ we have
$\HH^j(\widetilde{\HH}^{i}(M)) \cong \HH^{i + j}(M)^{[-j, -j]}$.
\end{lemma}

\begin{proof}
Let us compute:
\[
\HH^j(\widetilde{\HH}^{i}(M)) 
\coloneqq 
\HH^j\Big((\dt^{\leqslant i} \circ \dt^{\geqslant i})(M)[i]\Big)
=
\HH^{i+j}\Big(\dt^{\leqslant i} ( \dt^{\geqslant i}(M))\Big)
\cong 
\HH^{i+j}(\dt^{\geqslant i}(M))^{\leqslant -j}
\cong 
\HH^{i+j}(M)^{[-j, -j]}.
\]
In the last two identifications, we have used property (1) of \Cref{diagonal t structure}. 
\end{proof}

\begin{remark}
\label{diagonal t-bound for smooth varieties}
In particular, as a consequence of \Cref{H and tildeH},
for any smooth proper variety $X$ over $k$ of dimension $\leqslant D$,
we have $\mathrm{R\Gamma}(X, W\Omega^{\bullet}_{X/k})$ lives in $\widetilde{\DGr}_c^{[0, 2D]}(\RR)$.
\end{remark}

\begin{lemma}
\label{characterizing dc inside dc'}
Let $M \in \dc' \subset \DGr_c(\RR)$, the following are equivalent:
\begin{enumerate}
\item $M \in \dc \subset \dc'$;
\item $M$ has bounded cohomology;
\item $M$ has bounded grading.
\end{enumerate}
\end{lemma}

\begin{proof}
The equivalence between (1) and (2) follows from the fact that
\(
\dc = \dc' \cap \DGr^b_c(\RR).
\)
The equivalence between (2) and (3) follows from the definition of
$M \in \dc'$: this implies that the $i$-th cohomology of $M$
has no graded pieces other than those in gradings $-i$ and $-i + 1$.
\end{proof}

\begin{lemma}
\label{First quadrant implies inside dc}
Let $M \in \DGr_c(\RR)$ be an object that is cohomologically bounded below
and grading-left bounded (that is, the grading $\leqslant -N$ pieces of $M$ vanish for some integer $N$).
Then for each $i \in \mathbb{Z}$, we have
\(
\widetilde{\HH}^i(M) \in \dc \subset \dc'.
\)
\end{lemma}

\begin{proof}
Fix an integer $i$. By \Cref{characterizing dc inside dc'}, it suffices to show that
\(
\HH^j(\widetilde{\HH}^{i}(M)) = 0
\)
whenever $j$ is sufficiently large or sufficiently small.
Using \Cref{H and tildeH}, we obtain the following:
\begin{enumerate}
\item If $j$ is sufficiently small, the vanishing follows from the assumption that $M$ is cohomologically bounded below.
\item If $j$ is sufficiently large, the vanishing follows from the assumption that $M$ is grading-left bounded.
\end{enumerate}
This completes the proof.
\end{proof}

\medskip 

Using the diagonal \(t\)-structure on \(\DGr_c^b(\RR)\), Ekedahl established a fundamental inequality
that reveals deep connections between Hodge--Witt numbers and Hodge numbers.

\begin{theorem}[Ekedahl's inequality {\cite[Theorem IV.3.3]{Ek3}}]
\label{Ekeineq}
Let \(M \in \DGr_c^b(\RR)\). Then for any \((i,j)\in \mathbb{Z}^2\), we have
\(
h_W^{i,j}(M) \leqslant h^{i,j}(M).
\)
Consequently, if \(X\) is a smooth proper variety over a perfect field \(k\) of characteristic \(p\), then for any \((i,j)\in \mathbb{N}^2\) we have
\(
h_W^{i,j}(X) \leqslant h^{i,j}(X).
\)
\end{theorem}

This result provides the following insight into the behaviour of the Hodge--Witt numbers \(h_W^{i,j}\).

\begin{proposition}[{\cite[Proposition III.4.1]{Ek3}, \cite[Theorem IV.1.2(i)]{Ek3}}]
\label{Mazur--Ogus}
Let \(M \in \DGr_c^b(\RR)\). The following conditions are equivalent:

{\rm (1)} The cohomology groups of the total complex \(\Tot(M) \in \calD(W)\) are torsion-free, and the Hodge--de Rham spectral sequence (see  {\Cref{Hodge--de Rham spectral sequence}})
\[
E_1^{i,j}=\HH^j(\RR_1\otimes^{\rmL}_\RR M)^i
\;\Rightarrow\;
\HH^{i+j}(k\otimes^{\rmL}_W \Tot(M))
\]
degenerates at the \(E_1\)-page.

{\rm (2)} For every \((i,j)\), we have
\(
h^{i,j}=h_W^{i,j}.
\)

{\rm (3)} For each \(n\in \mathbb{Z}\), let \(b_n\) denote the dimension of \(\HH^n(\Tot(M))[1/p]\). Then
\(
\sum_{i+j=n} h^{i,j}=b_n.
\)
\end{proposition}

\begin{proof}
A direct computation shows that
\(
\sum_{i+j=n} h_W^{i,j}(M)
=
\sum_{i+j=n} m^{i,j}(M)
=
b_n .
\)
Therefore, by Ekedahl's inequality \Cref{Ekeineq}, condition (2) is equivalent to condition (3).

By the universal coefficient theorem and the spectral sequence of \Cref{Hodge--de Rham spectral sequence}, we have
\[
b_n
\leqslant
\dim_k \HH^n(k \otimes^{\rmL}_W \Tot(M))
\leqslant
\sum_{i+j=n} h^{i,j}.
\]
Condition (1) is precisely the statement that both inequalities are equalities for all \(n \in \mathbb{Z}\).
Thus condition (1) is also equivalent to condition (3).
\end{proof}

\begin{definition}[{\cite[Definition IV.1.1]{Ek3}}]
\label{Mazur--Ogus definition}
An object \(M\in \DGr_c^b(\RR)\) is called \emph{Mazur--Ogus} if it satisfies the equivalent conditions of \Cref{Mazur--Ogus}. 
A smooth proper variety \(X\) over a perfect field \(k\) of characteristic \(p\) is called \emph{Mazur--Ogus} if the object
\(
\mathrm{R}\Gamma(X,W\Omega^\bullet_{X/k}) \in \DGr_c^b(\RR)
\)
is Mazur--Ogus.
\end{definition}

\begin{remark}
In Ekedahl's original definition \cite[Definition I.4.3]{Ek3}, 
the Mazur--Ogus condition for an object \(M\in \dc\) is formulated differently.
However, he later showed in \cite[Theorem IV.1.2(i)]{Ek3} that the two definitions agree.
\end{remark}

\begin{theorem}[{\cite[Theorem IV.1.2(iv)]{Ek3}}]
\label{MazurOgus}
Assume that \(M\in \DGr_c^b(\RR)\) is Mazur--Ogus in the sense of  {\Cref{Mazur--Ogus definition}}.
Then there is an isomorphism in \(\DGr_c^b(\RR)\):
\[
M \simeq \bigoplus_n \widetilde{\HH}^n(M)[-n].
\]
\end{theorem}

\begin{proposition}[{\cite[Proposition III.4.11]{Ek3}}]
\label{Kunneth for Mazur--Ogus}
Assume \(M,N \in \dc\) are Mazur--Ogus objects. Then \(M \cdstimes N\in \dc\) is also Mazur--Ogus. 
Consequently, for two Mazur--Ogus objects \(M,N\in \DGr_c^b(\RR)\), we have
\begin{enumerate}
\item \(M\cdstimes N\) is Mazur--Ogus; and
\item \(\widetilde{\HH}^n(M\cdstimes N)\cong \bigoplus_{i+j=n}\widetilde{\HH}^i(M)\cdstimes \widetilde{\HH}^j(N)\).
\end{enumerate}
\end{proposition}

\begin{example}\label{The de Rham-Witt cohomology of P^n}
Let \(\mathbb{P}^n\) be the projective space over \(k\) of dimension \(n\).
For each \(0\le i\le n\), the crystalline cohomology satisfies
\(
\HH^{2i}_{\cris}(\mathbb{P}^n/W) \cong W,
\)
where the Frobenius operator acts as \(p^i\) times the usual Witt-vector Frobenius,
and \(\HH^{2i+1}_{\cris}(\mathbb{P}^n/W)=0\).
The Hodge numbers are given by \(h^{i,i}(\mathbb{P}^n)=1\) and \(h^{m,n}(\mathbb{P}^n)=0\) for \(m\neq n\).

A direct computation shows that the equality
\(
\sum_{i+j=n} h^{i,j}=b_n
\)
holds for all \(n\). Hence \(\mathbb{P}^n\) is Mazur--Ogus in the sense of \Cref{Mazur--Ogus definition}.
Unwinding the equality \(h_W^{i,j}=h^{i,j}\), together with the description of the slopes of its crystalline cohomology,
we see that \(\mathbb{P}^n\) is Hodge--Witt in the sense of \Cref{Hodge--Witt varieties}.
By \Cref{Hodge--Witt decomposition}, we obtain a canonical decomposition
\[
(\HH^m_{\cris}(\mathbb{P}^n/W), \varphi_{\cris})
\cong 
\bigoplus_{i+j=m}
(\HH^j(\mathbb{P}^n,W\Omega^i_{\mathbb{P}^n/k}), p^i F).
\]
Consequently, there is a decomposition in \(\DGr_c^b(\RR)\):
\[
\mathrm{R}\Gamma(\mathbb{P}^n,W\Omega^\bullet_{\mathbb{P}^n/k})
\;\cong\;
\bigoplus_{0\le i\le n} W(-i)[-i].
\]
\end{example}

For any Dieudonn\'e module \((M,F,V)\), there is a canonical decomposition
\(
(M,F,V)
=
(M_{\mathrm{uni}},F_{\mathrm{uni}},V_{\mathrm{uni}})
\oplus
(M_{\mathrm{mul}},F_{\mathrm{mul}},V_{\mathrm{mul}}),
\)
where \(M_{\mathrm{uni}}\) is the summand on which \(V\) acts topologically nilpotently,
and \(M_{\mathrm{mul}}\) is the summand on which \(V\) is bijective.
We introduce the following notation.

\begin{notation}
\label{' notation}
For any multiplicative-type Dieudonn\'e module \((M,F,V)\) (that is, \(V\) is bijective),
we denote by \(M'\) the Dieudonn\'e module \((M,V^{-1},pV)\).
\end{notation}

\begin{lemma}
\label{taking ' is exact}
The endofunctor on the category of finite-type Dieudonn\'e modules
given by
\(
(M,F,V)\longmapsto (M_{\mathrm{mul}})'
\)
is exact.
\end{lemma}

\begin{proof}
Taking the direct summand on which \(V\) is bijective is an exact operation.
Moreover, the passage from \((M,F,V)\) (with bijective \(V\)) to \((M,V^{-1},pV)\) is also exact.
\end{proof}

\begin{lemma}
\label{widetildeH^1(A)}
For any abelian variety \(A\), there is a canonical isomorphism
\[
\widetilde{\HH}^1(W\Omega^\bullet(A))
\cong
\HH^1_{\cris}(A/W)_{\mathrm{uni}}
\oplus
(\HH^1_{\cris}(A/W)_{\mathrm{mul}})'(-1)[1],
\]
where \(\D(-)\) denotes the usual contravariant Dieudonn\'e module functor.
\end{lemma}

\begin{proof}
By \cite[Corollaire II.2.17, Proposition II.2.19]{IlldRW}, 
any smooth proper variety satisfies the condition of \cite[IV.2.15.6]{IRdRW}
for \(n=1\).
Using \cite[Th\'eor\`eme IV.4.5]{IRdRW}, we obtain a canonical decomposition
\[
(\HH^1_{\cris}(A/W), \varphi_{\cris})
\cong
(\HH^1(WO_A), F)
\oplus
(\HH^0(W\Omega^1_{A/k}), pF).
\]
Consequently, there are canonical isomorphisms of Dieudonn\'e modules
\[
\HH^1(WO_A)
\cong
\HH^1_{\cris}(A/W)_{\mathrm{uni}},
\qquad
\HH^0(W\Omega^1_{A/k})
\cong
(\HH^1_{\cris}(A/W)_{\mathrm{mul}})'.
\]
By \Cref{H and tildeH}, the object \(\widetilde{\HH}^1(W\Omega^\bullet(A))\)
is an extension of \(\HH^1(WO_A)\) by \(\HH^0(W\Omega^1_{A/k})(-1)[1]\).
Since these two terms are concentrated in different gradings, the extension is canonically split.
Therefore we obtain the desired canonical isomorphism.
\end{proof}

\addtocontents{toc}{\protect\setcounter{tocdepth}{2}}

\section{de Rham--Witt cohomology of smooth stacks}
\label{dRW of smooth stacks}

Let \(k\) be a perfect field of characteristic \(p>0\).
We begin by following \cite[Section 2]{ABM} to extend the definition of the de Rham--Witt complex
(and its cohomology) from smooth \(k\)-schemes to smooth geometric Artin stacks over \(k\).
For a brief review of geometric Artin stacks, we refer the reader to \cite[Appendix A.1]{KP24},
especially \cite[Theorem A.1.6]{KP24} and the discussion preceding it.

\begin{notation}[cf.~{\cite[Notation 2.1]{ABM}}]
Let \(\mathrm{Sm}_k\) denote the category of smooth \(k\)-algebras. Its opposite category
\(\mathrm{Sm}_k^{\mathrm{op}}\), equipped with the Grothendieck topology in which covers are given by
finite families of smooth morphisms that are jointly surjective, is called the \emph{smooth site} of \(k\). 
A geometric Artin stack \(X/k\) is said to be \emph{smooth} if there exists a smooth
cover \(U \twoheadrightarrow X\) with \(U\) a smooth \(k\)-scheme.
This notion agrees with the smoothness of the morphism
\(X \to \Spec(k)\) as defined in \cite[A.1.15--16]{KP24}. 
We denote by \(\mathrm{SmStk}_k\) the category of smooth geometric Artin stacks over \(k\),
again equipped with the smooth topology.
\end{notation}

\begin{proposition}
\label{WOmega smooth sheaf}
The \(\DGr(\RR)\)-valued functor \(A \mapsto W\Omega^{\bullet}(A/k)\)
defines a sheaf on \(\mathrm{Sm}_k^{\mathrm{op}}\).
Similarly, for each \(n\), the \(\DGr(W_n[d]/d^2)\)-valued functor
\(
A \mapsto W_n\Omega^\bullet_{A/k}
\)
defines a sheaf on \(\mathrm{Sm}_k^{\mathrm{op}}\).
\end{proposition}

\begin{proof}
By \cite[\href{https://stacks.math.columbia.edu/tag/055V}{Tag 055V}]{stacks-project},
it suffices to show that the said functors satisfy the sheaf property with respect to \'{e}tale
coverings, which is shown in \cite[Proposition II.1.13, II.1.14]{IlldRW}.
\end{proof}

\begin{remark}
Let us mention that a more general descent statement for the de Rham--Witt complex is known, 
provided that one works in the animated setting, see \cite[Proposition 2.19]{DM25}.
\end{remark}

\begin{construction}[de Rham--Witt cohomology of geometric Artin stacks]
\label{HWArtinstack}
Let $X/k$ be a smooth geometric Artin stack. We set
\[
W\Omega^{\bullet}(X/k) \coloneqq \mathrm{R\Gamma}(\mathrm{Sm}^{\mathrm{op}}_{k, /X}, W\Omega^{\bullet}) \in \DGr(\RR).
\]
By \Cref{WOmega smooth sheaf}, we see that the assignment $X/k \mapsto W\Omega^{\bullet}(X/k)$
is again a sheaf on $\mathrm{SmStk}_k$. 
In particular, one may ``compute'' the value $W\Omega^{\bullet}(X/k)$ by induction on the geometricity of $X$.
When $X$ is $(-1)$-geometric, which means it is a smooth affine scheme $X = \Spec(A)$, it is given by
the usual $W\Omega^{\bullet}(A/k)$. Similarly, if $X$ is a smooth scheme, then we can compute
$W\Omega^{\bullet}(X/k)$ by Zariski descent. 
In general, suppose $X$ is a smooth $n$-geometric stack. Then there exists an $(n-1)$-representable smooth surjection
$\pi \colon U \twoheadrightarrow X$ from a smooth scheme $U$.
All the terms $U^{({\ast})} \coloneqq U^{\times_X (\ast + 1)}$ in the \v{C}ech nerve of $\pi$
are smooth $(n-1)$-geometric stacks, hence their de Rham--Witt cohomologies are ``computed'' by induction.
Moreover, we have an equivalence in $\DGr(\RR)$:
\[
W\Omega^{\bullet}(X/k) \xrightarrow{\cong} \mathrm{Tot}\big(
\begin{tikzcd}[column sep=large]
W\Omega^{\bullet}(U^{(0)}/k) \arrow[r, yshift=0.6ex] \arrow[r, yshift=-0.6ex] 
& W\Omega^{\bullet}(U^{(1)}/k) \arrow[r, yshift=1.2ex] \arrow[r] \arrow[r, yshift=-1.2ex] 
& W\Omega^{\bullet}(U^{(2)}/k) \arrow[r, yshift=1.8ex] \arrow[r, yshift=0.6ex] \arrow[r, yshift=-0.6ex] 
\arrow[r, yshift=-1.8ex] 
& \cdots
\end{tikzcd}
\big).
\]
\end{construction}

Let us establish some easy properties of the above construction.
We say that an object in $\DGr(\RR)$ is grading-left bounded by $0$ if it has no negative grading pieces.

\begin{lemma}
\label{dRW of a smooth stack is always complete}
Let $X/k$ be a smooth geometric Artin stack. Then
$W\Omega^{\bullet}(X/k) \in \widehat{\DGr}(\RR)$ and it is grading-left bounded by $0$.
\end{lemma}

\begin{proof}
When $X$ is an affine scheme, the completeness of its de Rham--Witt complex
follows from \Cref{Wn and Rn}.
The bound on the grading follows from the definition.
The completeness for general $X$ follows from \Cref{completeness closed under taking limit},
whereas the bound on the grading follows from the fact that limits are computed ``grading-wise''.
\end{proof}

\begin{remark}
\label{E1 structure}
For each smooth algebra $A/k$, the multiplication on $W\Omega^{\bullet}_{A/k}$
satisfies the properties listed in \Cref{star product abelian level}
(see \cite[Th\'eor\`eme I.1.3, Proposition I.2.18]{IlldRW}),
making it an associative algebra in the symmetric monoidal category of
graded left $\RR$-modules. Since $W\Omega^{\bullet}_{A/k}$ is concentrated in cohomological
degree $0$, and $- \stimes^{\rmL} -$ preserves the connective part, we see that
$W\Omega^{\bullet}_{-/k}$ defines a sheaf on $\mathrm{Sm}_k^{\mathrm{op}}$
valued in $\mathrm{Alg}_{\mathbb{E}_1}(\DGr(\RR))$.
Finally, since $W\Omega^{\bullet}_{A/k}$ is complete, we see that
$W\Omega^{\bullet}_{-/k}$ defines a sheaf on $\mathrm{Sm}_k^{\mathrm{op}}$
valued in $\mathrm{Alg}_{\mathbb{E}_1}(\widehat{\DGr}(\RR))$.
Therefore, we may regard $W\Omega^{\bullet}(-/k)$
as a sheaf on $\mathrm{SmStk}_k^{\mathrm{op}}$ taking values in $\mathbb{E}_1$-algebras in the symmetric monoidal category
$\widehat{\DGr}(\RR)$.
\end{remark}

\begin{lemma}
\label{Hodge and R1}
There is an equivalence of functors
\[
\RR_1 \otimes^{\rmL}_{\RR} W\Omega^{\bullet}(-/k)
\xrightarrow{\cong}
(\wedge^{\bullet}\mathbb{L}_{-/k})
\colon \mathrm{SmStk}_k^{\mathrm{op}} \to \DGr(k[d]/d^2).
\]
\end{lemma}

\begin{proof}
Consider the following natural map
\[
\RR_1 \otimes^\rmL_{\RR} W\Omega^{\bullet}(-/k)
=
\RR_1 \otimes^\rmL_{\RR} \mathrm{R\Gamma}(\mathrm{Sm}^{\mathrm{op}}_{k, /-}, W\Omega^{\bullet})
\to
\mathrm{R\Gamma}(\mathrm{Sm}^{\mathrm{op}}_{k, /-}, \RR_1 \otimes^\rmL_{\RR} W\Omega^{\bullet}).
\]

We have the following identification of the right-hand side:
\[
\mathrm{R\Gamma}(\mathrm{Sm}^{\mathrm{op}}_{k, /-}, \RR_1 \otimes^\rmL_{\RR} W\Omega^{\bullet})
\cong
\mathrm{R\Gamma}(\mathrm{Sm}^{\mathrm{op}}_{k, /-}, \Omega^{\bullet})
\cong
(\wedge^{\bullet}\mathbb{L}_{-/k}).
\]
Here the first identification is Illusie--Raynaud's \Cref{Wn and Rn}, and the second identification follows from
\cite[Theorem 2.5.(1) and Construction 2.7]{ABM}.
This map is an equivalence: since $\RR_1$ is a perfect right $\RR$-module, we know that
$\RR_1 \otimes^{\rmL}_\RR (\td)$ commutes with taking limits.
\end{proof}

\begin{remark}
\label{E1 structure and tensor R1}
For each smooth algebra $A/k$, the $\mathbb{E}_1$-structure on $W\Omega^{\bullet}_{A/k}$
gives rise to a multiplication on $\RR_1 \otimes^{\rmL}_\RR W\Omega^{\bullet}_{A/k} \cong 
\Omega^{\bullet}_{A/k}$. Tracing through the definition, we see that the multiplication
is none other than the wedge product in Hodge cohomology (see \cite[Th\'eor\`eme I.1.3]{IlldRW}).
Therefore, the equivalence in \Cref{Hodge and R1} is
automatically promoted to an equivalence of $\mathbb{E}_1$-algebras in $\DGr(k[d]/d^2)$.
\end{remark}

For a class of ``nice'' smooth geometric Artin stacks, their de Rham--Witt cohomology
satisfies a finiteness property.
Recall that a geometric Artin stack $X$ is called quasi-compact if there is a smooth surjection
$\Spec(A) \twoheadrightarrow X$ from an affine scheme; and a morphism of geometric Artin stacks
$X \to Y$ is called quasi-compact if for any map $\Spec(A) \to Y$, the fiber product
$X \times_Y \Spec(A)$ is quasi-compact, see \cite[Definition A.1.20]{KP24}.
Also recall that a geometric Artin stack $X$ over $k$ is called quasi-separated if the diagonal
morphism $X \to X \times_k X$ is quasi-compact; this is what is called $0$-quasi-separated
in \cite[Definition A.1.21]{KP24}.
Finally, we remind the reader that a smooth geometric Artin stack $X/k$ is called \emph{Hodge-proper}
if $\HH^j(X, \wedge^i \mathbb{L}_{X/k})$ is finite dimensional for all $i$ and $j$,
cf.~\cite[Definition 0.2.1]{KP22}. 

\begin{theorem}
\label{Hodge-proper dRW is coherent}
Let $X/k$ be a quasi-compact quasi-separated smooth Hodge-proper geometric Artin stack.
Then the object $W\Omega^{\bullet}(X/k) \in \DGr(\RR)$ lies in the full subcategory
$\DGr^+_c(\RR)$. Moreover, for each integer $i$,
we have \(\widetilde{\HH}^i(R\Gamma(W\Omega^\bullet))\in \dc \subset \dc'\). 
\end{theorem}

\begin{proof}
By \Cref{coherence criteria} and \Cref{Hodge and R1}, it suffices to know that
for each $i$ the $i$-th Hodge cohomology is concentrated in finitely many gradings: 
Namely, $\HH^i(X, \wedge^{j}\mathbb{L}_{X/k}) = 0$ for $j \geqslant N_i$ for some natural number $N_i$. 
To that end, we argue by induction on the geometricity.
For $(-1)$-geometric Artin stacks, namely affine schemes, this is automatic, as we may choose
all $N_i$ to be the dimension of $X$. 
Suppose the claim is verified for all quasi-compact quasi-separated smooth $(n-1)$-geometric Artin stacks,
and let $X$ be a quasi-compact quasi-separated smooth $n$-geometric Artin stack.
Then there exists an $(n-1)$-representable smooth surjection $\pi \colon U \twoheadrightarrow X$
from a smooth affine scheme $U$,
and all the terms $U^{({\ast})} \coloneqq U^{\times_X (\ast + 1)}$ in the \v{C}ech nerve of $\pi$
are quasi-compact quasi-separated smooth $(n-1)$-geometric Artin stacks. 
Let $N_i^{(m)}$ be the natural number such that
$\HH^{i}(U^{(m)}, \wedge^{\geq N_i^{(m)}}\mathbb{L}_{U^{(m)}/k}) = 0$.
By \cite[Theorem 2.5.(1) \& Construction 2.7]{ABM},
the Hodge cohomology of $X$ is given by the totalization of the Hodge cohomologies
of $U^{({\ast})}$.
Our claim follows by choosing
\[
N_i \coloneqq \max\{N_{i-m}^{(m)} \mid 0 \leqslant m \leqslant i\}.
\]

The last statement follows from the combination of \Cref{First quadrant implies inside dc}
and \Cref{dRW of a smooth stack is always complete}.
\end{proof}

Next, let us study a class of ``nice'' morphisms between smooth geometric Artin stacks.
Recall that in \cite[Definition 5.1]{ABM}, the authors define
a map of syntomic algebraic stacks $X \to Y$ to be a
\emph{Hodge $d$-equivalence} if 
\[
\mathrm{H}^i\Big(\mathrm{Cone}\big(\mathrm{R\Gamma}(Y, \wedge^{j}\mathbb{L}_{Y/k})
\to \mathrm{R\Gamma}(X, \wedge^{j}\mathbb{L}_{X/k}) \big)\Big) = 0
\]
for all $i + j < d$. Similarly, we make the following definition.

\begin{definition}
We say that a map of smooth geometric Artin stacks $X \to Y$ is a \emph{de Rham--Witt $d$-equivalence}
if
\[
\mathrm{H}^i\Big(\mathrm{Cone}\big(W\Omega^{\bullet}(Y/k)
\to W\Omega^{\bullet}(X/k) \big)\Big)^j = 0
\]
for all $i + j < d$. 

We say that the map $X \to Y$ is a \emph{crystalline $d$-equivalence}
if
\[
\mathrm{H}^i\Big(\mathrm{Cone}\big(R\Gamma_{\cris}(Y/W)
\to R\Gamma_{\cris}(X/W) \big)\Big) = 0
\]
for all $i < d$. 
\end{definition}

\begin{lemma}
\label{Hodge equivalence implies crystalline equivalence}
Let $f \colon X \to Y$ be a map of smooth geometric Artin stacks.
If $f$ is a Hodge $d$-equivalence, then it is a crystalline $d$-equivalence.
\end{lemma}

\begin{proof}
If we define \emph{de Rham $d$-equivalence} analogously, then a Hodge $d$-equivalence
implies a de Rham $d$-equivalence by the Hodge--de Rham spectral sequence.

We now note that a de Rham $d$-equivalence implies a crystalline $d$-equivalence.
Indeed, a derived $p$-complete complex lies in $\mathcal{D}^{\geqslant m}(\mathbb{Z}_p)$
if its derived reduction modulo $p$ does so. This condition implies that its derived reduction
modulo $p^n$ lies in $\mathcal{D}^{\geqslant m}(\mathbb{Z}_p)$ for all $n \in \mathbb{N}$,
and taking limits preserves coconnectivity.
\end{proof}

\begin{proposition}
\label{Hodge equivalence implies dRW equivalence}
Let $f \colon X \to Y$ be a map of smooth geometric Artin stacks.
If $f$ is a Hodge $d$-equivalence, then it is a de Rham--Witt $(d-1)$-equivalence. 
Moreover, if both \(X\) and \(Y\) are quasi-compact quasi-separated smooth Hodge-proper geometric Artin stacks, then \(f\) is a de Rham--Witt \(d\)-equivalence. 
\end{proposition}

\begin{proof}
Let $M \coloneqq \mathrm{Cone}\big(W\Omega^{\bullet}(Y/k)
\to W\Omega^{\bullet}(X/k) \big)$, which is bounded below by $0$.
Moreover, by \Cref{dRW of a smooth stack is always complete}, we know that
$M$ is complete. Hence, for the first statement,
it suffices to show that $\mathrm{H}^i(\RR_n \otimes^{\rmL}_\RR M)^j = 0$
for all $n$ and all $i + j < d-1$.
This follows from the combination of \Cref{Hodge and R1} and \Cref{Ek2I.1.1}
(with $A = \{0\}$).

If, in addition, both \(X\) and \(Y\) are quasi-compact, quasi-separated, smooth Hodge-proper geometric Artin stacks, then \(M\) is automatically grading-left bounded by \Cref{dRW of a smooth stack is always complete}. By \Cref{Hodge-proper dRW is coherent} together with \Cref{coherence and Hodge-d-equivalence gives de Rham--Witt d-equivalence}, it follows that \(f\) is a de Rham--Witt \(d\)-equivalence.
\end{proof}

\begin{example}
\label{dRW of BGm}
The above allows us to compute the de Rham--Witt cohomology of $B\mathbb{G}_m$.
Consider the following diagram of smooth stacks:
\[
\mathbb{P}^1_k \to \mathbb{P}^2_k \to \cdots \to B\mathbb{G}_m,
\]
where the inclusion $\mathbb{P}^{n-1}_k \to \mathbb{P}^n_k$ is given by
the hyperplane at infinity, and the map to $B\mathbb{G}_m$ is classified by $\mathcal{O}(1)$
on projective spaces. 
Let us tentatively write $\mathbb{P}^{\infty}_k \coloneqq B\mathbb{G}_m$.
Then the map $\mathbb{P}^i_k \to \mathbb{P}^j_k$ is a Hodge $(2i+1)$-equivalence
(hence a de Rham--Witt $(2i+1)$-equivalence thanks to \Cref{Hodge equivalence implies dRW equivalence}) for all 
$j \in \mathbb{N}_{\geqslant i} \cup \{\infty\}$. 
Consequently, the map
\(
W\Omega^{\bullet}(B \mathbb{G}_m) \to \lim_n \mathrm{R\Gamma}(\mathbb{P}^n_k, W\Omega^{\bullet})
\)
is an equivalence, and the limit is eventually constant in each cohomological degree and grading.
Therefore, we obtain the following computation:
\[
W\Omega^{\bullet}(B \mathbb{G}_m) \cong \bigoplus_{n \geq 0} W(k)(-n)[-n],
\]
with $F$ and $V$ on $W(k)$ given by the usual Witt vector Frobenius and Verschiebung.
\end{example}

The following result follows from the proof of \cite[Theorem 1.2 \& Proposition 5.11]{ABM}:

\begin{theorem}[Antieau--Bhatt--Mathew]
\label{Enough Hodge equivalence approximations}
Suppose that $k$ is a perfect field of characteristic $p > 0$ and that $G$
is a finite $k$-group scheme. 
For any integer $d > 0$, there exists a smooth projective
$k$-scheme $X$ of dimension $d$ together with a map $X \to BG \times B\mathbb{G}_m$
which is a Hodge $d$-equivalence.
\end{theorem}

\begin{proof}
We note that the stack $[\mathbb{P}(V)/G]$ in the \cite[first paragraph of the proof of Theorem 1.2]{ABM}
admits a map to $B\mathbb{G}_m$ (see the proof of \cite[Proposition 5.11]{ABM})
as well as a map to $BG$. According to the proof of \cite[Proposition 5.11]{ABM},
the induced map $[\mathbb{P}(V)/G] \to BG \times B\mathbb{G}_m$ is an $N$-equivalence,
where $N$ is twice the dimension of the $G$-representation $V$.
According to the construction of the complete intersection $Z$ in the proof of
\cite[Theorem 1.2]{ABM}, we know that $N$ is much larger than $d$.
\end{proof}

\begin{corollary}
\label{Enough dRW equivalence approximations}
Suppose that $k$ is a perfect field of characteristic $p > 0$ and that $G$
is a finite $k$-group scheme. 
For any integer $d > 0$, there exists a smooth projective
$k$-scheme $X$ of dimension $d$ together with a map $X \to BG \times B\mathbb{G}_m$
which is a de Rham--Witt $d$-equivalence.
\end{corollary}

\begin{proof}
This follows from the combination of \Cref{Hodge equivalence implies dRW equivalence}
and \Cref{Enough Hodge equivalence approximations}.
\end{proof}

The remainder of this section shall be devoted to establishing
a K\"{u}nneth formula for the de Rham--Witt cohomology of smooth geometric Artin stacks. 

\begin{theorem}
\label{KF for stacks}
Let \(X,Y\) be smooth geometric Artin stacks over \(k\). Assume that one of the following holds:
\begin{enumerate}
\item Both $X$ and $Y$ are quasi-compact and quasi-separated; 
\item $X$ or $Y$ is quasi-compact, quasi-separated, and Hodge-proper.
\end{enumerate}
Then the following natural map is an isomorphism:
\[
W\Omega^\bullet(X/k) \cdstimes W\Omega^\bullet(Y/k)
\xrightarrow[\cong]{p_1^* \cup p_2^*} W\Omega^\bullet(X\times Y/k). 
\]
\end{theorem}

Here we are using the $\mathbb{E}_1$-structure on $W\Omega^\bullet(X\times Y/k)$
(see \Cref{E1 structure}) to define the natural map.

\begin{proof}
We claim that the left-hand side is grading-left bounded by $0$.
Note that $W\Omega^{\bullet}(-/k)$ is always grading-left bounded by $0$, and the completion
of an object that is grading-left bounded by $0$ is also grading-left bounded by $0$.
Therefore, it suffices to know that $- \stimes^{\rmL} -$ preserves the full subcategory of $\DGr(\RR)$
spanned by objects that are grading-left bounded by $0$.

Since this property is preserved under colimits, we may further reduce to considering
the full subcategory spanned by objects that are cohomologically bounded above and grading-left bounded by $0$.
Now the claim follows from the fact that such objects can be represented by complexes whose
terms are direct sums of $\RR(i)$ with $i \leqslant 0$.

Therefore our statement concerns a map between two complete objects that are grading-left bounded by $0$.
By \Cref{Ek2I.1.1}, in order to show that the map is an isomorphism, it suffices to check that
it becomes an isomorphism after applying $\RR_1 \otimes^{\rmL}_\RR -$.
Using \Cref{Hodge and R1} and \Cref{E1 structure and tensor R1},
we are reduced to showing that the similarly defined map on Hodge cohomology is an isomorphism,
which is the content of the lemma below.
\end{proof}

\begin{lemma}
In the setting of \Cref{KF for stacks}, the following natural map is an isomorphism:
\[
\wedge^{\bullet}\mathbb{L}_{-/k}(X) \otimes^{\rmL}_k \wedge^{\bullet}\mathbb{L}_{-/k}(Y)
\xrightarrow[\cong]{p_1^* \cup p_2^*} \wedge^{\bullet}\mathbb{L}_{-/k}(X \times Y). 
\]
\end{lemma}

\begin{proof}
We note that working over a field $k$ implies that for any bounded-below object 
$M \in \DGr(k)$, the functor $M \otimes^{\rmL}_k -$ commutes with the totalization
of uniformly bounded-below objects.

Let us first prove case (1), so suppose that both $X$ and $Y$ are quasi-compact and quasi-separated.
We proceed by induction on the geometricity of $Y$. Suppose $Y$ is $n$-geometric. Then our assumption
implies that we can choose an $(n-1)$-representable smooth surjection $\pi \colon U \twoheadrightarrow Y$
from a smooth affine scheme $U$, and all the terms
$U^{({\ast})} \coloneqq U^{\times_Y (\ast + 1)}$ in the \v{C}ech nerve of $\pi$
are quasi-compact, quasi-separated smooth $(n-1)$-geometric Artin stacks.
By the fact recalled in the first paragraph, we are reduced to showing the claim in the case
where $Y$ is replaced by $U^{({\ast})}$. Therefore we are reduced to the case where $Y$
is affine. Repeating the same argument, we may further reduce to the case where $X$
is also affine, which is well known.

Next we reduce case (2) to case (1). Assume that $X$ is quasi-compact, quasi-separated,
and Hodge-proper. By mimicking the argument in the previous paragraph, we may reduce to the case
where $Y$ is a smooth scheme. Write
\[
Y = \bigcup_{\lambda \in \Lambda} U_{\lambda}
\]
as an increasing union of its quasi-compact open subschemes $U_{\lambda}$.
The Hodge cohomology of $Y$ (resp.~$X \times Y$) is the cofiltered limit of the Hodge cohomologies
of $U_{\lambda}$ (resp.~$X \times U_{\lambda}$).
The assumption that $X$ is Hodge-proper implies that tensoring with its Hodge cohomology
commutes with cofiltered limits. Therefore we may replace $Y$ by one of the $U_{\lambda}$,
which is quasi-compact and quasi-separated. Now we are in case (1), and the proof is complete.
\end{proof}

\begin{corollary}
\label{effect of multiplying BGm}
Let \(X\) be a smooth geometric Artin stack over \(k\). Then we have a natural equivalence:
\[
\bigoplus_{n \geqslant 0} W\Omega^{\bullet}(X)(-n)[-n] \xrightarrow{\cong} W\Omega^{\bullet}(X \times B\mathbb{G}_m).
\]
\end{corollary}

\begin{proof}
Combining \Cref{dRW of BGm} and \Cref{KF for stacks}, the statement follows from
the fact that $W(k)$ (with its Witt vector Frobenius and Verschiebung) is the unit for $- \stimes^{\rmL} -$.
\end{proof}

\section{de Rham--Witt cohomology of $B\alpha_p$}
\label{dRW of Balphap}

In this section, we explain the part of the de Rham--Witt cohomology of $B\alpha_p$ that
we are able to compute.

Let \(G\) be a finite flat commutative group scheme over \(k\). Let us study the 
de Rham--Witt cohomology of the classifying stack \(BG\).
Recall that any finite flat commutative group scheme over \(k\) is a closed subgroup scheme of an abelian variety:
see \cite[II.15.4, II.15.11]{Oort} or \cite[Theorem 3.1.1]{BBM82} (where a more general result
is attributed to Raynaud). 
Fix such an embedding \(G \hookrightarrow A\) into an abelian variety \(A\), 
and set $B \coloneqq A/G$, which is another abelian variety.
Since the map $A \xrightarrow{g} B$ is a $G$-torsor and $A$ is smooth, we obtain a smooth surjection
$B \twoheadrightarrow BG$.
According to \Cref{HWArtinstack}, the de Rham--Witt cohomology of $BG$ is given by the totalization
of the de Rham--Witt cohomology of the simplicial scheme $B^{\times_{BG} (* + 1)}$
obtained by taking the \v{C}ech nerve of $B \twoheadrightarrow BG$.

\begin{proposition}
\label{Rewrite the simplicial scheme}
There are natural isomorphisms of schemes
\[
B^{\times_{BG} (* + 1)} \cong A^{\times_k *} \times_k B.
\]
Moreover, under these isomorphisms,
the face maps $\{d_i\}_{0 \leqslant i \leqslant n} \colon B^{\times_{BG} (n + 1)} \to B^{\times_{BG} n}$
are identified with the maps
$f_i: A^{n}\times B\rightarrow A^{n-1}\times B$ given by
\[
f_i(a_0,\ldots,a_{n-1},b)=
\begin{cases}
(a_1,a_2,\ldots,a_{n-1},b), & \text{if } i=0,\\
(a_0,\ldots,a_{i-2},a_{i-1}+a_i,a_{i+1},\ldots,a_{n-1},b), & \text{if } 1 \leqslant i \leqslant n-1,\\
(a_0,\ldots,a_{n-2},g(a_{n-1})+b), & \text{if } i=n;
\end{cases}
\]
and the degeneracy maps $\{s_i\}_{0 \leqslant i \leqslant n} \colon B^{\times_{BG} (n + 1)} \to B^{\times_{BG} (n+2)}$
are identified with the maps
$0_i \colon A^{n}\times B\rightarrow A^{n+1}\times B$ given by
\[
0_i(a_0, \ldots,a_{n-1},b) = (a_0, \ldots,a_{i-1}, 0, a_i, \ldots, a_{n-1}, b).
\]
\end{proposition}

\begin{proof}
First, we write
\[
B^{\times_{BG} (* + 1)} \cong (A/G)^{\times_{[\Spec(k)/G]} (* + 1)} \cong A^{\times_k (* + 1)}/G,
\]
where $A^{\times_k (* + 1)}$ is equipped with the diagonal $G$-action.
We then use the isomorphisms
\[
\begin{aligned}
A^{\times_k (* + 1)}/G &\cong A^{\times_k *}\times B,\\
(a_0,a_1,\ldots,a_n)&\mapsto (a_0-a_1,\ldots,a_{n-1}-a_n,g(a_n)).
\end{aligned}
\]
A direct computation shows that the maps $d_i$ (resp.~$s_i$) correspond to $f_i$
(resp.~$0_i$) under these isomorphisms.
\end{proof}

\begin{example}
\label{description of face maps}
For example, we have the following face maps from \(A\times B\) to \(B\):
\[
\begin{aligned}   
&f_0(a_0,b)=(b),\\
&f_1(a_0,b)=(g(a_0)+b).
\end{aligned}
\]
In the next degree, we have the following face maps
from \(A\times A\times B\) to \(A\times B\):
\[
\begin{aligned}   
&f_0(a_0,a_1,b)=(a_1,b),\\
&f_1(a_0,a_1,b)=(a_0+a_1,b),\\
&f_2(a_0,a_1,b)=(a_0,g(a_1)+b).
\end{aligned}
\]
\end{example}

Using the explicit simplicial scheme described above, we obtain
\(
W\Omega^{\bullet}(BG) \cong \mathrm{Rlim}_{[*] \in \Delta}
W\Omega^{\bullet}(A^{\times_k *} \times_k B).
\)

\begin{construction}
\label{diagonal filtration and spectral sequence}
Let us filter each $W\Omega^{\bullet}(A^{\times_k *} \times_k B)$ using the diagonal $t$-structure
introduced in \Cref{subsection: Diagonal}.
By \Cref{diagonal t-bound for smooth varieties},
each $W\Omega^{\bullet}(A^{\times_k *} \times_k B)$ lies in 
$\widetilde{\DGr}^{[0, 2(* + 1) \dim(A)]}_c(\RR)$.
Therefore, we obtain an increasing exhaustive filtration (indexed by $\mathbb{N}$)
\[
\mathrm{Fil}_n \coloneqq \mathrm{Rlim}_{[*] \in \Delta}
\dt^{\leqslant n} W\Omega^{\bullet}(A^{\times_k *} \times_k B)
\]
on $W\Omega^{\bullet}(BG)$, whose associated graded pieces are
\[
\mathrm{gr}_n \coloneqq \mathrm{Rlim}_{[*] \in \Delta} 
\widetilde{\HH}^n(W\Omega^{\bullet}(A^{\times_k *} \times_k B))[-n].
\]
As a result, we obtain a spectral sequence of objects in \(\dc\):
\begin{equation}
\label{TotSS}
E_1^{i,j}=\widetilde{\HH}^j(W\Omega^\bullet(A^{i}\times B))
\Longrightarrow
\widetilde{\HH}^{i+j}(W\Omega^\bullet(BG)).
\end{equation}
\end{construction}

\begin{notation}
Throughout the remainder of this section, we abbreviate \(\widetilde{\HH}^i(W\Omega^\bullet(-))\)
by \(\widetilde{\HH}^i(-)\).
\end{notation}

\begin{proposition} 
\label{two rows of E2 for BG}
On the \(E_2\)-page of the spectral sequence \((\ref{TotSS})\), the only nonzero terms
in the zeroth and first rows are
\[
E_2^{0,0}=W(k), \qquad 
E_2^{1,1}= \D(G_{\mathrm{uni}})\oplus \D(G_{\mathrm{mul}})'(-1)[1].
\]
In other words, the \(E_2\)-page of the spectral sequence \((\ref{TotSS})\) has the form
\[
\begin{tikzcd}
	{E_2^{0,2}} & {E_2^{1,2}} & {E_2^{2,2}} & {E_2^{3,2}} & {\cdots} \\
	0 & {E_2^{1,1}} & 0 & 0 & {\cdots} \\
	W & 0 & 0 & 0 & {\cdots}
	\arrow[from=1-1, to=2-3]
	\arrow[from=1-2, to=2-4]
	\arrow[from=2-1, to=3-3]
	\arrow[from=2-2, to=3-4]
\end{tikzcd}
\]
\end{proposition}

The notation $\D(G_{\mathrm{mul}})'$ is introduced in \Cref{' notation}.

\begin{proof}
The zeroth row of the spectral sequence reads
\(
W \xrightarrow{0} W \xrightarrow{\mathrm{id}} W \xrightarrow{0} W \longrightarrow \cdots .
\)
For the first row, by \Cref{Kunneth for Mazur--Ogus} we have
\(
\widetilde{\HH}^1(A^n\times B)
\cong
\widetilde{\HH}^1(A)^{\oplus n}\oplus \widetilde{\HH}^1(B).
\)
These cohomology groups are additive in the sense that \(f^*+g^*=(f+g)^*\). 
In degree \(n\), a direct computation shows that the map
\(
\sum_{i=0}^n (-1)^i f_i: A^n\times B\rightarrow A^{n-1}\times B
\)
can be written as
\[
\Bigl(\sum_{i=0}^n (-1)^if_i\Bigr)(a_0,\ldots,a_{n-1},b)=
\begin{cases}
(0,a_1+a_2,0,a_3+a_4,\ldots,0,g(a_{n-1})+b), 
& \text{if } n \text{ is even},\\[4pt]
(-a_0,a_2,-a_2,\ldots,a_{n-1},-g(a_{n-1})), 
& \text{if } n \text{ is odd}.
\end{cases}
\]

Consequently, we may express
\[
\sum_{i=0}^n (-1)^if_i^*:
\widetilde{\HH}^1(A)^{n-1}\oplus \widetilde{\HH}^1(B)
\longrightarrow
\widetilde{\HH}^1(A)^{n}\oplus \widetilde{\HH}^1(B)
\]
as
\[
\Bigl(\sum_{i=0}^n (-1)^if_i^*\Bigr)(x_0,\ldots,x_{n-2},y)=
\begin{cases}
(0,x_1,x_1,x_3,x_3,\ldots,x_{n-3},x_{n-3},g^*(y),y), 
& \text{if } n \text{ is even},\\[4pt]
(-x_0,0,x_1-x_2,\ldots,x_{n-2}-g^*(y),0), 
& \text{if } n \text{ is odd}.
\end{cases}
\]

The proposition now follows by direct inspection, together with
\Cref{taking ' is exact} and \Cref{widetildeH^1(A)}.
\end{proof}

\begin{situation}
\label{alphap setup}
From now on, we specialize to the case where $G = \alpha_p$ and $k$ is algebraically closed.
Choose a supersingular elliptic curve over $k$, and write \(g \colon E\to E^{(1)}\) for the relative Frobenius.
Then \(g\) exhibits \(E \to E^{(1)}\) as an \(\alpha_p\)-torsor, and we may apply the discussion
carried out so far in this section.
\end{situation}

\begin{lemma}
\label{lemma: vanishing on 0-th column of TotSS}
Specializing \Cref{diagonal filtration and spectral sequence} to \Cref{alphap setup},
we have \(E_1^{0, \geqslant 3} = 0\).
\end{lemma}

\begin{proof}
This follows from the fact that \(\dim(E^{(1)}) = 1\).
\end{proof}

Our goal in this article is to understand $\dt^{\leqslant 3} W\Omega^{\bullet}(B\alpha_p)$.
Thus, the only remaining task is to analyze the first few terms of the second row
of the spectral sequence $(\ref{TotSS})$.

\begin{proposition}
\label{Compute second row of TotSS}
Specializing \Cref{diagonal filtration and spectral sequence} to \Cref{alphap setup},
we have $E_2^{0, 2} = 0$, and there is a short exact sequence in $\dc$:
\[
0 \to U_{-1} \to E_2^{1, 2} \to k(-1)[1] \to 0.
\]
\end{proposition}

Here $k$ denotes the left graded $\RR$-module with its usual (bijective) Frobenius and $V = 0$.

\begin{proof}
Consider the following cochain complex whose terms lie in $\dc$:
\(
\widetilde{\HH}^2(E^{(1)}) 
\xrightarrow{d_1^{0,2}} 
\widetilde{\HH}^2(E \times_k E^{(1)})
\xrightarrow{d_1^{1,2}} 
\widetilde{\HH}^2(E \times E \times_k E^{(1)}).
\)
Our task is to show that $d_1^{0,2}$ has no kernel and to understand the 
cohomology in the middle. 
By \Cref{Kunneth for Mazur--Ogus}, the terms can be identified with
\[
\widetilde{\HH}^2(E^{(1)}) = W(-1)[1] 
\xrightarrow{d_1^{0,2}} 
\widetilde{\HH}^2(E \times_k E^{(1)})
=
\Big(\widetilde{\HH}^1(E) \cdstimes \widetilde{\HH}^1(E^{(1)})\Big)
\oplus \widetilde{\HH}^2(E) \oplus \widetilde{\HH}^2(E^{(1)})
\xrightarrow{d_1^{1,2}} 
\]
\[
\xrightarrow{d_1^{1,2}} 
\widetilde{\HH}^2(E \times E \times_k E^{(1)}) =
\Big(\widetilde{\HH}^1(E)\cdstimes\widetilde{\HH}^1(E^{(1)})\Big)^{\oplus 2} 
\oplus 
\Big(\widetilde{\HH}^1(E)\cdstimes\widetilde{\HH}^1(E)\Big)
\oplus
\widetilde{\HH}^2(E)^{\oplus 2}
\oplus 
\widetilde{\HH}^2(E^{(1)}).
\]

Using the first part of \Cref{description of face maps}, we see that
$d_1^{0,2}=f_0^*-f_1^*$ and \(f_0^*=(0,0,\id)\).
To describe $f_1^*$, let us denote the map
\(
m^*-\pi_1^*-\pi_2^* \colon \widetilde{\HH}^2(E^{(1)})\to
\widetilde{\HH}^1(E)\cdstimes\widetilde{\HH}^1(E^{(1)})
\)
by \(h_{E^{(1)}}\).
Then we may write
\(
f_1^* = ((g^*\cdstimes \id)\circ h_{E^{(1)}},\, g^*,\, \id).
\)
As a result,
\(
d_1^{0,2}=f_0^*-f_1^*
= (-(g^*\cdstimes \id)\circ h_{E^{(1)}}, -g^*,0).
\)
Here in the second component of this map,
\(
g^* \colon \widetilde{\HH}^2(E^{(1)})
\to \widetilde{\HH}^2(E),
\)
can be identified with
\(
W(k)(-1)[1] \xrightarrow{\cdot p} W(k)(-1)[1].
\)
In particular, this map is injective in $\dc$. Therefore \(E_2^{0,2}=0\).

Using the second part of \Cref{description of face maps}, we may compute the
three maps
\[
f_0^*,f_1^*,f_2^*\colon \widetilde{\HH}^2(E \times E^{(1)})\longrightarrow
\widetilde{\HH}^2(E \times E \times E^{(1)}).
\]
In terms of the K\"{u}nneth components, first we have
\(
f_0^* = (0, \pi_1, 0, 0, \pi_2, \pi_3).
\) 
To describe $f_1^*$, let us denote the map
\(
m^*-\pi_1^*-\pi_2^* \colon \widetilde{\HH}^2(E)\to
\widetilde{\HH}^1(E)\cdstimes\widetilde{\HH}^1(E)
\)
by \(h_E\). Then we have
\(
f_1^* = (\pi_1, \pi_1, h_E \circ \pi_2, \pi_2, \pi_2, \pi_3).
\) 
Lastly, let $h_{E^{(1)}}$ be defined similarly to $h_E$. Then we have
\[
f_2^* = (\pi_1, (g^*\cdstimes \id) \circ h_{E^{(1)}} \circ \pi_3,
(\id \cdstimes g^*) \circ \pi_1, \pi_2, g^* \circ \pi_3, \pi_3).
\]
Taking the alternating sum, we obtain
\(
d_1^{1,2} =
(0, (g^*\cdstimes \id) \circ h_{E^{(1)}} \circ \pi_3, 
(\id \cdstimes g^*) \circ \pi_1 - h_E \circ \pi_2, 
0, g^* \circ \pi_3, \pi_3).
\)

Let us summarize our knowledge so far: We see that $\Ker(d_1^{1,2}) \subset \Ker(\pi_3)$,
and the latter also contains the image of $d_1^{0,2}$.
Therefore, $E_2^{1,2}$ can be described as the middle cohomology (in $\dc$)
of the following sequence:
\[
\widetilde{\HH}^2(E^{(1)}) \xrightarrow{(-(g^*\cdstimes \id)\circ h_{E^{(1)}}, -g^*)} \Big(\widetilde{\HH}^1(E) 
\cdstimes \widetilde{\HH}^1(E^{(1)})\Big)
\oplus \widetilde{\HH}^2(E) \xrightarrow{(\id \cdstimes g^*) \circ \pi_1 - h_E \circ \pi_2}
\Big(\widetilde{\HH}^1(E) 
\cdstimes \widetilde{\HH}^1(E)\Big).
\]
Consider the subcomplex
\(
0 \to \widetilde{\HH}^1(E) 
\cdstimes \widetilde{\HH}^1(E^{(1)}) \xrightarrow{\id \cdstimes g^*} \widetilde{\HH}^1(E) 
\cdstimes \widetilde{\HH}^1(E),
\)
and the associated quotient complex
\(
\widetilde{\HH}^2(E^{(1)}) = W(-1)[1] \xrightarrow{\cdot (-p)} W(-1)[1] = \widetilde{\HH}^2(E) \to 0.
\)
Using \Cref{widetildeH^1(A)},
the subcomplex is identified with $\widetilde{\HH}^1(E) \cdstimes \mathbb{D}(\alpha_p)[-2]$.
Since $k$ is algebraically closed, applying \Cref{widetildeH^1(A)} again, we obtain
an isomorphism $\widetilde{\HH}^1(E) \simeq E_{1/2}$.
By \Cref{key computation} and \Cref{H and tildeH}, we obtain the cohomology of the subcomplex: 
$\widetilde{\HH}^1 \simeq U_{-1}$ and $\widetilde{\HH}^2 \simeq U_1$. 
Therefore, the associated long exact sequence in $\dc$ has the form
\[
0 \to U_{-1} \to E_2^{1,2} \to k(-1)[1] \to U_1.
\]
By considering the standard $t$-structure, we see that the last arrow must be $0$,
which completes the proof of the description of $E_2^{1,2}$.
\end{proof}

We see that $E_2^{1,2}$ is the cone of a map $k(-1) \to U_{-1}$. It is simple to
classify all such maps. 

\begin{lemma}
\label{classify k(-1) to U-1}
There is an isomorphism $\mathrm{Hom}_{\DGr(\RR)}(k(-1), U_{-1}) \cong k$
as abelian groups. For any map corresponding to $\lambda \in k^{\times}$,
the cone is isomorphic to $U_0$.
\end{lemma}

\begin{proof}
The grading $1$ piece of $U_{-1}$ is identified with 
$k \cdot Fd \oplus \prod_{n \geqslant 0} k \cdot dV^n$.
Any map is determined by the image of $1 \in k$ in grading $1$; write this image as
\(
\lambda_{-1} Fd + \sum_{n \geqslant 0} \lambda_n dV^n .
\)
For the map to be an $\RR$-module map, the above sum must be stable under $F$,
which amounts to the relations $\lambda_m = \lambda_{m+1}^p$. Therefore the maps are in bijection
with $k$, sending such a map to $\lambda_{-1}$. 
If $\lambda_{-1} \neq 0$, we see that the cone is isomorphic to the following graded left $\RR$-module:
\(
k[\![V]\!] \xrightarrow{d} \prod_{n \geqslant 0} k \cdot dV^n,
\)
which is $U_0$.
\end{proof}

\begin{theorem}
\label{nonsplit SS}
The short exact sequence from \Cref{Compute second row of TotSS} is non-split.
Consequently, there is an isomorphism $E_2^{1,2} \simeq U_0$.
Moreover, for all pairs of natural numbers $(i, j)$ with $i + j \leqslant 3$, we have
\[
\HH^j(W\Omega^{\bullet}(B\alpha_p))^{[i,i]} =
\begin{cases}
W & \text{if } (i,j) = (0,0) \\
\mathbb{D}(\alpha_p) & \text{if } (i,j) = (0, 2) \\
U_0 & \text{if } (i, j) = (0, 3) \\
0 & \text{otherwise.}
\end{cases}
\]
\end{theorem}

\begin{proof}
Let us begin by analyzing $\HH^j(W\Omega^{\bullet}(B\alpha_p))^{[i,i]}$.
Combining \Cref{two rows of E2 for BG}, \Cref{lemma: vanishing on 0-th column of TotSS},
and \Cref{Compute second row of TotSS}, we see that
$\widetilde{\HH}^n(W\Omega^{\bullet}(B\alpha_p))$ is:
\begin{itemize}
\item $W$ when $n = 0$;
\item $0$ when $n = 1$;
\item $\mathbb{D}(\alpha_p)$ when $n = 2$; and
\item $E_2^{1,2}$ when $n = 3$.
\end{itemize}
Using \Cref{H and tildeH}, the last statement now follows from the second statement and the above
description. The second statement follows from the first statement together with
\Cref{classify k(-1) to U-1}.

Our only remaining task is to show that the exact sequence from
\Cref{Compute second row of TotSS} does not split.
Suppose otherwise. Then we would have
$E_2^{1,2} = U_{-1} \oplus k(-1)[1]$.
Using \Cref{H and tildeH}, we obtain
$\HH^2(W\Omega^{\bullet}(B\alpha_p))^{[1,1]} = k(-1)$.
In particular, we arrive at a decomposition
\[
\HH^2(W\Omega^{\bullet}(B\alpha_p))
= \mathbb{D}(\alpha_p) \oplus k(-1) \oplus M
\]
as graded left $\RR$-modules, where $M^{\leqslant 1} = 0$.

We claim that this contradicts the known description of the Hodge cohomology of $B\alpha_p$
obtained in \cite{ABM}, as follows.
By \Cref{Hodge and R1}, we have an equivalence
\(
\RR_1 \otimes^{\rmL}_\RR W\Omega^{\bullet}(B\alpha_p)
\simeq \wedge^{\bullet}\mathbb{L}_{-/k}(B\alpha_p)
\)
in $\DGr(k[d]/d^2)$.
Since $\HH^1(W\Omega^{\bullet}(B\alpha_p))^{\leqslant 2} = 0$, we see that
$\mathrm{Tor}^\RR_0(\RR_1, \HH^2(W\Omega^{\bullet}(B\alpha_p)))$
has no component in grading $\leqslant 2$.
Using the resolution in \Cref{Rn finite flat dim}, we see that, within the grading
$\leqslant 2$ range, there is an injection
\[
\mathrm{Tor}^\RR_1(\RR_1, \HH^2(W\Omega^{\bullet}(B\alpha_p)))^{\leqslant 2}
\hookrightarrow
\HH^1(B\alpha_p, \wedge^{\bullet}\mathbb{L}_{-/k})^{\leqslant 2}.
\]
Using the decomposition
$\HH^2(W\Omega^{\bullet}(B\alpha_p))
= \mathbb{D}(\alpha_p) \oplus k(-1) \oplus M$,
together with the calculation in \cite[Corollaire I.3.6]{IRdRW}, we see that
there is an injection of graded $k[d]/d^2$-modules
\[
(k \hookrightarrow k^{\oplus 2} \to 0)
\hookrightarrow
\HH^1(B\alpha_p, \wedge^{\bullet}\mathbb{L}_{-/k})^{\leqslant 2}.
\]
In particular, we deduce that the $(1,1)$-Hodge cohomology of $B\alpha_p$
contains a $2$-dimensional subspace annihilated by the Hodge--to--de~Rham differential $d_1$.
This contradicts \cite[Propositions 4.10 and 4.12]{ABM}.
The authors show (\cite[Proposition 4.10]{ABM}) that this cohomology group is exactly
$2$-dimensional with generators $u$ and $\alpha \cdot s$
(using the notation of loc.~cit.).
Moreover, they show (\cite[Proposition 4.12]{ABM}) that
$d_1(\alpha) = u$ up to a unit and $d_1(s) = 0$.
Therefore we have
$d_1(u) = d_1^2(\alpha) = 0$
but
$d_1(\alpha \cdot s) = u \cdot s \neq 0$,
where the last nonvanishing again uses \cite[Proposition 4.10]{ABM}.
Hence we obtain a contradiction, proving the first statement.
\end{proof}

\begin{corollary}
\label{the counterexample}
Let $k$ be a field of characteristic $p > 0$.
There exists a smooth projective fourfold $X$ over $k$ such that, for all pairs of natural numbers $(i,j)$ with $i + j \leqslant 3$, the base change $X_{\overline{k}}$ has Hodge--Witt cohomology given by
\[
\HH^j(X_{\overline{k}}, W\Omega^{\bullet}_{X_{\overline{k}}/\overline{k}})^{[i,i]} =
\begin{cases}
W & \text{if } (i,j) = (0,0), \\
W(-1) & \text{if } (i,j) = (1,1), \\
\mathbb{D}(\alpha_p) & \text{if } (i,j) = (0,2), \\
U_0 & \text{if } (i,j) = (0,3), \\
0 & \text{otherwise.}
\end{cases}
\]

Moreover, its first and third crystalline cohomology groups vanish.
In particular, the Hodge--Witt numbers of $X$ in total degree $\leqslant 3$ are given below, and they are asymmetric in degree $3$:
\begin{center}
    
\begin{tikzpicture}
    \draw [thick,->] (0, 0) -- (2.5, 0);
    \draw [thick,->] (0, 0) -- (0, 2.5);
    \node at (0.25,0.25){\(1\)}; 
    \node at (0.25,0.85){\(0\)};
    \node at (0.85,0.25){\(0\)};
    \node at (0.85,0.85){\(1\)};
    \node at (1.45,0.25){\(0\)};
    \node at (0.25,1.45){\(0\)};
    \node at (2.05,0.25){\(0\)};
    \node at (1.45,0.85){\(1\)};
    \node at (0.85,1.45){\(-2\)};
    \node at (0.25,2.05){\(1\)};

    \node at (-1.5,0.25){\(h_W^{i,j}(X):\)};
    \node at (0.25,-0.2){\textcolor{teal}{0}};
    \node at (0.85,-0.2){\textcolor{teal}{1}};
    \node at (1.45,-0.2){\textcolor{teal}{2}};
    \node at (2.05,-0.2){\textcolor{teal}{3}};
    \node at (3,-0.2){\textcolor{teal}{\(i\)}}; 

    \node at (-0.2,0.25){\textcolor{teal}{0}};
    \node at (-0.2,0.85){\textcolor{teal}{1}};
    \node at (-0.2,1.45){\textcolor{teal}{2}};
    \node at (-0.2,2.05){\textcolor{teal}{3}};
    \node at (-0.2,3){\textcolor{teal}{\(j\)}};     
\end{tikzpicture}
\end{center}
\end{corollary}

\begin{proof}
By \Cref{Enough dRW equivalence approximations} and \Cref{Hodge equivalence implies crystalline equivalence},
we may find a smooth projective fourfold $X$ together with a map
$X \to B\alpha_p \times B\mathbb{G}_m$
which is both a de Rham--Witt $4$-equivalence and a crystalline $4$-equivalence
(and remains so after base change to $\overline{k}$).
Using \Cref{effect of multiplying BGm} and \Cref{nonsplit SS},
we obtain the claimed description of its Hodge--Witt cohomology.
The statement about its crystalline cohomology follows from
\cite[Proposition 4.17]{ABM}.
\footnote{The analogue of \Cref{effect of multiplying BGm} also holds for crystalline cohomology,
by a very similar argument.}
Finally, for the computation of the Hodge--Witt numbers, we use the fact that
base change of the ground perfect field does not change these numbers.
\end{proof}

\bibliographystyle{amsalpha}
\bibliography{main}

\end{document}